\documentclass[a4paper,reqno,final]{amsart}

\usepackage[foot]{amsaddr}
\usepackage{graphicx,color,enumerate,nicefrac,bm}
\usepackage{cite}
\usepackage{stmaryrd,amsmath, amssymb, amstext, amsfonts, mathrsfs}
\usepackage{algorithmicx,algpseudocode}
\usepackage{epstopdf}
\usepackage{a4wide}

\newcommand{\NNN}[1]{\left|\!\left|\!\left|#1\right|\!\right|\!\right|}
\newcommand{\BB}{\mathcal{B}}
\newcommand{\T}{\mathcal{T}}

\newcommand{\I}{\mathcal{I}}
\newcommand{\E}{\mathcal{E}}
\newcommand{\eps}{\epsilon}
\newcommand{\F}{\mathsf{F}_{\eps}}
\newcommand{\jl}{[\![}
\newcommand{\jr}{]\!]}
\newcommand{\jmp}[1]{\jl#1\jr}
\newcommand{\al}{\langle\!\langle}
\newcommand{\ar}{\rangle\!\rangle}
\newcommand{\avg}[1]{\al#1\ar}
\newcommand{\dd}{\mathsf{d}}
\newcommand{\dx}{\,\dd\bm x}

\newcommand{\NN}[1]{\left\|#1\right\|}
\newcommand{\DG}{\text{\tiny\sf{DG}}}
\newcommand{\PF}{\text{\tiny\sf{PF}}}
\newcommand{\NF}{\mathsf{NF}}
\newcommand{\uDG}{u^{\DG}}

\newcommand{\utDG}{\widehat{\mathfrak{u}}^{\DG}}
\renewcommand{\a}{a_{\DG}}
\renewcommand{\l}{\ell_{\DG}}
\newcommand{\nT}{\nabla_{\T}}
\newcommand{\ds}{\,\dd s}
\newcommand{\V}{\mathcal{V}_{\DG}}

\newcommand{\W}{\mathcal{W}_{\DG}}
\newcommand{\R}{\mathsf{R}_{\eps}}
\newcommand{\dprod}[1]{\left\langle#1\right\rangle}
\newcommand{\Cp}{C_\eqref{eq:equiv}}
\newcommand{\CP}{C_{\PF}}

\newcommand{\ut}{\widehat{\mathfrak{u}}}
\newcommand{\f}[1]{\widehat{\mathfrak{f}}(#1)}
\newcommand{\LL}{\mathsf{L}}
\newcommand{\A}{\mathsf{A}_{\DG}}

\renewcommand{\P}{\mathsf{P}_{\DG}}
\newcommand{\Ihp}{\mathsf{I}_{hp}}
\newcommand{\id}{\mathsf{id}}

\newtheorem{theorem}{Theorem}[section]
\newtheorem{lemma}[theorem]{Lemma}
\newtheorem{proposition}[theorem]{Proposition}

\theoremstyle{definition}
\newtheorem{example}[theorem]{Example}

\newtheorem{algorithm}[theorem]{Algorithm}
\newtheorem{remark}[theorem]{Remark}

\title[$hp$--Adaptive NDG Methods for Semilinear PDE]{An $hp$--Adaptive Newton-Discontinuous-Galerkin Finite Element Approach for Semilinear Elliptic Boundary Value Problems}

\author[P.~Houston]{Paul Houston}
\address{
School of Mathematical Sciences, University of Nottingham,
University Park, Nottingham, NG7 2RD, UK}
\email{Paul.Houston@nottingham.ac.uk}

\author[T.~P.~Wihler]{Thomas P.~Wihler}
\address{Mathematics Institute, University of Bern, CH-3012 Bern, Switzerland}
\email{wihler@math.unibe.ch}
\thanks{TW acknowledges the support of the Swiss National Science Foundation (SNF), Grant No.~200021-162990}

\begin{document}

\begin{abstract}
In this paper we develop an $hp$--adaptive procedure for the numerical solution of general second-order semilinear elliptic boundary value problems, with possible singular perturbation. Our approach combines both adaptive Newton schemes and an $hp$--version adaptive discontinuous Galerkin finite element discretisation, which, in turn, is based on a robust $hp$--version \emph{a posteriori} residual analysis. Numerical experiments underline the robustness and reliability of the proposed approach for various examples.
\end{abstract}



\maketitle

\section{Introduction}
The subject of this paper is the adaptive numerical approximation of second-order
semilinear elliptic problems of the form
\begin{equation}\label{eq:PDE}
\begin{aligned}
-\eps\Delta u+u &=f(\cdot,u) \text{ in } \ \Omega,\qquad
u=0 \text{ on } \partial \Omega. 
\end{aligned}
\end{equation}
Here, $\Omega\subset\mathbb{R}^2$ is an open and bounded Lipschitz domain, $\epsilon\in(0,1]$ represents a (possibly small singular perturbation) parameter, $f:\,\overline\Omega\times\mathbb{R}\to\mathbb{R}$ is a continuously differentiable function, and~$u:\,\Omega\to\mathbb{R}$ is an unknown solution; in the sequel, we will omit to explicitly express the dependence of~$f$ on the first argument, and simply write~$f(u)$ instead. Problems of this type appear in a wide range of application areas of practical interest, such as, for example, nonlinear reaction-diffusion in ecology and chemical models~\cite{CantrellCosner:03,Edelstein-Keshet:05,BorisyukErmentroutFriedmanTerman:05,Ni:11,OkuboLevin:01}, economy~\cite{BarlesBurdeau:95}, or classical and quantum physics~\cite{BerestyckiLions:83,GibbonJamesMoroz:79,BaroneEspositoMageeScott:71,Strauss:77}. 

Partial differential equations (PDEs) of the form~\eqref{eq:PDE} may admit
a unique solution, no solution at all, or more typically a multitude of solutions, or indeed
infinitely many such solutions. Moreover, in the singularly perturbed case, i.e., when
$0< \epsilon \ll 1$, solutions of \eqref{eq:PDE}, when they exist, may contain 
sharp layers in the form of interior/boundary layers, or isolated spike--like solutions,
and their numerical approximation represents a challenging computational task.
Indeed, to efficiently and reliably compute discrete approximations to 
the analytical solution~$u$ of~\eqref{eq:PDE}, 
it is essential to exploit {\em a posteriori} bounds which
not only provide information regarding the size of the discretisation error, measured in some
appropriate norm, but also yield local error indicators which may subsequently be employed
to enrich the underlying approximation space in an adaptive manner. 
Of course, a key aspect of this general
solution procedure is the design and implementation of a nonlinear solver
which can efficiently compute the approximation $u_h$ to $u$; we shall return
to this issue below.

In general, the traditional approach exploited within the literature for the
design of adaptive finite element methods, for example, is to first discretise
the underlying PDE problem, in our case \eqref{eq:PDE}, 
and to derive an {\em a posteriori} error bound for the resulting (nonlinear) scheme;
this is typically a very mathematically challenging task. However,
once such a bound has been established, then given a suitable initial
mesh and polynomial approximation order, the underlying nonlinear system of discrete 
equations arising from the underlying finite element discretisation may 
be solved based on employing, for example, a (damped) Newton iteration. 
Denoting this computed numerical approximation by $u_h$, the
size of the error between $u$ and $u_h$ may then
be estimated by exploiting this {\em a posteriori}
error bound. If this bound is below a given user tolerance, then sufficient accuracy has 
been attained and the adaptive algorithm may be terminated. Otherwise, the
computational mesh ($h$--refinement) or the polynomial degree ($p$--refinement),
or both ($hp$--refinement) are locally enriched based on identifying regions in
the domain where the elementwise error indicators, which stem from the
{\em a posteriori} error bound, are locally large. On the basis of this new
finite element space, a new approximation $u_h$ to $u$ may be computed, and the whole
process repeated until either the desired accuracy has been attained, or a maximum
number of refinement steps have been completed.

Stimulated by the work undertaken in the recent article \cite{AmreinWihler:15},
we consider an alternative approach based on the 
so-called adaptive Newton-Galerkin paradigm for the numerical
approximation of nonlinear problems of the type~\eqref{eq:PDE}. 
More precisely, this general technique is based on applying local 
Newton-type linearisations on the continuous level that allow for the 
approximation of the semilinear PDE~\eqref{eq:PDE} by a sequence of linearised problems. 
These resulting \emph{linear} PDEs are then discretised by means of an adaptive finite 
element procedure, which, in turn, is based on a suitable \emph{a posteriori} 
residual analysis. The adaptive Newton-Galerkin procedure provides an 
\emph{interplay} between the (adaptive, or damped) Newton method and 
the adaptive finite element approach, whereby we either perform a Newton step 
(if the Newton linearisation effect dominates) or enrich the current finite element space based on the above {\em a posteriori} residual indicators (in the case that the finite element discretisation constitutes the main source of error); for related work we refer to~\cite{ChaillouSuri:07,El-AlaouiErnVohralik:11}, or the articles~\cite{BernardiDakroubMansourSayah:15,CongreveWihler:15,GarauMorinZuppa:11} on (derivative-free) fixed-point iteration schemes. Finally, we point to the works~\cite{ChaillouSuri:06,Han:94} dealing with modelling errors in linearised models.

In the current article, we extend the work undertaken in~\cite{AmreinWihler:15} to the framework of $hp$--version adaptive interior penalty discontinuous Galerkin (DG) schemes, thereby giving rise to $hp$--adaptive Newton-discontinuous Galerkin (NDG) methods. Here, the proof of
the resulting {\em a posteriori} residual bound for the interior penalty DG
discretisation of the underlying linearised PDE problem is based on two key steps:
firstly, we introduce a suitable residual operator on a given
enriched space, which, when measured in an appropriate norm, is equivalent 
to the error measured in terms of the underlying DG energy norm. Secondly, an upper
bound on the norm of the residual operator is derived based on 
exploiting the general techniques developed in the
articles~\cite{HoustonSchotzauWihler:07,HoustonSchotzauWihler:06,Zhu_3D};
we also refer to \cite{SchotzauZhu:11} for the application to
convection--diffusion problems, and to~\cite{HoustonSuliWihler:08,CongreveHoustonSuliWihler:13} for the treatment of strongly monotone
quasilinear PDEs, cf., also,
\cite{CongreveHoustonWihler:13,Congreve_Houston_2014} for $hp$--version 
two-grid DG methods.
The proof of this upper bound crucially relies on the approximation of
discontinuous finite element functions by conforming ones,
cf., also, \cite{KarakashianPascal03} for the $h$--version case.
Moreover, in the current setting, following \cite{Verfurth:98}, particular
care is devoted to the derivation of $\epsilon$-robust approximation estimates.
The resulting {\em a posteriori} bound
consists of two key terms: one stemming from the Newton linearisation error,
and the second which measures the approximation error in the underlying DG
scheme. On the basis of this general $hp$--version bound, we devise 
a fully automatic $hp$--adaptive NDG scheme for the numerical approximation
of PDEs of the form~\eqref{eq:PDE}. Indeed, the performance of the
resulting adaptive strategy is demonstrated on both the Bratu and
Ginzburg Landau problems; moreover, the superiority of exploiting $hp$--enrichment
of the DG finite element space, in comparison with standard mesh adaptation ($h$--refinement), will be highlighted.

The structure of this article is as follows. In Section~\ref{sec:newton_linearisation}
we briefly outline the adaptive (damped) Newton linearisation
procedure employed within this article. The $hp$--version interior penalty
DG discretisation of the resulting linearised PDE problem is then given in
Section~\ref{sc:dgfem}. Section~\ref{sc:apost} is devoted to the derivation
of a residual-based {\em a posteriori} bound. On the basis of this bound in Section~\ref{sec:adaptivity} we design a suitable adaptive refinement strategy, which controls both the
error arising in the Newton linearisation, as well as the error in the $hp$--DG
finite element scheme; in the latter case, we exploit automatic $hp$--refinement
of the underlying finite element space. The performance of this
proposed algorithm is demonstrated for a series of numerical examples
presented in Section~\ref{sec:numerics}. Finally, in Section~\ref{sec:conlusions} we summarise
the work presented in this article and discuss potential future extensions.


\section{Newton Linearisation} \label{sec:newton_linearisation}

\subsection{An Adaptive Newton Approach}

We will briefly revisit an adaptive `black-box' prediction-type Newton algorithm from~\cite{AmreinWihler:15}, and refer to~\cite{Deuflhard:04} for more sophisticated approaches in more specific situations. Let us consider two Banach spaces $ X, Y $, with norms~$\|\cdot\|_X$ and~$\|\cdot\|_Y$, respectively. Then, given an open subset~$\Xi\subset X$, and a (possibly nonlinear) operator~$\F:\,\Xi\to Y$, we are interested in solving the {\em nonlinear} operator equation
\begin{equation}\label{eq:F=0}
\F(u)=0,
\end{equation} 
for some unknown zeros~$u\in\Xi$. Supposing that the Fr\'echet derivative~$\F'$ of~$\F$ exists in~$\Xi$ (or in a suitable subset), the classical Newton  method for solving~\eqref{eq:F=0} starts from an initial guess~$u_0\in\Xi$, and generates a sequence~$\{u_n\}_{n\ge 1}\subset X$ that is defined iteratively by the {\em linear} equation
\begin{equation}\label{eq:newton}
\F'(u_n)(u_{n+1}-u_n)=-\F(u_n),\qquad n\ge 0.
\end{equation}
Naturally, for this iteration to be well-defined, we need to assume that~$\F'(u_n)$ is invertible for all~$n\ge 0$, and that~$\{u_n\}_{n\ge 0}\subset\Xi$. 

In order to improve the reliability of the Newton method~\eqref{eq:newton} in the case that the initial guess~$u_0$ is relatively far away from a root $u_{\infty}\in\Xi$ of~$\F$, $\F(u_\infty)=0$, introducing some damping in the Newton method is a well-known remedy. In that case~\eqref{eq:newton} is rewritten as
\begin{equation}\label{eq:damped}
u_{n+1}=u_n-\Delta t_n\F'(u_n)^{-1}\F(u_n),\qquad n\ge 0,
\end{equation} 
where $\Delta t_n>0$, $n\ge 0$, is a damping parameter that may be adjusted {\em adaptively} in each iteration step. The selection of the Newton parameter~$\Delta t_n$ is based on the following idea from~\cite{AmreinWihler:15}: provided that~$\F'(u)$ is invertible on a suitable subset of~$\Xi\subset X$, we define the {\em Newton-Raphson transform} by
\[
u\mapsto\NF(u):=-\F'(u)^{-1}\F(u);
\]
see, e.g., \cite{SchneebeliWihler:11}. Then, rearranging terms in~\eqref{eq:damped}, we notice that
\begin{equation*}
\frac{u_{n+1}-u_{n}}{\Delta t_n}=\NF(u_n), \qquad n\ge 0,
\end{equation*}
i.e., \eqref{eq:damped} can be seen as the discretisation of the dynamical system
\begin{equation}\label{eq:davy}
\begin{split}
\dot{u}(t)&=\NF(u(t)), \quad t\geq 0,\qquad
 u(0)=u_0,
 \end{split}
\end{equation}
by the forward Euler scheme, with step size~$\Delta t_n>0$. 
For~$t\in[0,\infty)$, the solution~$u(t)$ of~\eqref{eq:davy}, if it exists, defines a trajectory in~$X$ that starts at~$u_0$, and that will potentially converge to a zero of~$\F$ as~$t\to\infty$. Indeed, this can be seen (formally) from the integral form of~\eqref{eq:davy}, that is,
\begin{equation*}
\F(u(t))=\F(u_0)e^{-t},\qquad t\ge 0,
\end{equation*}
which implies that~$\F(u(t))\to 0$ as~$t\to\infty$.

Now taking the view of dynamical systems, our goal is to compute an upper bound for the value of the step sizes~$\Delta t_n>0$ from~\eqref{eq:damped}, $n\ge 0$, so that the discrete forward Euler solution~$\{u_n\}_{n\ge 0}$ from~\eqref{eq:damped} stays reasonably close to the continuous solution of~\eqref{eq:davy}. Specifically, for a prescribed tolerance~$\tau>0$, a Taylor expansion analysis (see~\cite[Section~2]{AmreinWihler:15} for details) reveals that 
\[
u(t)=u_0+t\NF(u_0)+\frac{t^2}{2h}\eta_h+\mathcal{O}(t^3)+\mathcal{O}(t^2h\|\NF(u_0)\|_X^2),
\]
where, for any sufficiently small~$h>0$, we let~$\eta_h=\NF(u_0+h\NF(u_0))-\NF(u_0)$. Hence, after the first time step of length~$\Delta t_0>0$ there holds
\begin{equation}\label{eq:O}
u(\Delta t_0)-u_1=\frac{\Delta t_0^2}{2h}\eta_h+\mathcal{O}(\Delta t_0^3)+\mathcal{O}(\Delta t_0^2h\|\NF(u_0)\|_X^2),
\end{equation}
where~$u_1$ is the forward Euler solution from~\eqref{eq:damped}. Therefore, upon setting
\[
\Delta t_0=\sqrt{2\tau h\|\eta_h\|_X^{-1}},
\] 
we arrive at
\[
\|u(\Delta t_0)-u_1\|_X\le\tau+\mathcal{O}(\Delta t_0^3)+\mathcal{O}(\Delta t_0^2h\|\NF(u_0)\|_X^2).
\]
In order to balance the~$\mathcal{O}$-terms in~\eqref{eq:O} it is sensible to make the choice
\[
h=\mathcal{O}(\Delta t_0\|\NF(u_0)\|_X^{-2}),
\] 
i.e.,
\begin{equation}\label{eq:h}
h=\gamma \Delta t_0\|\NF(u_0)\|_X^{-2},
\end{equation}
for some parameter~$\gamma>0$. This leads to the following \emph{adaptive Newton algorithm}.\\

\begin{algorithm}~\label{al:zs}
Fix a tolerance $\tau>0$ as well as a parameter~$\gamma>0$, and set~$n\gets0$.
\begin{algorithmic}[1]
\State Start the Newton iteration with an initial guess $ u_{0}\in\Xi$.
\If {$n=0$} {choose
\[
\Delta t_0=\min\left\{\sqrt{2\tau\NN{\NF(u_{0})}_{X}^{-1}},1\right\},
\]
based on~\cite[Algorithm~2.1]{AmreinWihler:15} (cf. also~\cite{AmreinWihler:14}),}
\ElsIf {$n\ge 1$} {let $\kappa_{n}=\Delta t_{n-1}$, and $h_n=\gamma \kappa_n\|\NF(u_n)\|_X^{-2}$ based on~\eqref{eq:h}; define the Newton step size
\begin{equation}\label{eq:knew}
\Delta t_n=\min\left\{\sqrt{2\tau h_n\NN{\NF(u_n+h_n\NF(u_n))-\NF(u_n)}^{-1}_{X}},1\right\}.
\end{equation}}
\EndIf
\State Compute~$u_{n+1}$ based on the Newton iteration~\eqref{eq:damped}, and go to~({\footnotesize 3:}) with $n\leftarrow n+1 $. 
\end{algorithmic}
\end{algorithm}

We notice that the minimum in~\eqref{eq:knew} ensures that the step size~$\Delta t_n$ is chosen to be~1 whenever possible. Indeed, this is required in order to guarantee quadratic convergence of the Newton iteration close to a root (provided that the root is simple). Furthermore, we remark that the prescribed tolerance~$ \tau $ in the above adaptive strategy will typically be fixed {\em a priori}. Here,  for highly nonlinear problems featuring numerous or even infinitely many solutions, it is typically mandatory to select~$\tau\ll 1$ small in order to remain within the attractor of the given initial guess. This is particularly important if the starting value is relatively far away from a solution.

\subsection{Application to Semilinear PDEs}
In this article, we suppose that a (not necessarily unique) solution~$u\in X:=H^1_0(\Omega)$ of~\eqref{eq:PDE} exists; here, we denote by $H^1_0(\Omega)$ the standard Sobolev space of functions in~$H^1(\Omega)=W^{1,2}(\Omega)$ with zero trace on~$\partial\Omega$. Furthermore, signifying by~$X'=H^{-1}(\Omega)$ the dual space of~$X$, and upon defining the map $\F: X\rightarrow X'$ through
\begin{equation}\label{eq:Fweak}
\dprod{\F(u),v}:=  \int_{\Omega}\left\{ \eps\nabla u\cdot \nabla v+uv-f(u)v\right\}\dx\qquad \forall v\in X,
\end{equation}
where $\dprod{\cdot,\cdot}$ is the dual product in~$X'\times X$, the above problem~\eqref{eq:PDE} can be written as a nonlinear operator equation in~$X'$:
\begin{equation}\label{eq:F0}
u\in X:\qquad \F(u)=0.
\end{equation}
For any subset~$D\subseteq\Omega$, we denote by~$\|\cdot\|_{0,D}$ the $L^2$-norm on~$D$; in 
the case when $D=\Omega$, we simply write $\|\cdot\|_{0}$ in lieu of $\|\cdot\|_{0,\Omega}$.
With this notation, we note that the space~$X$ is equipped with the 
norm
\[
\|u\|^2_X:=\eps\|\nabla u\|^2_{0}+\|u\|_0^2,\qquad u\in X. 
\]
The Fr\'echet-derivative of the 
operator~$\F$ from~\eqref{eq:F0} at~$u\in X$ is given by
\begin{equation*}
\dprod{\F'(u)w,v}= \int_{\Omega}\left\{\eps\nabla w\cdot \nabla v+wv-f'(u)wv\right\}\dx,\qquad v,w\in X=H^1_0(\Omega),
\end{equation*}
where we write~$f'\equiv\partial_u f$. We note that, if there is a constant~$\omega>1$ for which~$f'(u)\in L^{\omega}(\Omega)$, then $\F'(u)$ is a well-defined linear and bounded mapping from~$X$ to~$X'$; see~\cite[Lemma~A.1]{AmreinWihler:15}.

Now given an initial guess~$u_0\in X$, the adaptive Newton method~\eqref{eq:damped} for~\eqref{eq:F0} is defined iteratively to find~$u_{n+1}\in X$ from~$u_n\in X$, $n\ge 0$, such that
\begin{equation*}
\F'(u_n)(u_{n+1}-u_n)=-\Delta t_n\F(u_n),
\end{equation*}
in~$X'$. When applied to~\eqref{eq:Fweak} and~\eqref{eq:F0}, this turns into
\begin{align*}
\int_{\Omega}\{\eps\nabla (u_{n+1}-u_n)&\cdot \nabla v+(u_{n+1}-u_n)v-f'(u_n)(u_{n+1}-u_n)v\}\dx\\
&=-\Delta t_n\int_{\Omega}\left\{ \eps\nabla u_n\cdot \nabla v+u_nv-f(u_n)v\right\}\dx\qquad\forall v\in X.
\end{align*}
Hence, for~$n\ge 0$, the updated Newton iterate~$u_{n+1}$ is defined through the \emph{linear} weak formulation
\begin{equation}\label{eq:weak}
\begin{split}
\int_\Omega \{\eps\nabla \ut_{n+1}\cdot\nabla v&+\ut_{n+1}v -f'(u_n) \ut_{n+1}v \}\dx 
=\Delta t_n\int_\Omega \{ f(u_n)-f'(u_n)u_n\} v\dx
\qquad\forall v\in X,
\end{split}
\end{equation}
where~$\ut_{n+1}=u_{n+1}-(1-\Delta t_n)u_n$. Incidentally, if there exists a constant~$\delta$ with~$\eps^{-1}(f'(u_n)-1)\le\delta<\CP^{-2}$ on~$\Omega$, where~$\CP=\CP(\Omega)>0$ is the constant in the Poincar\'e-Friedrichs inequality on~$\Omega$,
\[
\|w\|_{0}\le \CP\|\nabla w\|_{0}\qquad\forall w \in X,
\]
then \eqref{eq:weak} is a linear second-order diffusion-reaction problem that is coercive on~$X$. In particular, \eqref{eq:weak} exhibits a unique solution~$u_{n+1}\in X$ in this case.


\section{$hp$--DG Discretisation}\label{sc:dgfem}

\subsection{Meshes, Spaces, and DG Flux Operators}

We will employ a standard $hp$--DG setting; see, e.g.,~\cite{HoustonSchotzauWihler:07,SchotzauZhu:11}.

\subsubsection{Meshes and DG Spaces}
Let $\T$ be a subdivision of $\Omega$ into disjoint open parallelograms $\kappa$ such that $\overline \Omega =
\bigcup_{\kappa \in \T} {\overline \kappa}$. We assume
that~$\T$ is shape-regular, and that each $\kappa \in \T$ is an affine
image of the unit square~$\widehat\kappa=(0,1)^2$; i.e., for each
$\kappa\in\T$ there exists an affine element mapping $\Psi_\kappa:\,\widehat\kappa\to\kappa$ such that $\kappa = \Psi_\kappa({\widehat \kappa})$. By $h_\kappa$ we denote the
element diameter of $\kappa\in\T$, $h = \max_{\kappa \in {\mathcal
T}_h} h_\kappa$ is the mesh size, and $\bm n_\kappa$ signifies the unit outward
normal vector to $\kappa$ on~$\partial\kappa$. Furthermore, we assume that $\T$ is of bounded
local variation, i.e., there exists a constant $\rho_1\ge 1$,
independent of the element sizes, such that $\rho_1^{-1}\le \nicefrac{h_\kappa}{h_{\kappa'}}\le \rho_1$, for any pair of elements $\kappa, \kappa'\in\T$ which share a common edge $e=(\partial\kappa\cap\partial\kappa')^\circ$. In this context, let us consider the set ${\mathcal E}$ of all one-dimensional open edges of all elements $\kappa \in\T$. Further, we denote by $\mathcal{E}_{\I}$ the set of all edges $e$ in $\mathcal{E}$ that are contained in $\Omega$ (interior edges). Additionally, introduce ${\mathcal E}_{\BB}$ to be the set of boundary edges consisting of all  $e\in{\mathcal E}$ that are contained in $\partial\Omega$. In our analysis, we allow the meshes to be 1-irregular, i.e., each edge of an element~$\kappa\in\T$ may contain (at most) one hanging node, which we assume to be located at the centre of~$e$. Suppose that $e$ is an edge of an element $\kappa \in \T$; then, by $h_e$, we denote the length of $e$. Due to our assumptions on the subdivision $\T$ we have that, if $e \subset \partial \kappa$, then $h_e$ is commensurate with $h_{\kappa}$, the diameter of
$\kappa$. 

For a nonnegative integer $k$, we denote by ${\mathcal Q}_k({\widehat \kappa})$ 
the set of all tensor-product polynomials on $\widehat \kappa$ of degree $k$ in
each co-ordinate direction. To each $\kappa \in \T$ we
assign a polynomial degree $p_{\kappa}$ (local approximation order). We store the quantities $h_\kappa$ and $p_{\kappa}$ in the
vectors ${\bm h} = \{ h_{\kappa}:\, \kappa \in\T\}$ and ${\bm p} = \{
p_{\kappa}:\, \kappa \in \T\}$, respectively, and
consider the DG finite element space
\begin{equation}\label{eq:VDG}
\V = \{ v \in L^2(\Omega):\;v |_{\kappa} \circ \Psi_{\kappa} \in
{\mathcal Q}_{p_{\kappa}}
    (\widehat{\kappa}) \quad \forall \kappa \in \T \} \; .
\end{equation}
We shall suppose that the polynomial degree vector $\mathbf{p}$, with
$p_{\kappa}\geq 1$ for each $\kappa \in \mathcal{T}$, has 
bounded local variation, i.e., there exists a constant $\rho_2\geq
1$ independent of $\bm h$ and $\bm p$, such that, for any pair of
neighbouring elements $\kappa, \kappa'\in\T$, we have $\rho_2^{-1} \leq \nicefrac{p_{\kappa}}{p_{\kappa'}} \leq \rho_2$. Moreover, for an edge~$e=(\partial\kappa\cap\partial\kappa')^\circ$ shared by two elements~$\kappa,\kappa'\in\T$, we define~$p_e:=\nicefrac12(p_{\kappa}+p_{\kappa'})$, or~$p_e=p_\kappa$ if~$e=(\partial\kappa\cap\partial\Omega)^\circ$, for some~$\kappa\in\T$, is a boundary edge. 

\subsubsection{Jump and Average Operators}
Let $\kappa$ and $\kappa'$ be two adjacent elements of~ $\T$, and
$\bm x$ an arbitrary point on the interior edge $e\in\E_{\I}$ given by
$e=(\partial \kappa\cap\partial \kappa')^\circ$.  Furthermore, let $v$
and~$\bm{q}$ be scalar- and vector-valued functions, respectively,
that are sufficiently smooth inside each
element~$\kappa,\kappa'$. Then, the averages of $v$ and
$\bm{q}$ at $\bm{x}\in e$ are given by
\[
\avg{v}=\frac{1}{2}(v|_{\kappa}+v|_{\kappa'}), \qquad \avg{\bm{q}}
=\frac{1}{2}(\bm{q}|_{\kappa}+\bm{q}|_{\kappa'}),
\] 
respectively.  Similarly, the jumps of $v$ and $\bm{q}$ at $\bm{x}\in
e$ are given by
\[
\jmp{v} =v|_{\kappa}\,\bm{n}_{\kappa}+v|_{\kappa'}\,\bm{n}_{\kappa'},\qquad
\jmp{\bm{q}}=\bm{q}|_{\kappa}\cdot\bm{n}_{\kappa}+\bm{q}|_{\kappa'}\cdot\bm{n}_{\kappa'},
\]
respectively. On a boundary edge $e\in\E_\BB$,
we set $\avg{v}=v$, $\avg{\bm{q}}=\bm{q}$ and $\jmp{v}=v\bm{n}$,
with~$\bm{n}$ denoting the unit outward normal vector on the boundary
$\partial\Omega$.

Furthermore, we introduce, for an edge $e\in{\mathcal E}$, the discontinuity penalisation parameter~$\sigma$ by
\begin{equation}\label{sigma}
\sigma|_{e} = \frac{p_e^2}{h_e}.
\end{equation}
We conclude this section by equipping the DG space $\V$ with the DG norm
\begin{equation}\label{eq:DGnorm}
\NN{v}_{\DG}^2 := \eps\NN{\nT v}^2_{0}+\NN{v}^2_{0}+\int_{\E}(\eps\sigma+\sigma^{-1})|\jmp{v}|^2\ds,
\end{equation}
which is induced by the DG inner product
\begin{equation}\label{eq:ip}
(v,w)_{\DG} = \int_{\Omega}\left\{\eps
\nT v\cdot\nT w +vw\right\}\dx + \int_{\E} (\eps\sigma+\sigma^{-1})\jmp{w}\cdot\jmp{v}\ds. 
\end{equation}
Here, $\nT$ is the element-wise gradient operator. For an element~$\kappa\in\T$ we shall also use the norm
\[
\|v\|_{\eps,\kappa}^2:=\eps\NN{\nabla v}^2_{0,\kappa}+\NN{v}^2_{0,\kappa},
\]
for~$v\in H^1(\kappa)$.

\subsubsection{Conforming Subspaces}

For a given DG finite element space~$\V$, cf.~\eqref{eq:VDG}, we define the extended space
\[
\W:=H^1_0(\Omega)+\V.
\] 
With this notation, the following result holds.

\begin{lemma}\label{lm:A}
There exists a linear operator~$\A:\,\W\to H^1_0(\Omega)$ such that
\begin{equation}\label{eq:equiv}
\begin{split}
\|w-\A w\|^2_{0}&\le \Cp\sum_{e\in\E}\int_{e}\sigma^{-1}|\jmp{w}|^2\ds,\\
\|\nT(w-\A w)\|^2_{0}&\le \Cp\sum_{e\in\E}\int_{e}\sigma|\jmp{w}|^2\ds, 
\end{split}
\end{equation}
for any~$w\in\W$, where~$\Cp>0$ is a constant independent of~$\T$ and of~$\bm p$.
\end{lemma}

\begin{proof}
Consider the space~$\V^\|:=\V\cap H^1_0(\Omega)$, and denote by~$\P^\|:\,\V\to\V^\|$ 
the orthogonal projection with respect to the inner product defined in~\eqref{eq:ip},
i.e.,
\[
w\in\V:\quad (w-\P^\|w,v)_{\DG}=0\qquad\forall v\in \V^\|.
\]
Then, defining the subspace~$\V^\perp:=(\id-\P^\|)\V$, we have the direct 
sum~$\V=\V^\|\oplus\V^\perp$, as well as
\begin{equation}\label{eq:spacedecomp}
\W=H^1_0(\Omega)\oplus\V^\perp.
\end{equation}
Based on our assumptions on the mesh~$\T$, and referring to~\cite[Theorem~4.4]{SchotzauZhu:11}, there exists an operator~$\Ihp:\,\V\to H^1_0(\Omega)$ that satisfies
\begin{align*}
\sum_{\kappa\in\T}\|v-\Ihp v\|^2_{L^2(\kappa)}&\le C\sum_{e\in\E}\int_{e}\sigma^{-1}|\jmp{v}|^2\ds,\\
\sum_{\kappa\in\T}\|\nabla(v-\Ihp v)\|^2_{L^2(\kappa)}&\le C\sum_{e\in\E}\int_{e}\sigma|\jmp{v}|^2\ds,
\end{align*}
for any~$v\in\V$. By virtue of~\eqref{eq:spacedecomp}, we can now construct the operator~$\A$ as follows: for any~$w\in\W$, there exist unique representatives~$w_0\in H^1_0(\Omega)$ and~$w_{\DG}^\perp\in\V^\perp$ with~$w=w_0+w_{\DG}^\perp$.
Hence, defining~$\A w:=w_0+\Ihp w_{\DG}^\perp\in H^1_0(\Omega)$, and employing the previous estimates, we obtain
\begin{align*}
\|\nT(w-\A w)\|^2_{0}
&=\sum_{\kappa\in\T}\|\nabla(w_{\DG}^\perp-\Ihp w_{\DG}^\perp)\|^2_{L^2(\kappa)}
\le C\sum_{e\in\E}\int_{e}\sigma|\jmp{w_{\DG}^\perp}|^2\ds.
\end{align*}
Since~$w_0\in H^1_0(\Omega)$, we notice that~$\jmp{w_0}|_e=\bm 0$ for all~$e\in\E$;
thereby,
\[
\|\nT(w-\A w)\|^2_{0}
\le C\sum_{e\in\E}\int_{e}\sigma|\jmp{w}|^2\ds,
\]
which proves the second bound in~\eqref{eq:equiv}. The first inequality results from an analogous argument.
\end{proof}

\begin{remark}\label{rm:A}
We note that any $v\in H^1_0(\Omega)$ satisfies~$\jmp{v}=\bm 0$ on~$\E$; thereby, in view of~\eqref{eq:equiv}, it follows that~$\A v=v$ for all~$v\in H^1_0(\Omega)$. Furthermore, for~$w\in\W$, upon application of the triangle inequality and Lemma~\ref{lm:A}, we deduce that
\begin{align*}
\|\A w\|_{X}^2
&=\eps\|\nabla\A w\|_0^2+\|\A w\|^2_0\\
&\le2\eps\|\nabla w\|_0^2+2\|w\|^2_0
+2\eps\|\nabla(w-\A w)\|_0^2+2\|w-\A w\|^2_0\\
&\le2\eps\|\nabla w\|_0^2+2\|w\|_0^2
+2\Cp\sum_{e\in\E}\int_{e}(\eps\sigma+\sigma^{-1})|\jmp{w}|^2\ds.
\end{align*}
Thus the following stability estimate holds
\begin{equation}\label{eq:stab}
\|\A w\|_{X}
\le C_\eqref{eq:stab}\|w\|_{\DG}\qquad\forall w\in\W,
\end{equation}
where~$C_\eqref{eq:stab}=\sqrt{2\max(1,\Cp)}$.
\end{remark}

\subsection{Linear $hp$--DG Approximation}

The $hp$--version interior penalty DG discretisation of~\eqref{eq:weak} 
is given by: find~$\uDG_{n+1}\in\V$ from~$\uDG_n$ such that
\begin{equation}\label{eq:dgfem}
\a(\uDG_n;\uDG_{n+1},v)=\l(\uDG_n;v)\qquad\forall v\in\V.
\end{equation}
Here, for a method parameter~$\theta\in[-1,1]$ and a penalty parameter~$C_{\sigma}\ge 0$, we define the forms
\begin{equation}\label{eq:a}
\begin{split}
\a(\uDG_n;\uDG_{n+1},v):=
&\int_\Omega\left\{\eps\nT\utDG_{n+1}\cdot\nT v+\utDG_{n+1} v-f'(\uDG_n)\utDG_{n+1}v\right\}\dx\\
&-\int_{\E}\left\{\avg{\eps\nT\utDG_{n+1}}\cdot\jmp{v}+\theta\jmp{\utDG_{n+1}}\cdot\avg{\eps\nT v}\right\}\ds\\
&\quad+C_{\sigma}\int_{\E}\eps\sigma\jmp{\utDG_{n+1}}\cdot\jmp{v}\ds,
\end{split}
\end{equation}
and
\[
\l(\uDG_n;v) = \int_\Omega\f{\uDG_n}v\dx,
\]
for~$v\in\V$, where for $n\ge 0$, we set
\begin{equation} \label{eq:hat}
\begin{split}
\utDG_{n+1} & :=\uDG_{n+1}-(1-\Delta t_n)\uDG_n,	\\
\f{\uDG_n} & := \Delta t_n (f(\uDG_n)-f'(\uDG_n)\uDG_n).
\end{split}
\end{equation}
The choices~$\theta\in\{-1,0,1\}$ correspond, respectively, to the non-symmetric (NIPG), 
incomplete (IIPG), and symmetric (SIPG) interior penalty DG schemes; 
cf.~\cite{StammWihler:10}. For the IIPG and SIPG methods,
the penalty parameter $C_{\sigma}$ must be chosen sufficiently large to guarantee
stability of the underlying DG scheme, 
cf. \cite{WihlerFrauenfelderSchwab:03}, for example. Furthermore, an additional
constraint on the minimal value of $C_{\sigma}$ will be introduced in 
Proposition~\ref{prop:residual_bounds} below.

\section{$hp$--Version \emph{A Posteriori} Analysis}
\label{sc:apost}

\subsection{A DG Residual}
We introduce a residual operator
\[
\R:\,\W\to \W',
\]
where~$\W'$ is the dual space of~$\W$, as follows: given the operator~$\A$ constructed in Lemma~\ref{lm:A}, and~$w\in\W$, let us define
\begin{equation}\label{eq:R}
\begin{split}
\dprod{\R(w),v}:&=\int_\Omega\left\{\eps\nT w\cdot\nabla \A v+w\A v-f(w)\A v\right\}\dx\\
&\quad+C_{\sigma}\int_{\E}(\eps\sigma+\sigma^{-1})\jmp{w}\cdot\jmp{v}\ds
\qquad\forall v\in \W,
\end{split}
\end{equation}
with~$\sigma$ from~\eqref{sigma}, and~$C_{\sigma}$ appearing in~\eqref{eq:a}. Furthermore, for~$w\in\W$, we introduce the norm
\begin{equation}\label{eq:NNN}
\NNN{\R(w)}:=\sup_{\phi\in\W}\frac{\dprod{\R(w),\phi}}{\|\phi\|_{\DG}}.
\end{equation}
For a solution~$u\in H^1_0(\Omega)$ of~\eqref{eq:PDE}, we again note 
that~$\jmp{u}=\bm 0$ on~$\E$, and, hence, due to~\eqref{eq:Fweak} and~\eqref{eq:F0}, we conclude that
\begin{equation}\label{eq:R0}
\dprod{\R(u),v}=0\qquad\forall v\in \W.
\end{equation}
Moreover, the following result shows that, under suitable conditions on the nonlinearity~$f$, the norm~$\NNN{\R(\cdot)}$ defined in~\eqref{eq:NNN} is directly related to the DG-norm given in~\eqref{eq:DGnorm}. In this sense, we may employ the norm~$\NNN{\R(\cdot)}$ as a natural measure for the approximation in the Newton-DG formulation~\eqref{eq:dgfem}. 

\begin{proposition} \label{prop:residual_bounds}
Suppose that there exist constants~$\varrho_0>-1$ and $L\ge0$ such that~$f$ satisfies
\begin{equation}\label{eq:f}
\varrho_0\le -f',\qquad\text{and}\qquad |f'|\le L,
\end{equation}
on~$\overline\Omega\times\mathbb{R}$. Furthermore, assume that the penalty 
parameter~$C_{\sigma}$ is sufficiently large so that
\[
C_{\sigma}\ge \frac{c_0}{2}+\frac{\Cp(1+L)^2}{2c_0},
\]
where~$\Cp$ is the constant arising in the bounds~\eqref{eq:equiv}, and~$c_0=1+\min(0,\varrho_0)>0$. Then, for any weak solution~$u\in H^1_0(\Omega)$ of~\eqref{eq:PDE}, the following bounds hold
\begin{equation}\label{eq:sup}
\frac{c_0}{2}\|u-w\|_{\DG}\le\NNN{\R(w)}
\le \sqrt2\max\left(C_{\eqref{eq:stab}}(1+L),C_{\sigma}\right)\|u-w\|_{\DG}
\end{equation}
for all~$w\in\W$, where~$C_\eqref{eq:stab}$ is the constant arising in~\eqref{eq:stab}.
\end{proposition}

\begin{proof}
The two bounds are proved separately. Let~$w\in\W$, then employing~\eqref{eq:R0}, and 
noting that~$\A u=u$, cf.~Remark~\ref{rm:A}, we obtain
\begin{align*}
\dprod{\R(w),w-u}
&=\dprod{\R(u)-\R(w),u-w}\\
&=\eps\int_\Omega\nT(u-w)\cdot\nabla(u-\A w)\dx+\int_\Omega(u-w)(u-\A w)\dx\\
&\quad-\int_\Omega(f(u)-f(w))(u-\A w)\dx
+C_{\sigma}\int_{\E}(\eps\sigma+\sigma^{-1})|\jmp{u-w}|^2\ds\\
&=\|u-w\|_{\DG}^2+\eps\int_\Omega\nT(u-w)\cdot\nT(w-\A w)\dx\\
&\quad+\int_\Omega(u-w)(w-\A w)\dx
-\int_\Omega(f(u)-f(w))(u-w)\dx\\
&\quad-\int_\Omega(f(u)-f(w))(w-\A w)\dx
+(C_{\sigma}-1)\int_{\E}(\eps\sigma+\sigma^{-1})|\jmp{u-w}|^2\ds.
\end{align*}
Given the assumptions on $f$ stated in~\eqref{eq:f} hold, we conclude that
\[
-(f(u)-f(w))(u-w)\ge\varrho_0|u-w|^2,\qquad |f(u)-f(w)|\le L|u-w|,
\]
on~$\overline\Omega\times\mathbb{R}$. Thus, applying the Cauchy-Schwarz inequality, we arrive at
\begin{align*}
\dprod{\R(w),w-u}
&\ge(1+\min(0,\varrho_0))\|u-w\|_{\DG}^2
-\eps\|\nT(u-w)\|_0\|\nT(w-\A w)\|_0\\
&\quad-(1+L)\|u-w\|_0\|w-\A w\|_0\\
&\quad+(C_{\sigma}-1-\min(0,\varrho_0))\int_{\E}(\eps\sigma+\sigma^{-1})|\jmp{u-w}|^2\ds.
\end{align*}
Setting~$c_0=1+\min(0,\varrho_0)$, we deduce that
\begin{align*}
\dprod{\R(w),w-u}
&\ge c_0\|u-w\|_{\DG}^2
-\frac{c_0\eps}{2}\|\nT(u-w)\|^2_0-\frac{\eps}{2c_0}\|\nT(w-\A w)\|^2_0\\
&\quad-\frac{c_0}{2}\|u-w\|_0^2-\frac{(1+L)^2}{2c_0}\|w-\A w\|^2_0\\
&\quad+(C_{\sigma}-c_0)\int_{\E}(\eps\sigma+\sigma^{-1})|\jmp{u-w}|^2\ds.
\end{align*}
By virtue of Lemma~\ref{lm:A}, and noting that~$\jmp{u}=\bm 0$ on~$\E$, we get
\begin{align*}
\dprod{\R(w),w-u}
&\ge\frac{c_0}{2}\|u-w\|^2_{\DG}
+\left(C_{\sigma}-\frac{c_0}{2}-\frac{\Cp(1+L)^2}{2c_0}\right)\int_{\E}(\eps\sigma+\sigma^{-1})|\jmp{u-w}|^2\ds\\
&\ge\frac{c_0}{2}\|u-w\|^2_{\DG}.
\end{align*}
This gives the first bound in~\eqref{eq:sup}. In order to show the second estimate, we employ~\eqref{eq:f} and the Cauchy-Schwarz inequality, for any~$v\in \W$, to infer that
\begin{align*}
\dprod{\R(w),v}
&=\dprod{\R(w)-\R(u),v}\\
&=\int_\Omega\left\{\eps\nT (w-u)\cdot\nabla \A v+(w-u)\A v-(f(w)-f(u))\A v\right\}\dx\\
&\quad+C_{\sigma}\int_{\E}(\eps\sigma+\sigma^{-1})\jmp{w-u}\cdot\jmp{v}\ds\\
&\le\eps\|\nT(w-u)\|_0\|\nabla\A v\|_0+(1+L)\|w-u\|_0\|\A v\|_0\\
&\quad+\left(C_{\sigma}^2C_\eqref{eq:stab}^{-2}\int_{\E}(\eps\sigma+\sigma^{-1})|\jmp{w-u}|^2\ds\right)^{\nicefrac12}
\left(C_\eqref{eq:stab}^2\int_{\E}(\eps\sigma+\sigma^{-1})|\jmp{v}|^2\ds\right)^{\nicefrac12}\\
&\le\max\left(1+L,C_{\sigma} C_\eqref{eq:stab}^{-1}\right)\|u-w\|_{\DG}
\left(\|\A v\|^2_{X}+C_\eqref{eq:stab}^2\int_{\E}(\eps\sigma+\sigma^{-1})|\jmp{v}|^2\ds\right)^{\nicefrac12}.
\end{align*}
Recalling the stability of~$\A$ from~\eqref{eq:stab} yields
\[
\dprod{\R(w),v}\le \sqrt2C_{\eqref{eq:stab}}\max\left(1+L,C_{\sigma} C_\eqref{eq:stab}^{-1}\right)\|u-w\|_{\DG}\|v\|_{\DG}.
\]
This implies the second bound in~\eqref{eq:sup}, and, thus, completes the proof.
\end{proof}

\subsection{\emph{A Posteriori} Residual Analysis}

In this section we develop a residual--based \emph{a posteriori} numerical analysis for the $hp$--NDG method~\eqref{eq:dgfem}.  

\subsubsection{$hp$--Approximation Estimates}\label{sc:approx}
Let~$v\in\W$ be arbitrary, and consider~$\A v\in H^1_0(\Omega)$ as in Lemma~\ref{lm:A}. Then, we may choose~$\phi^{\DG}\in\V$ such that, for all~$\kappa\in\T$, the stability bound
\[
\|\A v-\phi^{\DG}\|_{0,\kappa}\le \|\A v\|_{0,\kappa},
\]
as well as the approximation estimate
\begin{equation}\label{eq:interp}
\begin{split}
\|\nabla(\A v-\phi^{\DG})\|^2_{0,\kappa}
&+\frac{p_\kappa^2}{h_\kappa^2}\|\A v-\phi^{\DG}\|^2_{0,\kappa}
\le C_\eqref{eq:interp}\left(\|\nabla\A v\|^2_{0,\kappa}+\|\A v\|_{0,\kappa}^2\right)
\end{split}
\end{equation}
hold \emph{simultaneously}, where $C_\eqref{eq:interp}$ is a positive constant, independent of~$\bm h, \bm p$, 
and~$\A v$; see~\cite[\S~3.1]{KarkulikMelenk:15}. Since~$\eps\in(0,1]$, we infer the bound
\[
\eps\|\nabla(\A v-\phi^{\DG})\|^2_{0,\kappa}\le C_\eqref{eq:interp}\|\A v\|_{\eps,\kappa}^2,
\]
and
\begin{equation}\label{eq:interp3}
\eps^{\nicefrac12}\|\nabla\phi^{\DG}\|_{0,\kappa}
\le\eps^{\nicefrac12}\|\nabla(\A v-\phi^{\DG})\|_{0,\kappa}+
\eps^{\nicefrac12}\|\nabla\A v\|_{0,\kappa}
\le C_\eqref{eq:interp3}\|\A v\|_{\eps,\kappa}.
\end{equation}
Moreover, following the approach outlined in~\cite{Verfurth:98} (see also~\cite{AmreinWihler:15}), we deduce from the above estimates that
\begin{align}\label{eq:L2bound}
\|\A v-\phi^{\DG}\|^2_{0,\kappa}
&\le \min\left(1,C_\eqref{eq:interp}\eps^{-1}h_\kappa^2p^{-2}_\kappa\right)\|\A v\|^2_{\eps,\kappa}
\le \max\left(1,C_\eqref{eq:interp}\right)\alpha^2_\kappa\|\A v\|^2_{\eps,\kappa},
\end{align}
where, for $\kappa\in\T$,
\begin{equation}\label{eq:alpha}
\alpha_\kappa:=\min\left(1,\eps^{-\nicefrac12}h_\kappa p^{-1}_\kappa\right).
\end{equation}
Furthermore, applying a multiplicative trace inequality, that is,
\begin{equation}\label{eq:trace}
\|\psi\|^2_{0,\partial\kappa}
\le C_{\eqref{eq:trace}}\left(h_\kappa^{-1}\|\psi\|^2_{0,\kappa}+\|\psi\|_{0,\kappa}\|\nabla\psi\|_{0,\kappa}\right),\qquad \psi\in H^1(\kappa),
\end{equation}
we obtain
\[
\|\A v-\phi^{\DG}\|^2_{0,\partial\kappa}
\le C_\eqref{eq:trace}\max\left(1,C_\eqref{eq:interp}\right)\widetilde\beta^2_\kappa\|\A v\|_{\eps,\kappa}^2,
\]
where, for~$\kappa\in\T$, we define
\[
\widetilde\beta_\kappa:=\sqrt{h^{-1}_\kappa\alpha^2_\kappa+\eps^{-\nicefrac12}\alpha_\kappa}.
\]
Noting the bound
\[
\widetilde\beta_\kappa^2= \eps^{-\nicefrac12}\alpha_\kappa\left(\eps^{\nicefrac12}h_\kappa^{-1}\alpha_\kappa+1\right)
\le\eps^{-\nicefrac12}\alpha_\kappa(p_\kappa^{-1}+1)\le 2\eps^{-\nicefrac12}\alpha_\kappa,
\]
we deduce that
\begin{equation}\label{eq:boundary}
\|\A v-\phi^{\DG}\|_{0,\partial\kappa}
\le C_\eqref{eq:boundary}\beta_\kappa\|\A v\|_{\eps,\kappa},
\end{equation}
where
\begin{equation}\label{eq:beta}
\beta_\kappa:=\eps^{-\nicefrac14}\alpha_\kappa^{\nicefrac12}.
\end{equation}

\subsubsection{Upper \emph{A Posteriori} Residual Bound}

In order to derive an \emph{a posteriori} residual estimate for the $hp$--NDG discretisation~\eqref{eq:dgfem}, we recall the residual
\begin{align*}
\dprod{\R(\uDG_{n+1}),v}
&=\int_\Omega\left\{\eps\nT \uDG_{n+1}\cdot\nabla \A v+\uDG_{n+1}\A v-f(\uDG_{n+1})\A v\right\}\dx\\
&\quad+C_{\sigma}\int_{\E}(\eps\sigma+\sigma^{-1})\jmp{\uDG_{n+1}}\cdot\jmp{v}\ds
\equiv T_1+T_2,
\end{align*}
cf.~\eqref{eq:R}, where we define
\begin{align*}
T_1&:=\int_\Omega\left\{\eps\nT \utDG_{n+1}\cdot\nabla \A v+\utDG_{n+1}\A v-(f'(\uDG_{n})\utDG_{n+1}+\f{\uDG_n)}\A v\right\}\dx\\
&\quad+ C_{\sigma} \int_{\E}(\eps\sigma+\sigma^{-1})\jmp{\utDG_{n+1}}\cdot\jmp{v}\ds,\\
T_2&:=(1-\Delta t_n) \dprod{\R(\uDG_{n}),v}
+ \int_\Omega\left\{f(\uDG_n)+f'(\uDG_{n})(\uDG_{n+1}-\uDG_n)-f(\uDG_{n+1})\right\}\A v.
\end{align*}
Here, $\utDG_{n+1}$ and~$\f{\uDG_n}$ are given in~\eqref{eq:hat}, and~$v\in\W$ is again arbitrary. Recalling~\eqref{eq:dgfem}, we note that
\begin{align*}
\int_\Omega&\left\{\eps\nT\utDG_{n+1}\cdot\nT \phi^{\DG}+\utDG_{n+1}\phi^{\DG}-(f'(\uDG_n)\utDG_{n+1}+\f{\uDG_n})\phi^{\DG}\right\}\dx\\
&=\int_{\E}\left\{\avg{\eps\nT\utDG_{n+1}}\cdot\jmp{\phi^{\DG}}+\theta\jmp{\utDG_{n+1}}\cdot\avg{\eps\nT \phi^{\DG}}\right\}\ds
-C_{\sigma}\int_{\E}\eps\sigma\jmp{\utDG_{n+1}}\cdot\jmp{\phi^{\DG}}\ds,
\end{align*}
with~$\phi^{\DG}\in\V$ as in Section~\ref{sc:approx} above. Therefore,
\begin{align*}
T_1&=\int_\Omega\left\{\eps\nT\utDG_{n+1}\cdot\nT (\A v-\phi^{\DG})+\utDG_{n+1}(\A v-\phi^{\DG})\right\}\dx\\
&\quad-\int_\Omega(f'(\uDG_n)\utDG_{n+1} +\f{\uDG_n})(\A v-\phi^{\DG})\dx\\
&\quad+\int_{\E}\left\{\avg{\eps\nT\utDG_{n+1}}\cdot\jmp{\phi^{\DG}}+\theta\jmp{\utDG_{n+1}}\cdot\avg{\eps\nT \phi^{\DG}}\right\}\ds\\
&\quad+ C_{\sigma} \int_{\E}(\eps\sigma+\sigma^{-1})\jmp{\utDG_{n+1}}\cdot\jmp{v}\ds
-C_{\sigma}\int_{\E}\eps\sigma\jmp{\utDG_{n+1}}\cdot\jmp{\phi^{\DG}}\ds.
\end{align*}
Performing elementwise integration by parts in the first integral, and proceeding as in the proof of~\cite[Theorem~3.2]{HoustonSuliWihler:08}, the following estimate can be established:
\begin{align*}
C|T_1|
&\le\sum_{\kappa\in\T}\|\eps\Delta \utDG_{n+1}-\utDG_{n+1}
  +f'(\uDG_n)\utDG_{n+1}+\f{\uDG_n}\|_{0,\kappa}\|\A v-\phi^{\DG}\|_{0,\kappa}\\
  &\quad+\sum_{\kappa\in\T}\|\eps\jmp{\nT\utDG_{n+1}}\|_{0,\partial\kappa\setminus\partial\Omega}\|\A v-\phi^{\DG}\|_{0,\partial\kappa}
+\left(\sum_{\kappa\in\T}\frac{\eps p_\kappa^2}{h_\kappa}\|\jmp{\utDG_{n+1}}\|^2_{0,\partial\kappa}\right)^{\nicefrac12}\eps^{\nicefrac12}\|\nT\phi^{\DG}\|_0\\
&\quad  +\left(C_{\sigma}^2\int_{\E}(\eps\sigma+\sigma^{-1})|\jmp{\utDG_{n+1}}|^2\ds\right)^{\nicefrac12}
\left(\int_{\E}(\eps\sigma+\sigma^{-1})|\jmp{v}|^2\ds\right)^{\nicefrac12}\\
&\quad  +\left(C_{\sigma}^2\sum_{\kappa\in\T}\frac{\eps^2 \beta_\kappa^2p_\kappa^4}{h_\kappa^2}\|\jmp{\utDG_{n+1}}\|_{0,\partial\kappa}^2\right)^{\nicefrac12}
\left(\sum_{\kappa\in\T}\beta_\kappa^{-2}\|\jmp{\phi^{\DG}}\|_{0,\partial\kappa}^2\right)^{\nicefrac12}.
\end{align*}
Here, $C$ is a positive constant independent of~$\bm h$, $\bm p$, and~$\eps$, and~$\beta_\kappa$ is defined in~\eqref{eq:beta}. Observing that~$\jmp{\A v}=\bm 0$ on~$\E$, and recalling~\eqref{eq:boundary}, we infer the bound
\begin{align*}
\sum_{\kappa\in\T}\beta_\kappa^{-2}\|\jmp{\phi^{\DG}}\|^2_{0,\partial\kappa}
&=\sum_{\kappa\in\T}\beta_\kappa^{-2}\|\jmp{\phi^{\DG}-\A v}\|^2_{0,\partial\kappa}
\le C\sum_{\kappa\in\T}\beta_\kappa^{-2}\|\phi^{\DG}-\A v\|^2_{0,\partial\kappa}\le C\|\A v\|_{X}^2.
\end{align*}
Additionally, exploiting~\eqref{eq:interp3}, 
\eqref{eq:L2bound}, and~\eqref{eq:boundary}, yields
\begin{align*}
C|T_1|
&\le\sum_{\kappa\in\T}\|\eps\Delta \utDG_{n+1}-\utDG_{n+1}
  +f'(\uDG_n)\utDG_{n+1}+\f{\uDG_n}\|_{0,\kappa}\alpha_{\kappa}\|\A v\|_{\eps,\kappa}\\
  &\quad+\sum_{\kappa\in\T}\|\eps\jmp{\nT\utDG_{n+1}}\|_{0,\partial\kappa\setminus\partial\Omega}\beta_\kappa\|\A v\|_{\eps,\kappa}
+\left(\sum_{\kappa\in\T}\frac{\eps p_\kappa^2}{h_\kappa}\|\jmp{\utDG_{n+1}}\|^2_{0,\partial\kappa}\right)^{\nicefrac12}\|\A v\|_{X}\\
&\quad+\left(C_{\sigma}^2\sum_{\kappa\in\T}\left(\frac{\eps p_\kappa^2}{h_\kappa}+\frac{h_\kappa}{p_\kappa^2}\right)\|\jmp{\utDG_{n+1}}\|_{0,\partial\kappa}^2\ds\right)^{\nicefrac12}\|v\|_{\DG}\\
&\quad  +\left(C_{\sigma}^2\sum_{\kappa\in\T}\frac{\eps^2 \beta_\kappa^2p_\kappa^4}{h^2_\kappa}\|\jmp{\utDG_{n+1}}\|_{0,\partial\kappa}^2\right)^{\nicefrac12}\|\A v\|_X,
\end{align*}
with~$\alpha_\kappa$ defined in~\eqref{eq:alpha}. Observing that~$\alpha_\kappa\le\eps^{-\nicefrac12}h_\kappa p_\kappa^{-1}$ yields
\[
\max\left(\frac{\eps p_\kappa^2}{h_\kappa}+\frac{h_\kappa}{p_\kappa^2},\frac{\eps^2 \beta_\kappa^2p_\kappa^4}{h^2_\kappa}\right)\le \frac{\eps p_\kappa^3}{h_\kappa}+\frac{h_\kappa}{p_\kappa^2}.
\]
Hence, applying the Cauchy-Schwarz inequality, and making use of~\eqref{eq:stab}, we arrive at
\begin{align*}
|T_1|\le C\left(\sum_{\kappa\in\T}\eta_{\kappa,n}^2\right)^{\nicefrac12}\|v\|_{\DG},
\end{align*}
where, for any~$\kappa\in\T$, we define the local residual indicators
\begin{equation}\label{eq:eta}
\begin{split}
\eta_{\kappa,n}^2:
&=\alpha_\kappa^2\|\eps\Delta \utDG_{n+1}-\utDG_{n+1}
  +f'(\uDG_n)\utDG_{n+1}+\f{\uDG_n}\|^2_{0,\kappa}\\
&\quad+\beta^2_\kappa\eps^2\|\jmp{\nT\utDG_{n+1}}\|^2_{0,\partial\kappa\setminus\partial\Omega}
+\max\left(1,C_{\sigma}^2\right)\left(\frac{\eps p_\kappa^3}{h_\kappa}+\frac{h_\kappa}{p_\kappa^2}\right)\|\jmp{\utDG_{n+1}}\|^2_{0,\partial\kappa}.
\end{split}
\end{equation}

In order to deal with the term~$T_2$, we apply elementwise integration by parts to obtain
\begin{align*}
\int_\Omega &\{\eps\nT \uDG_n\cdot\nabla \A v+\uDG_n\A v-f(\uDG_n)\A v\}\dx\\
&=-\sum_{\kappa\in\T}\int_{\kappa}\{\eps\Delta\uDG_n-\uDG_n+f(\uDG_n)\}\A v\dx
+\int_{\E_\I}\jmp{\eps\nT\uDG_n}\A v\ds.
\end{align*}
Furthermore, we define the lifting operator
\[
\LL:\,\V\to\V,\qquad w\mapsto\LL(w),
\]
by
\[
\int_\Omega\LL(w)\phi^{\DG}\dx=\int_{\E_\I}\jmp{\nT w}\phi^{\DG}\ds
\qquad\forall\phi^{\DG}\in\V;
\]
cf., e.g., \cite{ArnoldBrezziCockburnMarini:01,SchotzauSchwabToselli:02}. Thereby, we note that
\begin{align*}
\int_\Omega& \{\eps\nT \uDG_n\cdot\nabla \A v+\uDG_n\A v-f(\uDG_n)\A v\}\dx\\
&=-\int_{\Omega}\{\eps\Delta_{\T}\uDG_n-\uDG_n+f(\uDG_n)-\eps\LL(\uDG_n)\}\A v\dx\\
&\quad+\int_{\E_\I}\jmp{\eps\nT\uDG_n}(\A v-\phi^{\DG})\ds
-\int_\Omega\eps\LL(\uDG_n)(\A v-\phi^{\DG})\dx,
\end{align*}
where~$\Delta_{\T}$ is the elementwise Laplacian operator. Applying the Cauchy-Schwarz inequality, and incorporating the bounds from Section~\ref{sc:approx}, we deduce that
\begin{align*}
\bigg|\int_\Omega& \{\eps\nT \uDG_n\cdot\nabla \A v+\uDG_n\A v-f(\uDG_n)\A v\}\dx\bigg|\\
&\le\|\eps\Delta_{\T}\uDG_n-\uDG_n+f(\uDG_n)-\eps\LL(\uDG_n)\|_{0}\|\A v\|_{0}\\
&\quad
+\sum_{\kappa\in\T}\eps\|\jmp{\nT\uDG_n}\|_{0,\partial\kappa\setminus\partial\Omega}
\|\A v-\phi^{\DG}\|_{0,\partial\kappa}
+\sum_{\kappa\in\T}\eps\|\LL(\uDG_n)\|_{0,\kappa}\|\A v-\phi^{\DG}\|_{0,\kappa}\\
&\le\|\eps\Delta_{\T}\uDG_n-\uDG_n+f(\uDG_n)-\eps\LL(\uDG_n)\|_{0}\|\A v\|_{X}+C\sum_{\kappa\in\T}\beta_\kappa\eps\|\jmp{\nT\uDG_n}\|_{0,\partial\kappa\setminus\partial\Omega}\|\A v\|_{\eps,\kappa}\\
&\quad+C\sum_{\kappa\in\T}\alpha_\kappa\eps\|\LL(\uDG_n)\|_{0,\kappa}\|\A v\|_{\eps,\kappa}\\
&\le\|\eps\Delta_{\T}\uDG_n-\uDG_n+f(\uDG_n)-\eps\LL(\uDG_n)\|_{0}\|\A v\|_{X}\\
&\quad
+C\left(\sum_{\kappa\in\T}\left(\beta_\kappa^2\eps^2\|\jmp{\nT\uDG_n}\|^2_{0,\partial\kappa\setminus\partial\Omega}
+\alpha_\kappa^2\eps^2\|\LL(\uDG_n)\|^2_{0,\kappa}\right)\right)^{\nicefrac12}\|\A v\|_{X}.
\end{align*}
Recalling~\eqref{eq:stab}, we get
\begin{align*}
\bigg|\int_\Omega& \{\eps\nT \uDG_n\cdot\nabla \A v+\uDG_n\A v-f(\uDG_n)\A v\}\dx\bigg|\\
&\le\|\eps\Delta_{\T}\uDG_n-\uDG_n+f(\uDG_n)-\eps\LL(\uDG_n)\|_{0}\|v\|_{\DG}\\
&\quad
+C\left(\sum_{\kappa\in\T}\left(\beta_\kappa^2\eps^2\|\jmp{\nT\uDG_n}\|^2_{0,\partial\kappa\setminus\partial\Omega}
+\alpha_\kappa^2\eps^2\|\LL(\uDG_n)\|^2_{0,\kappa}\right)\right)^{\nicefrac12}\|v\|_{\DG}.
\end{align*}
Furthermore, we have
\begin{align*}
\bigg|C_{\sigma}\int_{\E}(\eps\sigma+\sigma^{-1})\jmp{\uDG_{n}}\cdot\jmp{v}\ds\bigg|
&\le C\left(\sum_{\kappa\in\T}C_{\sigma}^2\left(\frac{\eps p_\kappa^2}{h_\kappa}+\frac{h_\kappa}{p_\kappa^2}\right)\|\jmp{\uDG_{n}}\|^2_{0,\partial\kappa}\right)^{\nicefrac12}\|v\|_{\DG}.
\end{align*}
Thus, in summary, we can bound~$T_2$ by
\[
|T_2|
\le C\delta_{n,\Omega}\|v\|_{\DG},
\]
where
\begin{equation}\label{eq:delta1}
\delta_{n,\Omega}:=(1-\Delta t_n)\delta^{(1)}_{n,\Omega}+\delta^{(2)}_{n,\Omega},
\end{equation}
with
\begin{equation}\label{eq:delta2}
\begin{split}
\delta^{(1)}_{n,\Omega}
&:=\|\eps\Delta_{\T}\uDG_n-\uDG_n+f(\uDG_n)-\eps\LL(\uDG_n) \|_{0} \\
&\quad
+ \left(\sum_{\kappa\in\T} \eps^2\alpha_\kappa\left(\eps^{-\nicefrac12}\|\jmp{\nT\uDG_n}\|^2_{0,\partial\kappa\setminus\partial\Omega} +
\alpha_\kappa\|\LL(\uDG_n)\|^2_{0,\kappa}\right) \right)^{\nicefrac12} \\
&\quad
+C_{\sigma} \left(\sum_{\kappa\in\T} \left(\frac{\eps p_\kappa^2}{h_\kappa}+\frac{h_\kappa}{p_\kappa^2}\right)\|\jmp{\uDG_{n}}\|^2_{0,\partial\kappa}\right)^{\nicefrac12},
\end{split}
\end{equation}
and
\begin{equation}\label{eq:delta3}
\delta^{(2)}_{n,\Omega}:=\|f(\uDG_n)+f'(\uDG_{n})(\uDG_{n+1}-\uDG_n)-f(\uDG_{n+1})\|_{0}.
\end{equation}

Thus we have proved the following key result.

\begin{theorem}\label{thm:apost}
For the $hp$--NDG method~\eqref{eq:dgfem}, the following upper \emph{a posteriori} residual bound holds
\[
\NNN{\R(\uDG_{n+1})}\le {\mathcal E}(\uDG_n,\uDG_{n+1},{\bm h},{\bm p}) \equiv 
C\left(\delta_{n,\Omega}^2+\sum_{\kappa\in\T}\eta_{\kappa,n}^2\right)^{\nicefrac12},
\]
where $C$ is a positive constant, independent of~$\bm h$, $\bm p$, the penalty parameter~$C_{\sigma}$, and~$\eps$. Moreover,~$\eta_{\kappa,n}$, $\kappa\in\T$, and~$\delta_{n,\Omega}$ are given in~\eqref{eq:eta} and~\eqref{eq:delta1}--\eqref{eq:delta3}, respectively. 
\end{theorem}

\begin{remark}
Following along the lines of~\cite[\S4.4.2]{AmreinWihler:15} and~\cite{HoustonSchotzauWihler:07}, it is possible to prove local lower residual bounds in terms of the error indicators~$\eta_{\kappa}$, $\kappa\in\T$, and some data oscillation terms. In contrast to the $h$--version approach in~\cite{AmreinWihler:15}, however, the local efficiency bounds will be slightly suboptimally scaled with respect to the local polynomial degrees due to the need of applying $p$--dependent norm equivalence results (involving cut-off functions).
\end{remark}

\section{$hp$--Adaptive NDG Scheme}
\label{sec:adaptivity}

In this section, we will discuss how the \emph{a posteriori} bound from Theorem~\ref{thm:apost} can be exploited in the design of an $hp$--adaptive NDG algorithm for the numerical approximation of~\eqref{eq:PDE}. 

\subsection{$hp$--Adaptive Refinement Procedure}\label{sc:hpadapt}
In order to enrich the finite element space $\V$, we shall apply an 
$hp$--adaptive refinement algorithm which is based on the following two ingredients:

\subsubsection*{(a) Element marking:} 
Each element $\kappa$ in the computational mesh $\mathcal{T}$ may be marked for
refinement on the basis of the size of the local residual indicators $\eta_{\kappa,n}$,
cf. \eqref{eq:eta}, $n\geq 0$. To this end, several strategies, such as 
equidistribution, fixed fraction, D\"{o}rfler marking, optimized mesh criterion, and so on,
cf. \cite{HoustonSuli02}, for example, have been proposed within the literature.
For the purposes of this article, we employ the {\em maximal strategy}: here,
we refine the set of elements $\kappa\in \mathcal{T}$ which satisfy the
condition
\[
\eta_{\kappa,n} > \Upsilon \max_{\kappa\in\T} \eta_{\kappa,n},
\]
where $0< \Upsilon <1$ is a given parameter. On the basis of \cite{De07,opac-b1101124,Houston2016977},
throughout this article, we set $\Upsilon = \nicefrac13$.

\subsubsection*{(b) $hp$--Refinement criterion:} 
Once an element $\kappa\in {\mathcal T}$ has been marked for refinement,
a decision must be made regarding whether to subdivide the element
($h$--refinement) or to increase the local degree of the polynomial
approximation on element $\kappa$ ($p$--refinement). Several strategies have
been proposed within the literature; for a recent review of
$hp$--refinement algorithms, we refer to~\cite{MitchellMcClain:14}.
Here we employ the $hp$--refinement strategy developed in \cite{HoustonSuli:05} 
where the local regularity of the analytical solution is estimated on the
basis of truncated local Legendre expansions of the computed numerical solution,
cf., also, \cite{FankhauserWihlerWirz:14,EibnerMelenk:07}.

\subsection{Fully Adaptive Newton-Galerkin Method}
We now propose a procedure that provides an \emph{interplay} of the Newton linearisation and automatic $hp$--finite element mesh refinements based on the {\em a posteriori} residual estimate from Theorem~\ref{thm:apost} (as outlined in the previous Section~\ref{sc:hpadapt}). To this end, we make the assumption that the NDG sequence $\left\{\uDG_{n+1}\right\}_{n\ge 0} $ given by~\eqref{eq:dgfem} is well-defined as long as the iterations are being performed. 

\begin{algorithm}\label{al:full}
Given a (coarse) starting mesh $\mathcal{T}$ in~$\Omega$, with an associated (low-order) polynomial degree distribution~$\bm p$, and an initial guess $ \uDG_{0} \in  \V $. Set~$n\gets 0$.
\begin{algorithmic}[1]
\State
Determine the Newton step size parameter $\Delta t_n$ based on~$\uDG_n$ by the adaptive procedure from Algorithm~\ref{al:zs}; the Newton-Raphson transform $\NF(\uDG_n)$ required for the computation of the step size parameter~$\Delta t_n$ is approximated using the $hp$--DG method on the current mesh.
\State 
Compute the DG solution~$\utDG_{n+1}$ from~\eqref{eq:dgfem}, and~$\uDG_{n+1}=\utDG_{n+1}+(1-\Delta t_n) \uDG_n$. Furthermore,  evaluate the corresponding residual indicators $ \{\eta_{\kappa,n}\}_{\kappa\in\T} $, and $\delta_{n,\Omega} $ from~\eqref{eq:eta} and~\eqref{eq:delta1}--\eqref{eq:delta2}, respectively.  
\If {
\begin{equation}\label{eq:test}
\delta_{n,\Omega}^2\le \Lambda \sum_{\kappa\in\T}{\eta_{\kappa,n}^2}
\end{equation}
holds, for some given parameter~$\Lambda >0$,} {$hp$--refine the space $\V$ adaptively based on the marking criterion and the $hp$--strategy outlined in Section~\ref{sc:hpadapt}; go back to step~({\footnotesize\sc 1:}) with the new mesh~$\mathcal{T}$ (and based on the previously computed solution~$\uDG_{n+1}$ interpolated on the refined mesh).}
\Else{, i.e., if~\eqref{eq:test} is not fulfilled, then set~$n\leftarrow n+1$, and perform another Newton step by going back to~({\footnotesize\sc 1:}).}
\EndIf
\end{algorithmic}
\end{algorithm}

\begin{remark} 
We note that our computational experience suggests that the choice of the element marking strategy can directly affect the robustness of the NDG scheme, particularly, when the numerical solution is far away from a given	solution. Indeed, it is essential to employ a marking scheme which adaptively adjusts the number of elements marked for refinement at each step of the adaptive process; algorithms such as the fixed fraction method which only mark a fixed percentage of elements at each refinement level can lead
to slow convergence of the combined adaptive Newton-Galerkin approach. 
\end{remark}

\section{Numerical Experiments} \label{sec:numerics}

In this section we present a series of numerical experiments to demonstrate the 
practical performance of the proposed $hp$--adaptive refinement strategy 
outlined in Algorithm~\ref{al:full}. To this end, throughout this section
we select $\tau=0.1$ and $\gamma=0.5$ in Algorithm~\ref{al:zs}, the penalty
parameter $C_{\sigma}=10$ and $\theta=1$ (SIPG) in the interior penalty
DG scheme \eqref{eq:dgfem}, cf. \eqref{eq:a}, and
$\Lambda = 0.5$ in Algorithm~\ref{al:full}, cf.~\cite{AmreinWihler:15}.
Throughout this section we shall compare the performance of the
proposed $hp$--adaptive refinement strategy with the corresponding
algorithm based on exploiting only local mesh subdivision, i.e.,
$h$--refinement.
Furthermore, within each inner linear iteration, we employ the \emph{direct} MUltifrontal Massively Parallel Solver (MUMPS) \cite{MUMPS:2,MUMPS:1,MUMPS:3}; in particular,  in Theorem~\ref{thm:apost}, we do not take into account any linear algebra errors resulting from iterative solvers (cf., e.g., \cite{El-AlaouiErnVohralik:11}).

\begin{figure}[t!]
	\begin{center}
\includegraphics[scale=0.5]{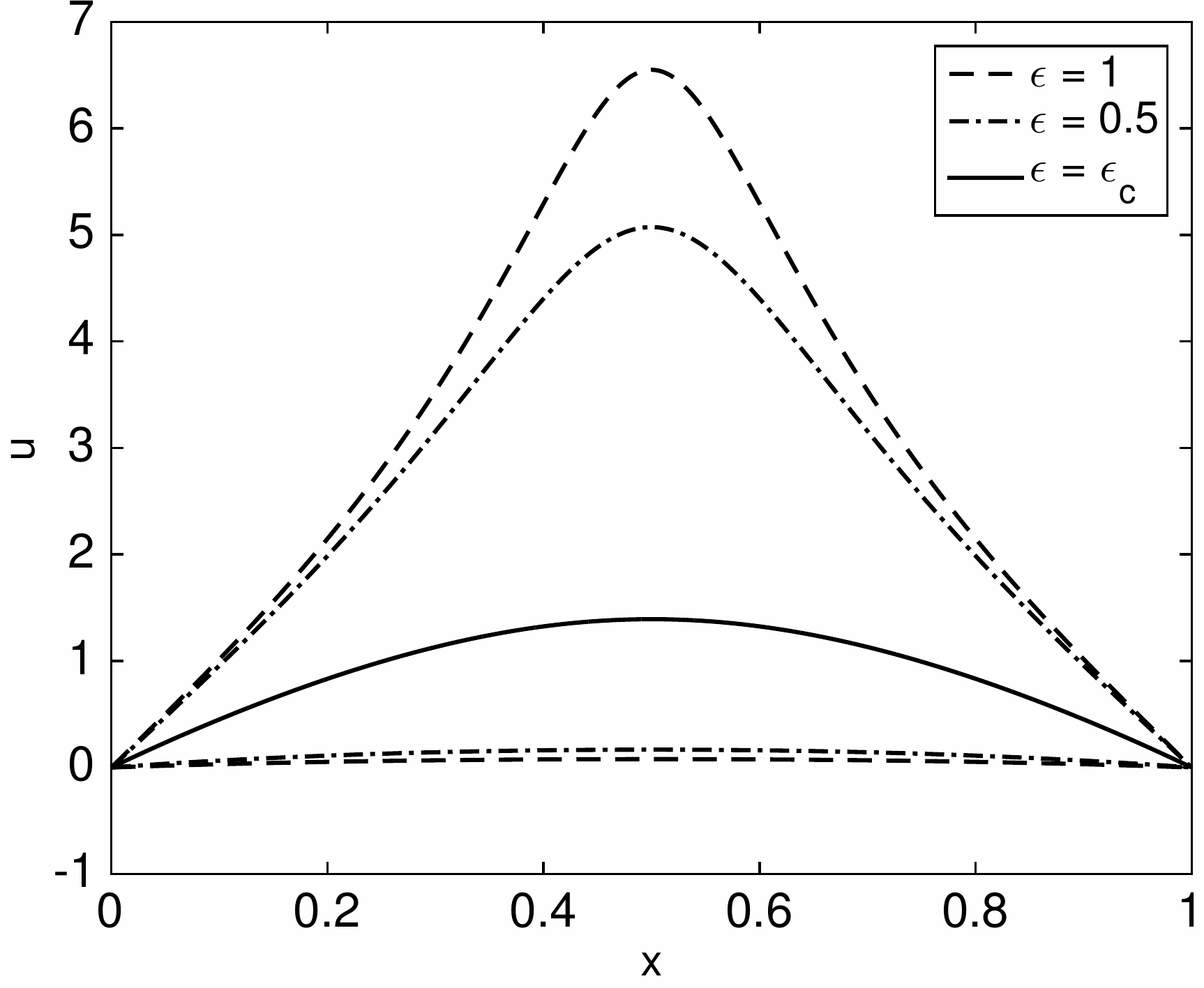} 
\end{center}
\caption{Bratu Problem. Slice at $y=0.5$, $0\leq x\leq 1$, of the upper 
and lower solutions computed 
with $\epsilon=1$ and $\epsilon = 0.5$, together with the critical solution
($\epsilon=\epsilon_c$).}
\label{fig:bratu_slices}
\end{figure}

\begin{figure}[t!]
	\begin{center}
	\begin{tabular}{cc}
\includegraphics[scale=0.19]{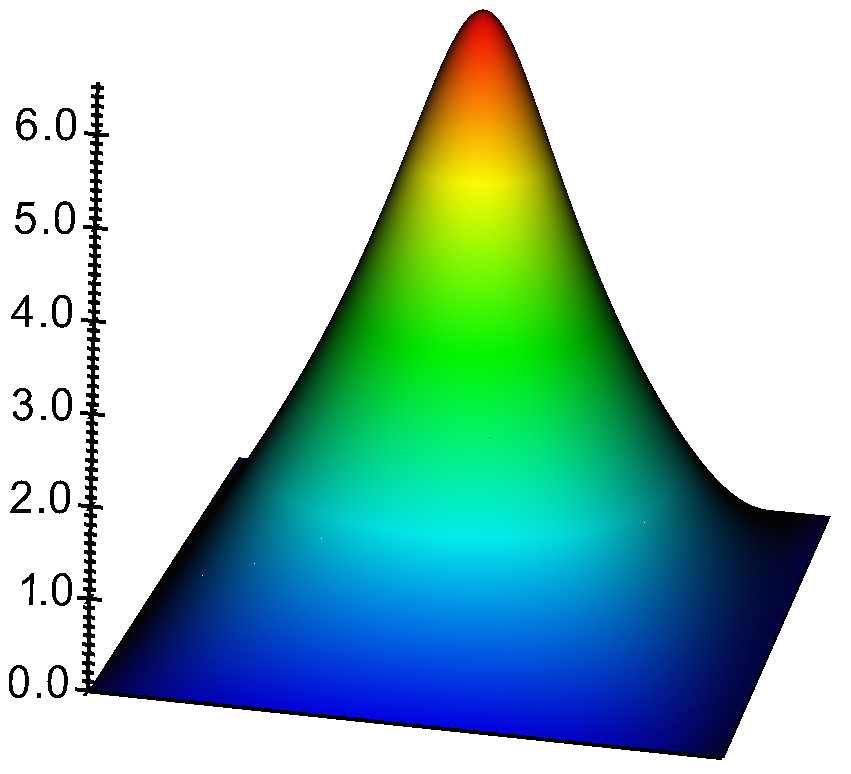} &
\includegraphics[scale=0.14]{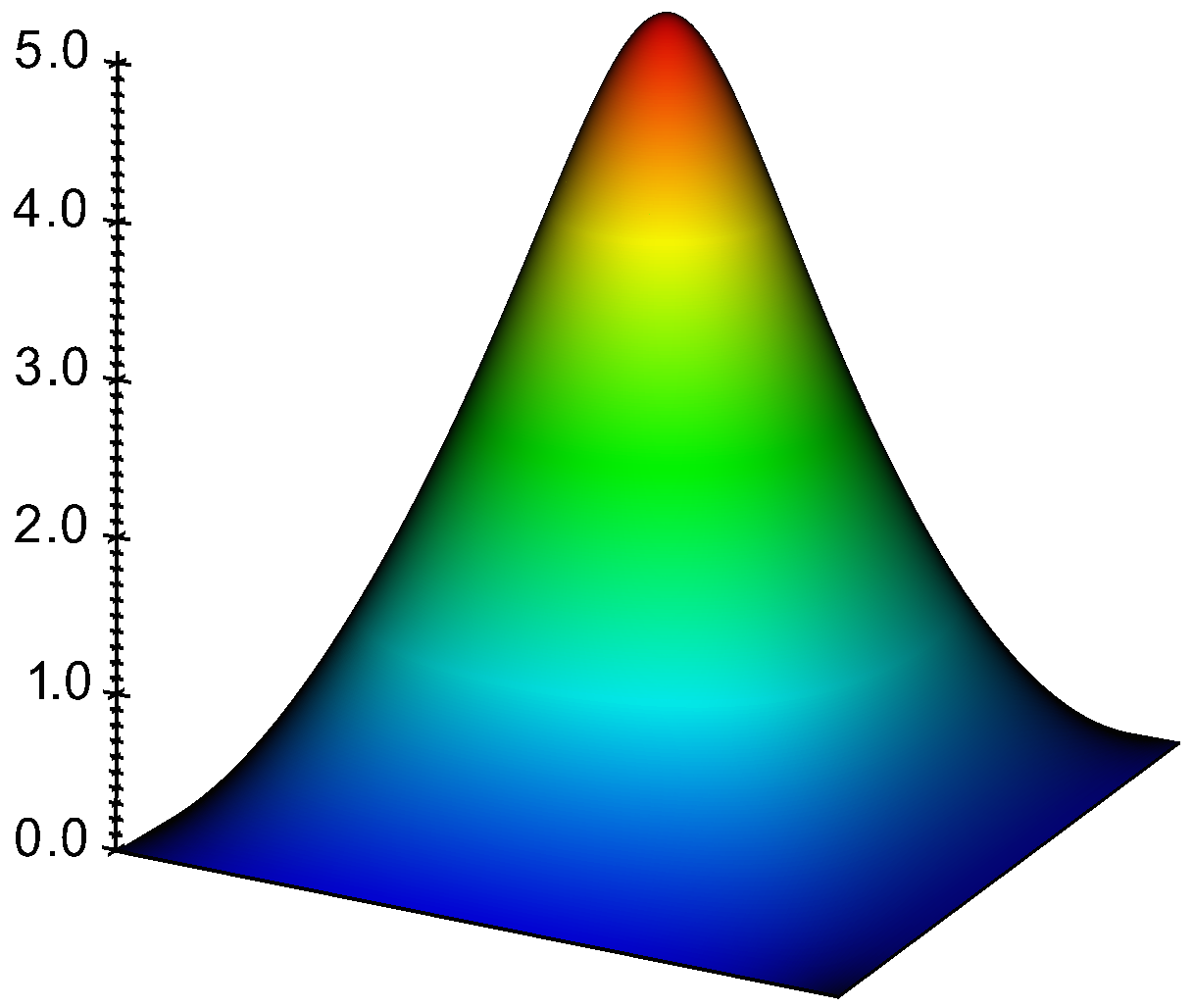} \\
(a) & (b)
\end{tabular}
\end{center}
\caption{Bratu Problem. Upper solution computed with:
(a) $\epsilon=1$; (b) $\epsilon = 0.5$.}
\label{fig:bratu_solutions}
\end{figure}

\begin{example}
In this first example, we consider the Bratu problem
\[
\eps\Delta u + {\rm e}^u = 0 \quad \mbox{ in } (0,1)^2,
\]
i.e., $f(u)={\rm e}^u+u$, subject to homogeneous Dirichlet boundary conditions on
$\partial\Omega$. Writing $\lambda = \nicefrac{1}{\epsilon}$, we recall that there
exists a critical parameter value $\lambda_c$  $(=1/\epsilon_c)$, such that for 
$\lambda > \lambda_c$ $(\epsilon < \epsilon_c)$ the problem has no solution, for 
$\lambda = \lambda_c$ $(\epsilon = \epsilon_c)$ there exists exactly one
solution, and for $\lambda < \lambda_c$ $(\epsilon > \epsilon_c)$ there are two solutions. 
In the one--dimensional
setting, an analytical expression for $\lambda_c$ is available, cf.
\cite{aschermattheijhrussell,CHHPSbratu,calvetti_2000}; for the two--dimensional case,
calculations have revealed that $\lambda_c = 6.808124423$ 
$(\epsilon_c=0.146883332)$ to 9 decimal places, 
see \cite{CHHPSbratu,mohsen08,mohsen_2014}, and the references cited therein.

\begin{figure}[t!]
	\begin{center}
		\begin{tabular}{cc}
\includegraphics[scale=0.38]{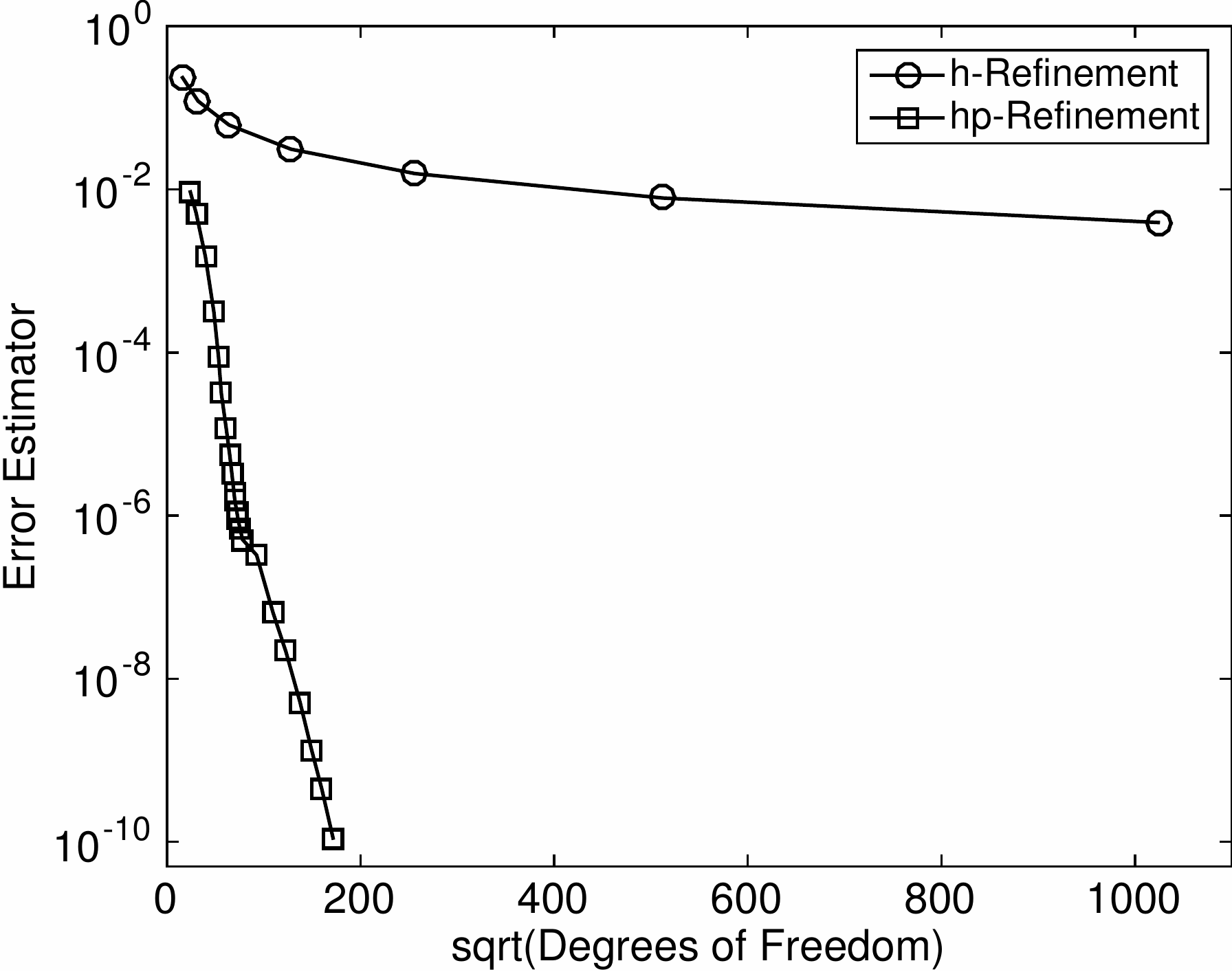} &
\includegraphics[scale=0.38]{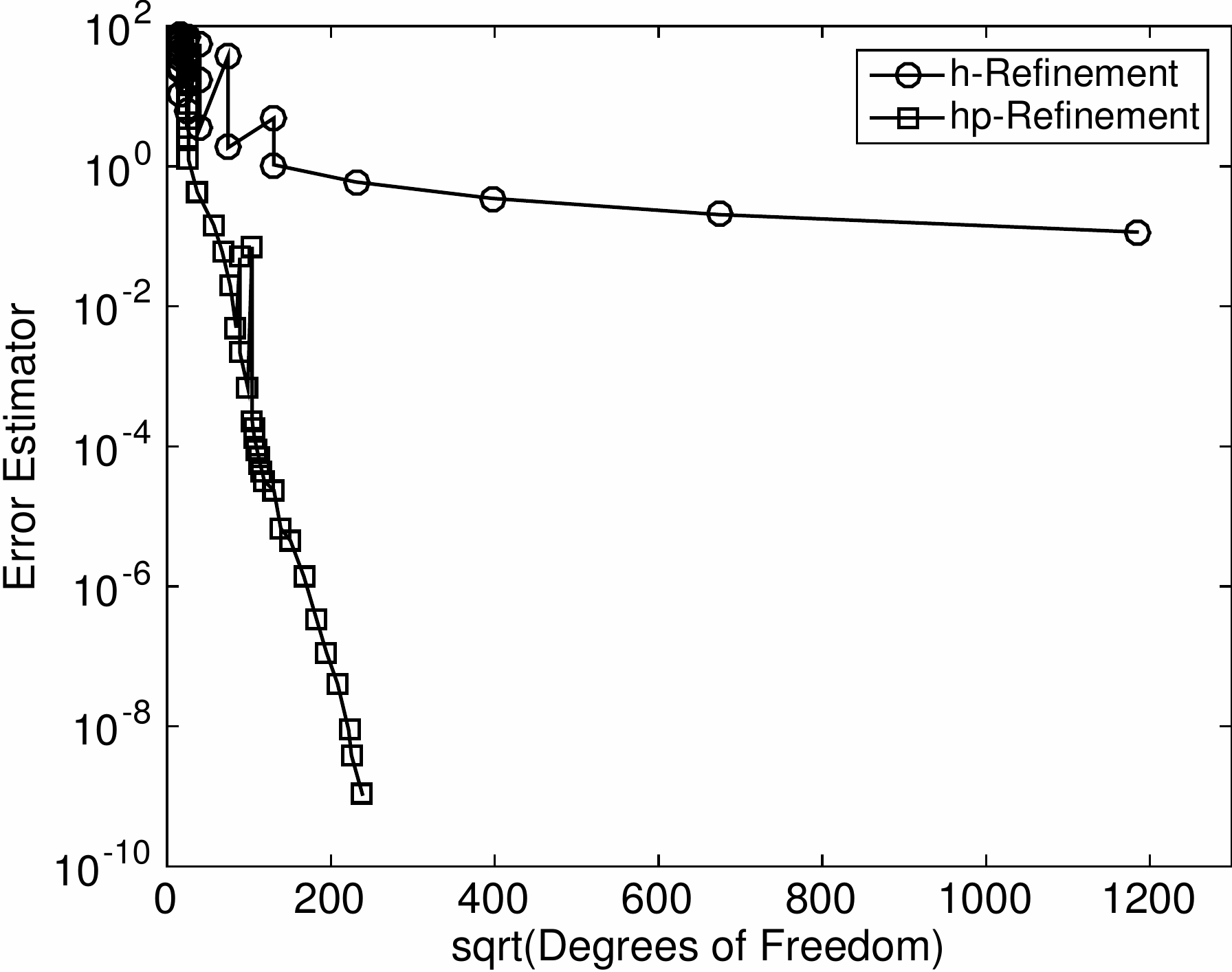} \\
(a) & (b) \\
~ & ~ \\
\includegraphics[scale=0.38]{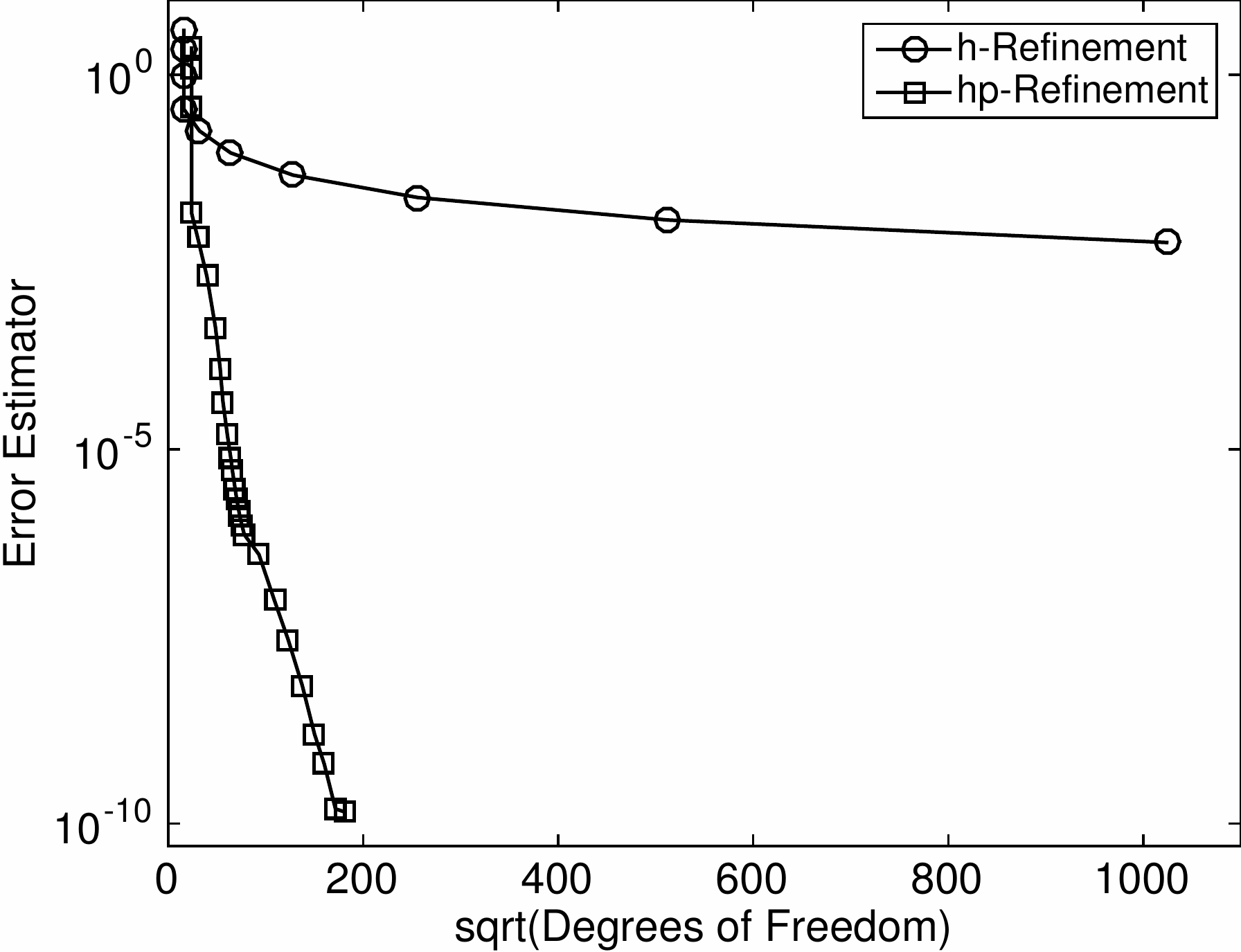} &
\includegraphics[scale=0.38]{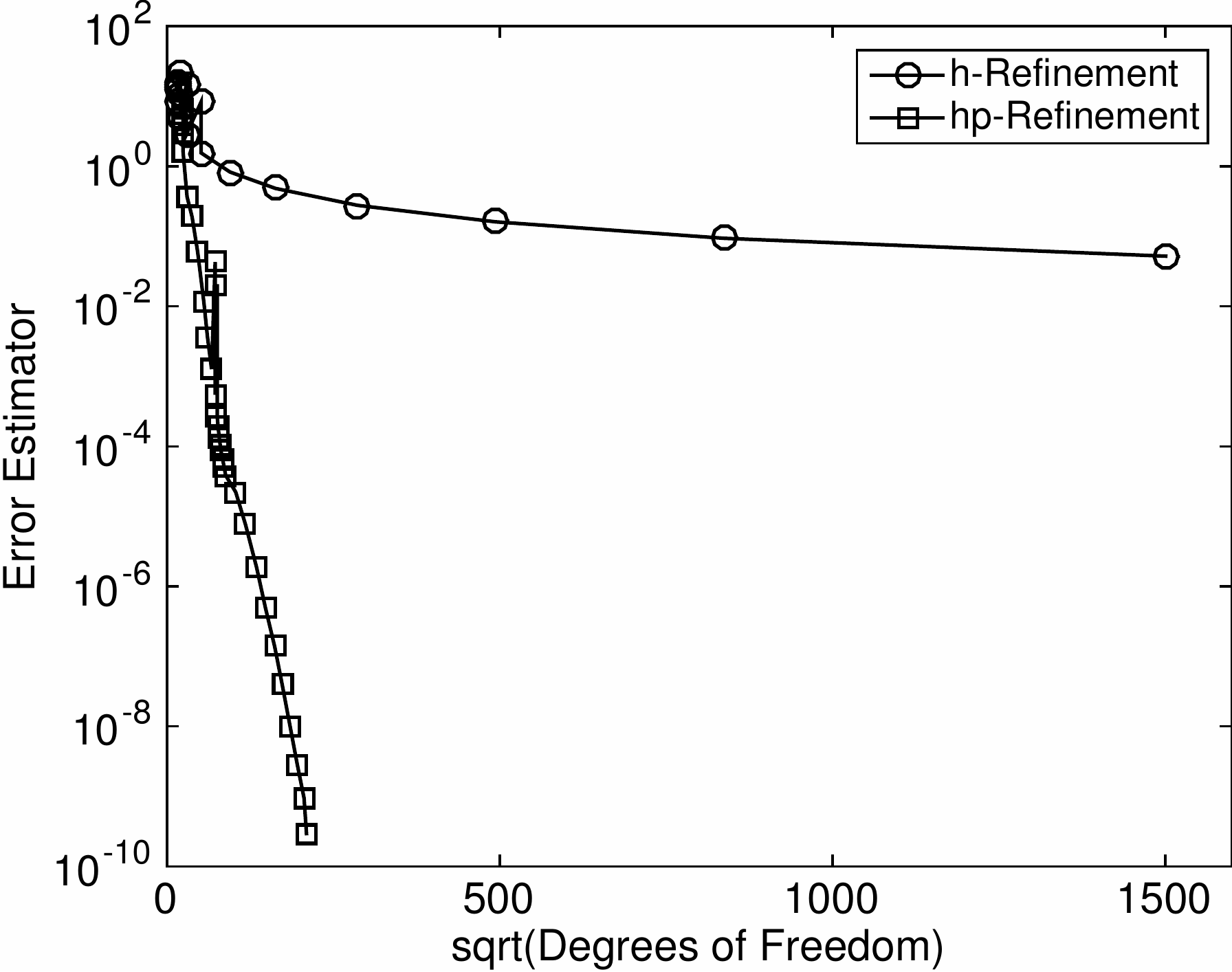} \\
(c) & (d) \\
~ & ~ \\
\multicolumn{2}{c}{\includegraphics[scale=0.38]{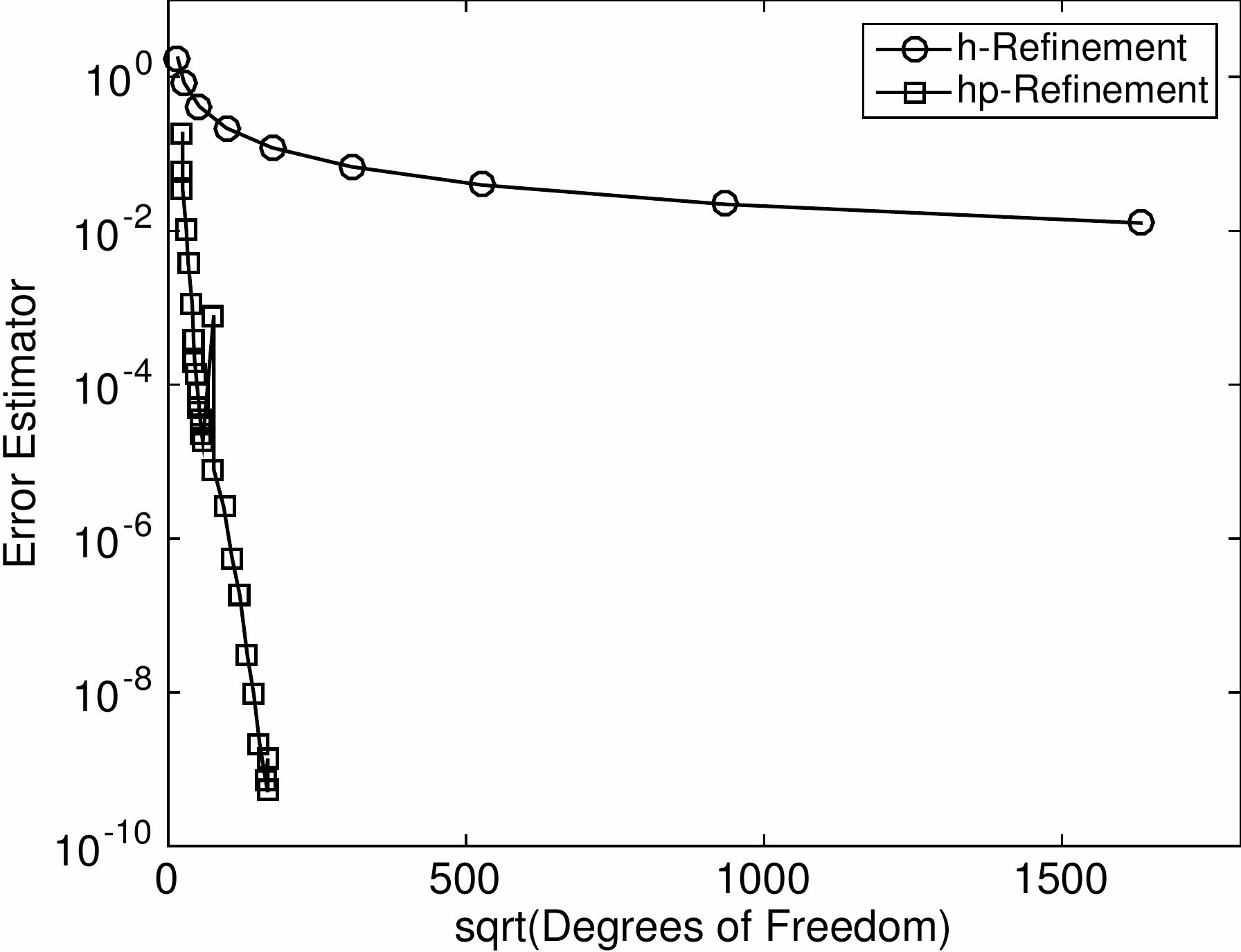}} \\
\multicolumn{2}{c}{(e)}
\end{tabular}
\end{center}
\caption{Bratu Problem. Comparison between $h$-- and $hp$--refinement.
(a)~$\epsilon=1$ (lower solution);
(b) $\epsilon=1$ (upper solution);
(c) $\epsilon=\nicefrac{1}{2}$ (lower solution);
(d) $\epsilon=\nicefrac{1}{2}$ (upper solution);
(e) $\epsilon=\epsilon_c$ (critical solution); }
\label{fig:bratu_convergence}
\end{figure}

\begin{figure}[t!]
	\begin{center}
		\begin{tabular}{cc}
\includegraphics[scale=0.38]{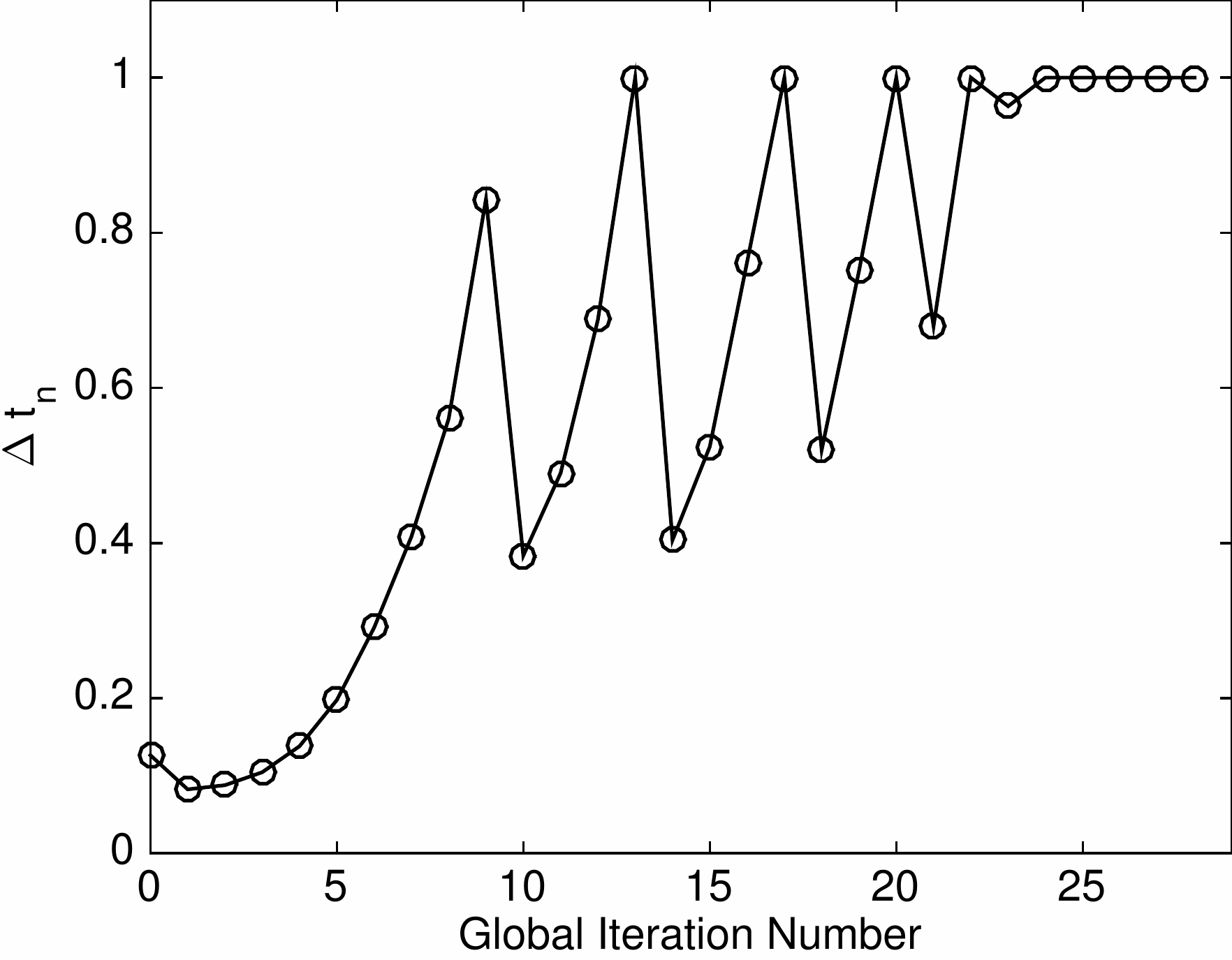} &
\includegraphics[scale=0.38]{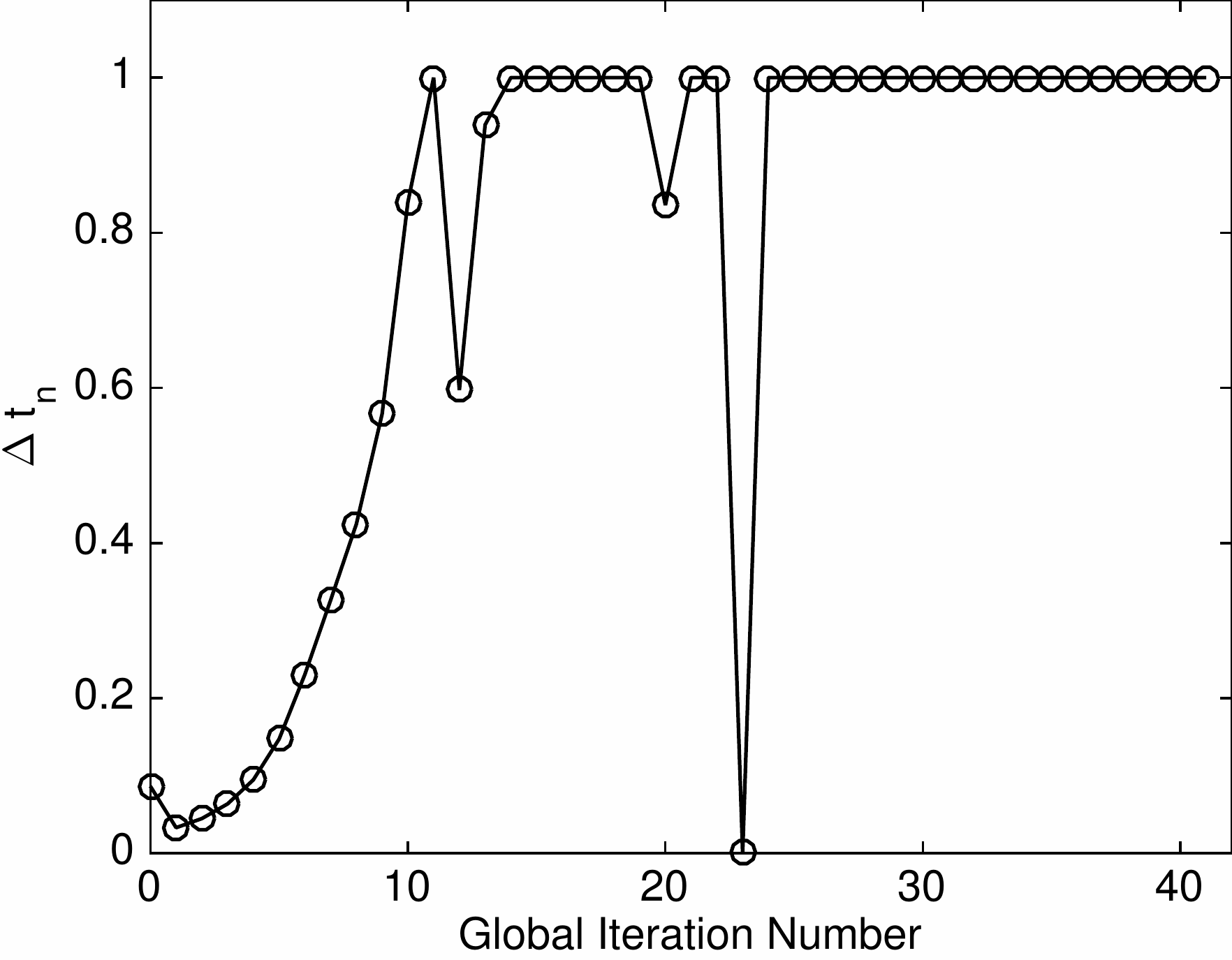} \\
(a) & (b) \\
~ & ~\\
\includegraphics[scale=0.38]{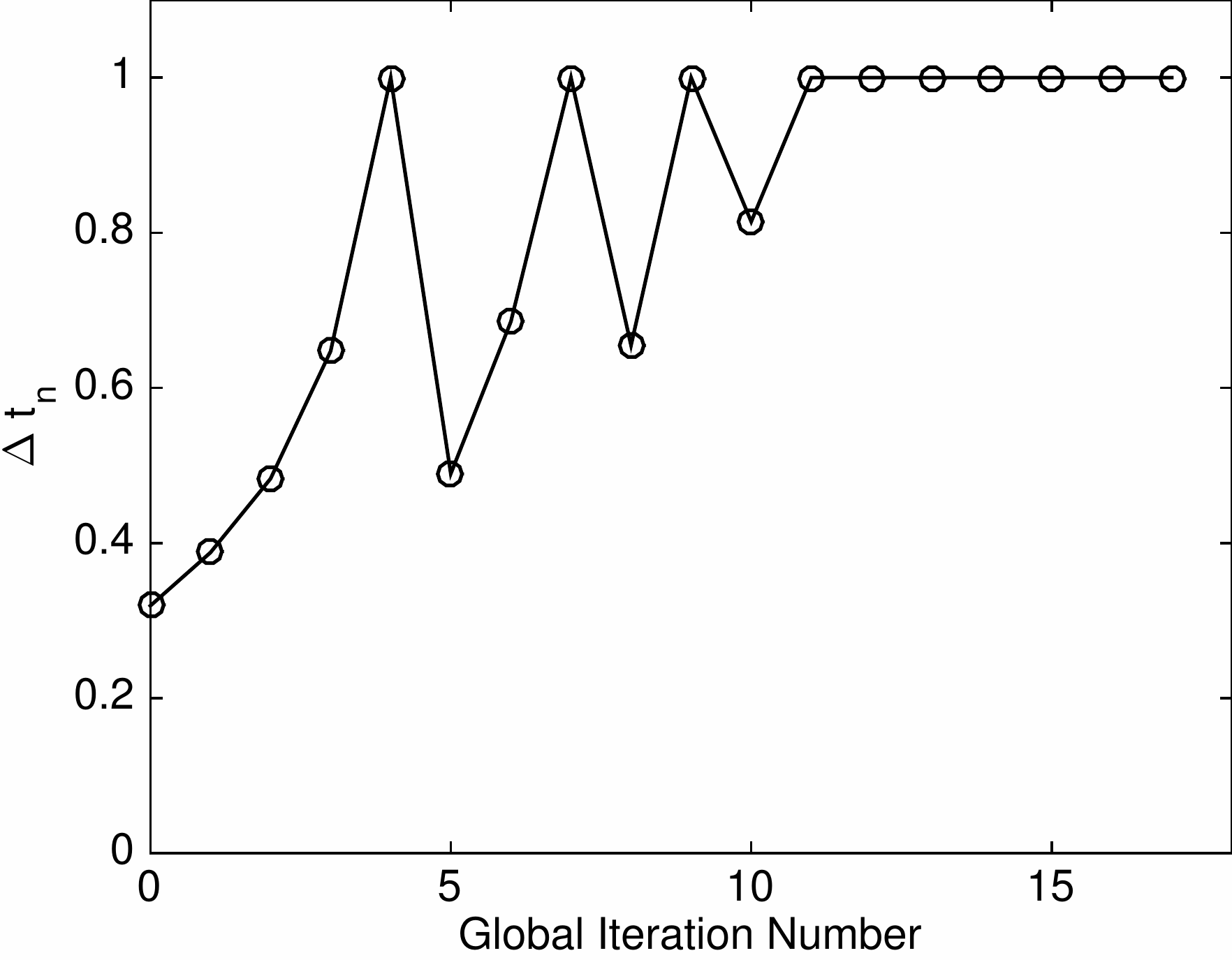} &
\includegraphics[scale=0.38]{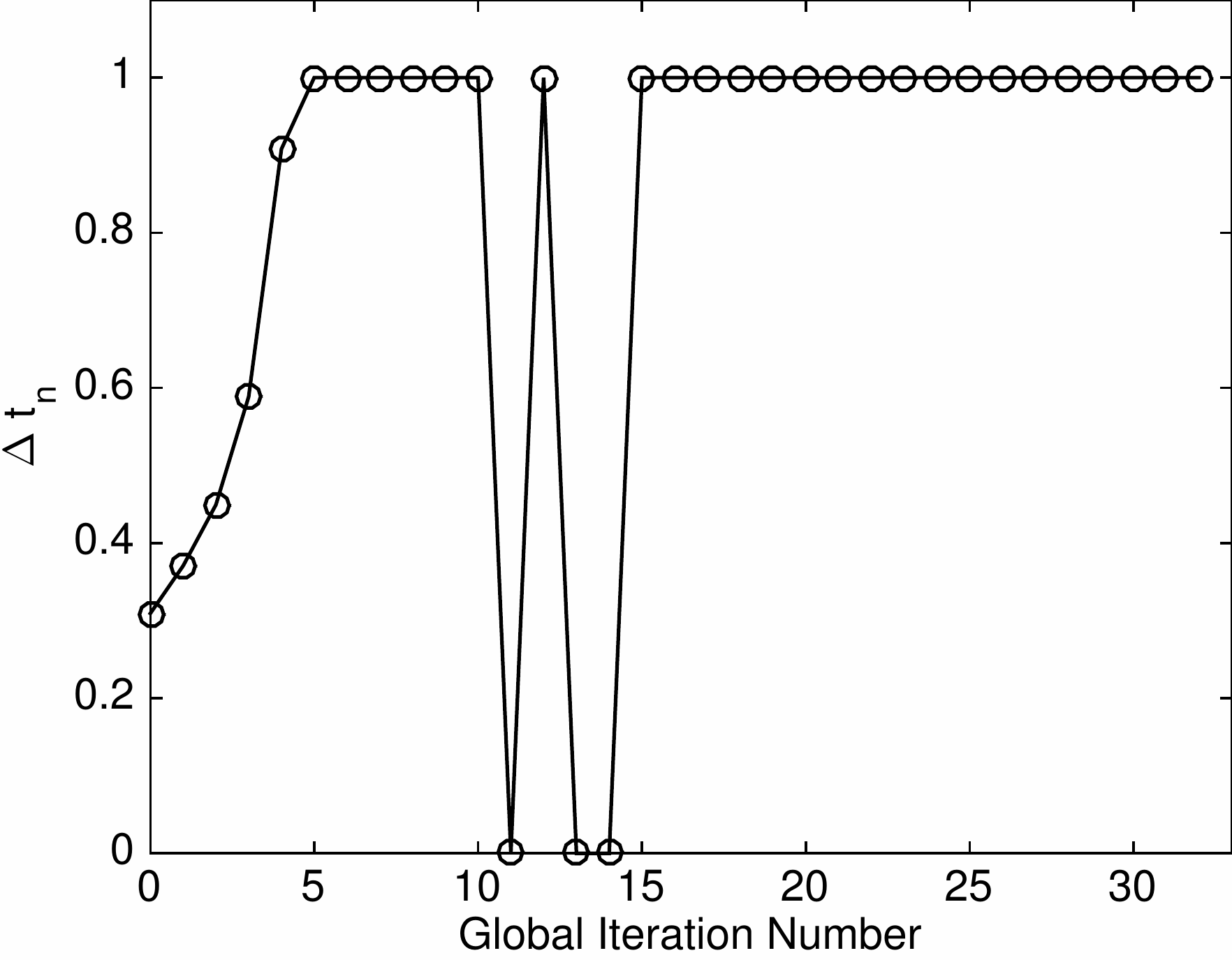} \\
(c) & (d) \\
~ & ~\\
\includegraphics[scale=0.38]{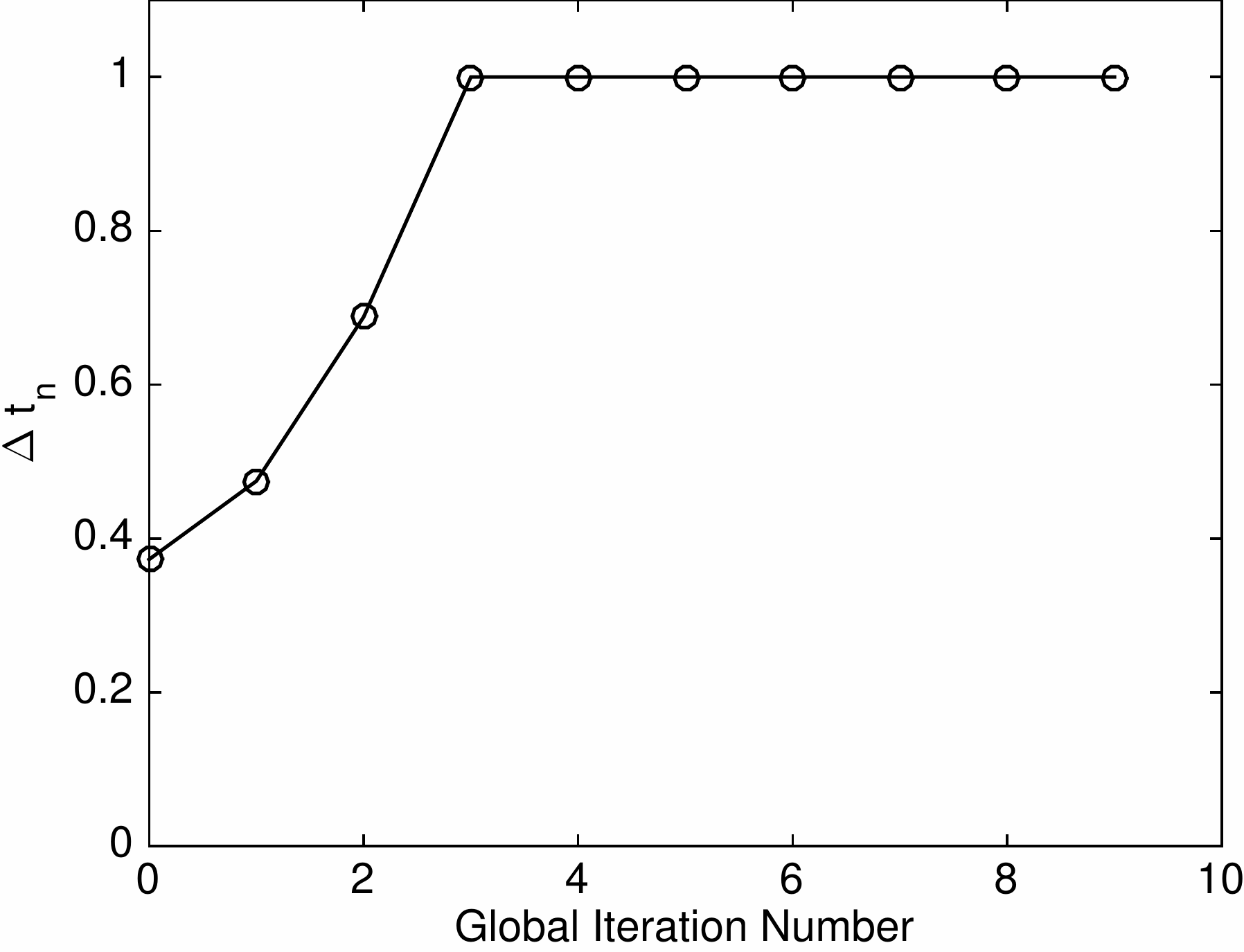} &
\includegraphics[scale=0.38]{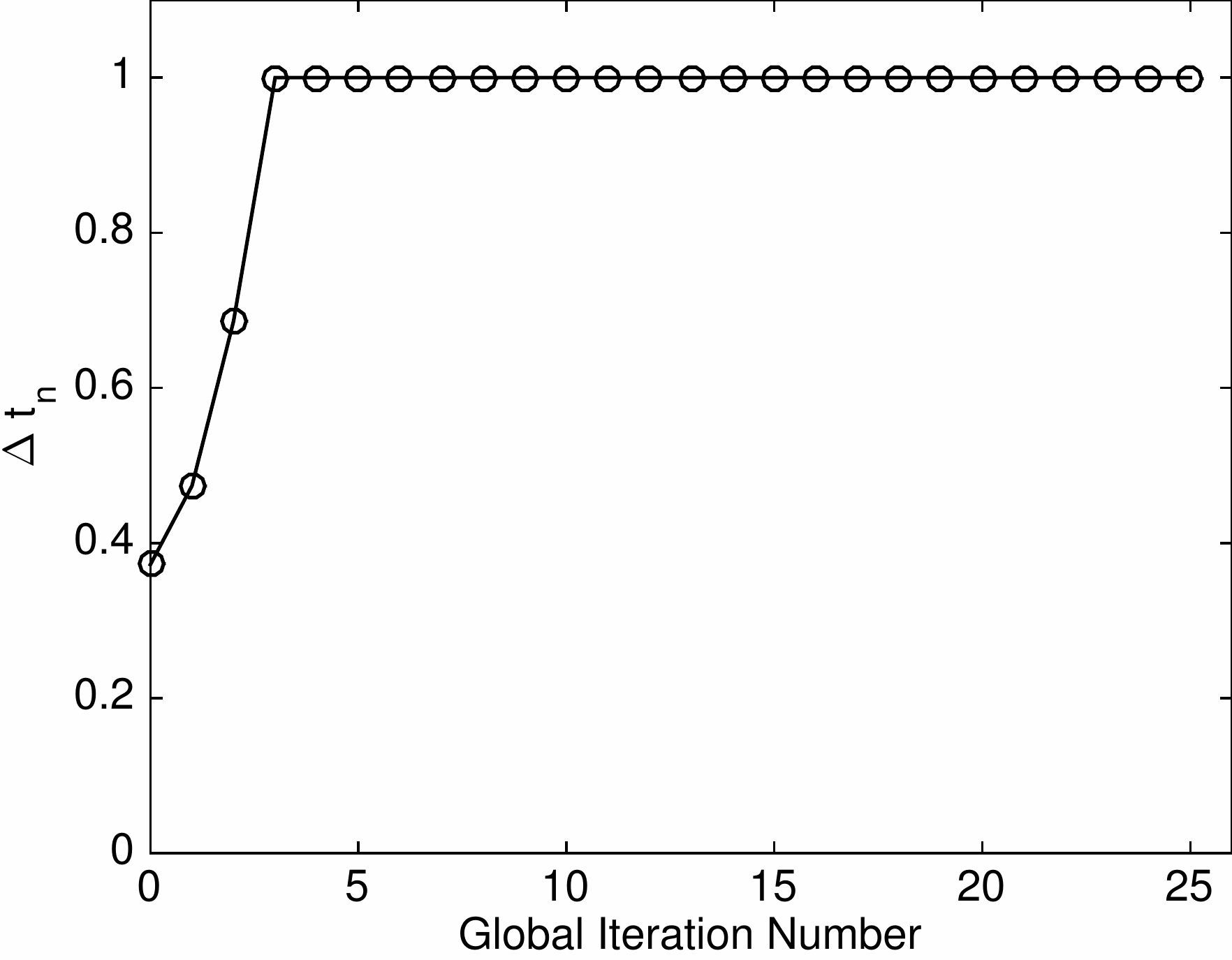} \\
(e) & (f) 
\end{tabular}
\end{center}
\caption{Bratu Problem. Damping parameter $\Delta t_n$.
Left: $h$--refinement; right: $hp$--refinement.
(a) \& (b) $\epsilon=1$ (upper solution);
(c) \& (d) $\epsilon=\nicefrac{1}{2}$ (upper solution);
(e) \& (f) $\epsilon=\nicefrac{1}{2}$ (lower solution).}
\label{fig:bratu_deltat}
\end{figure}

Following \cite{mohsen_2014}, we select the initial guess $ \uDG_{0} \in  \V $
to be the $L^2$--projection of the function $u_0$ onto $\V$, where
$$
u_0 = a \sin(\pi x)\sin(\pi y)
$$
and $a$ is a given amplitude. Noting that the maximum amplitude of the critical solution
computed with $\epsilon = \epsilon_c$ is approximately $1.39$, selecting $a$ to be 
smaller/larger than this value leads to convergence to the so--called
lower/upper solution, respectively. With this in mind we select
$a=2$ when $\epsilon = \epsilon_c$, $a\in\{\nicefrac{1}{10},6\}$ for 
$\epsilon = 1$, and $a\in\{1,4\}$ for $\epsilon=\nicefrac{1}{2}$; in the latter
two cases the smaller value of $a$ is employed for the computation
of the lower solution, while the larger value ensures convergence to
the upper solution.
In Figure~\ref{fig:bratu_slices} we plot a slice of each of the computed
numerical solutions at $y=0.5$, $0\leq x\leq 1$. Here, we observe that the
lower solutions tend to be rather flat in profile, while the upper solutions
have a stronger peak in the middle of the computational domain, cf.,
also, Figure~\ref{fig:bratu_solutions}.

\begin{figure}[t!]
	\begin{center}
	\begin{tabular}{cc}
\includegraphics[scale=0.4]{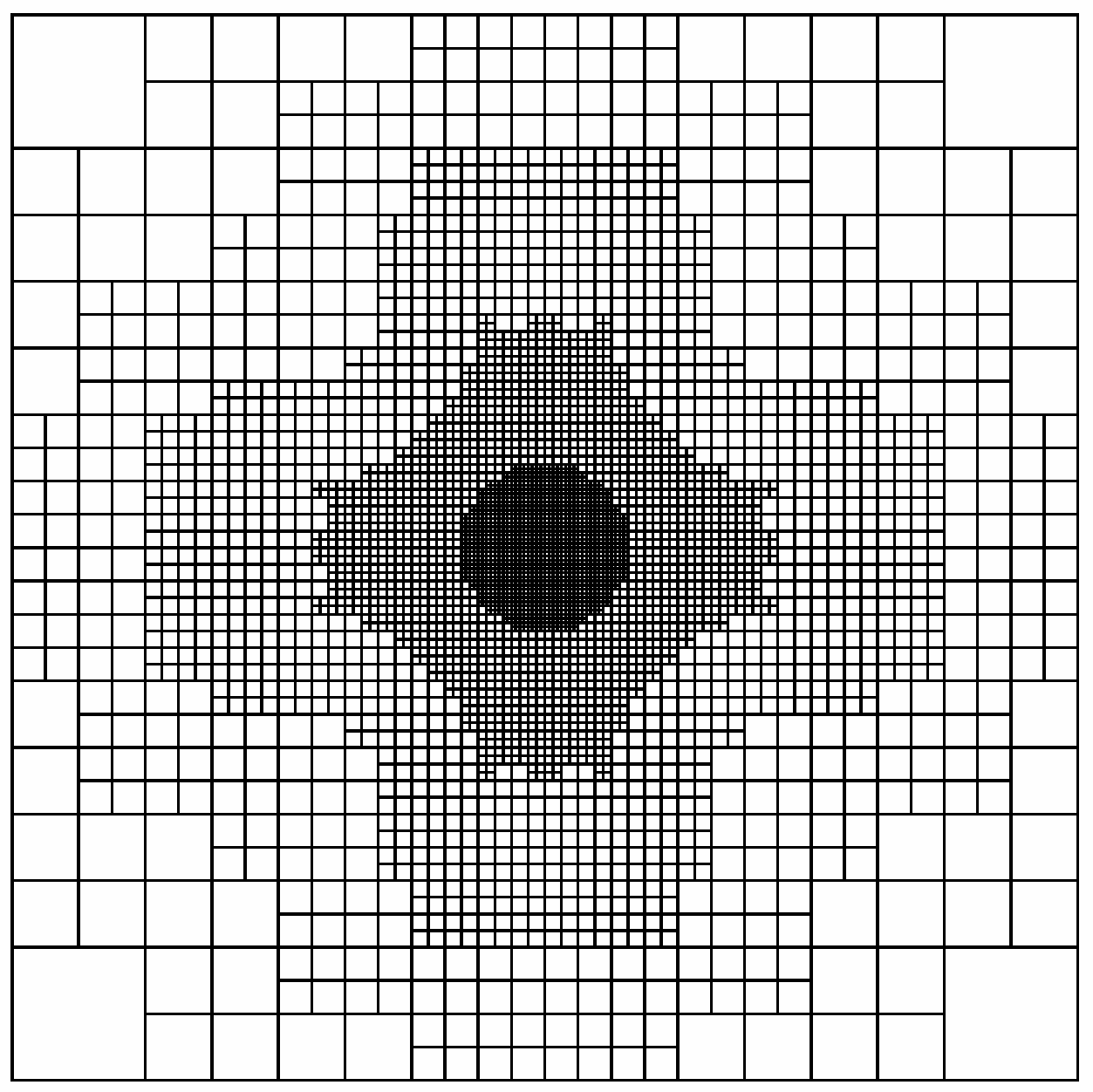} &
\includegraphics[scale=0.39]{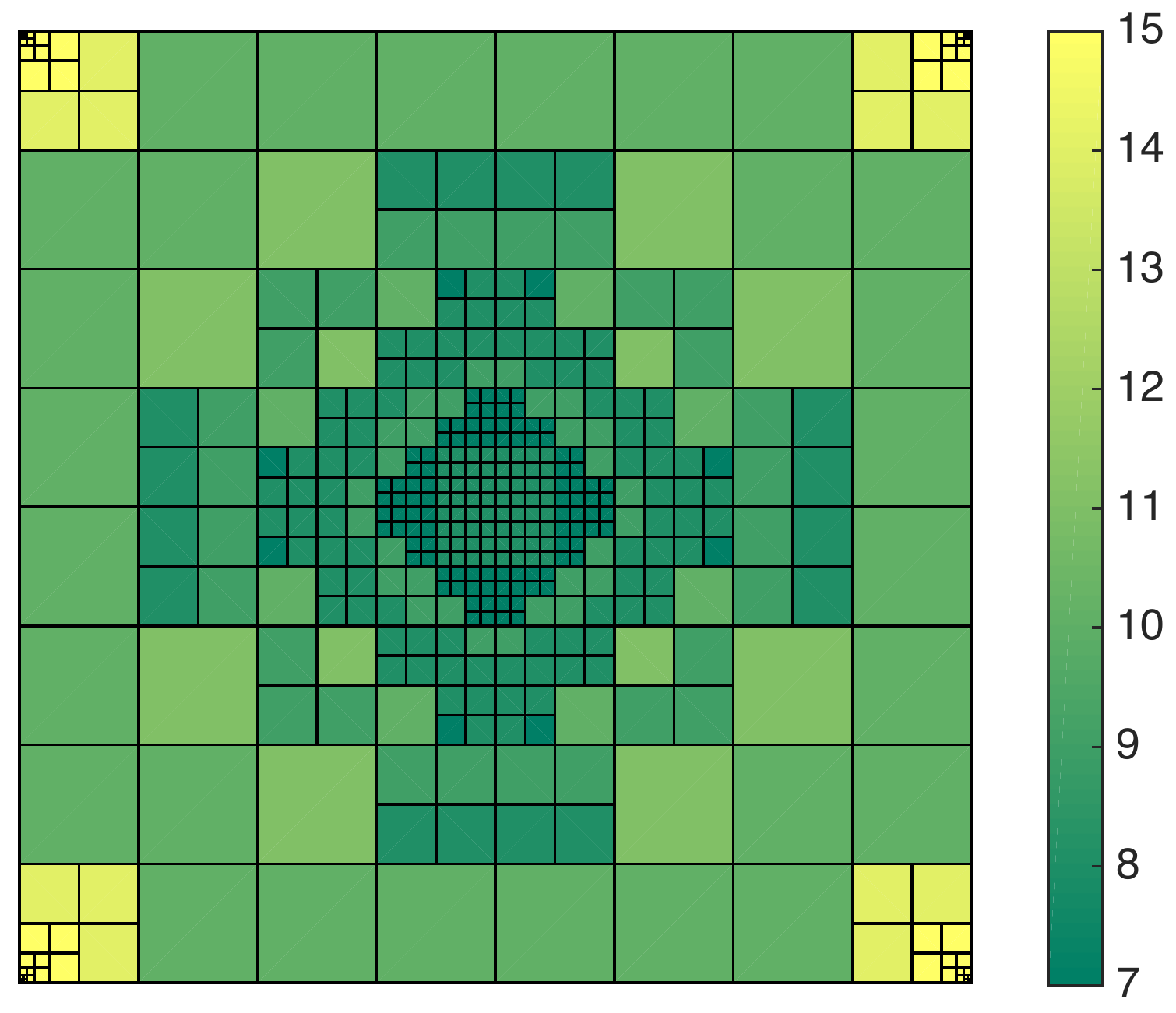} \\
(a) & (b) \\
\includegraphics[scale=0.4]{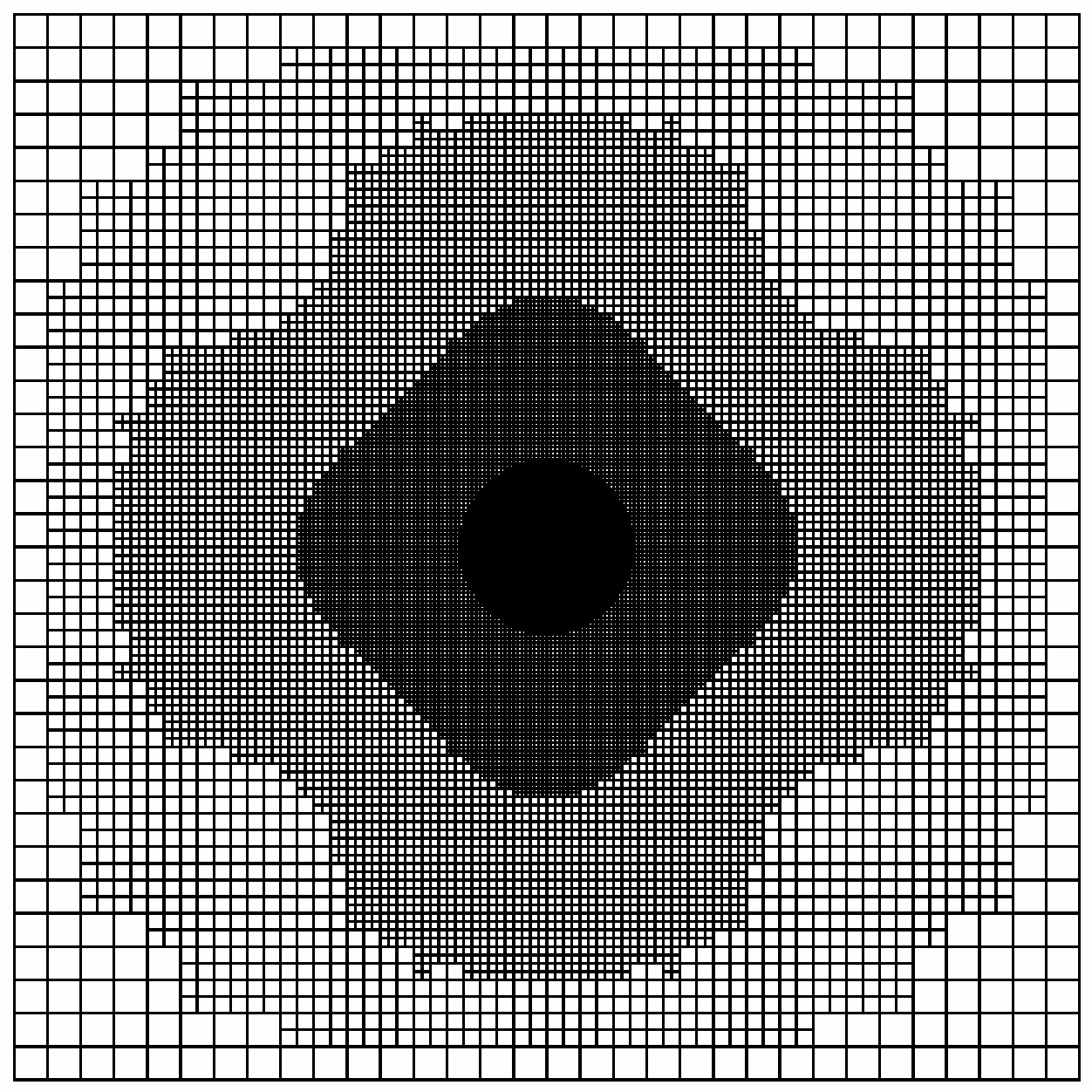} &
\includegraphics[scale=0.39]{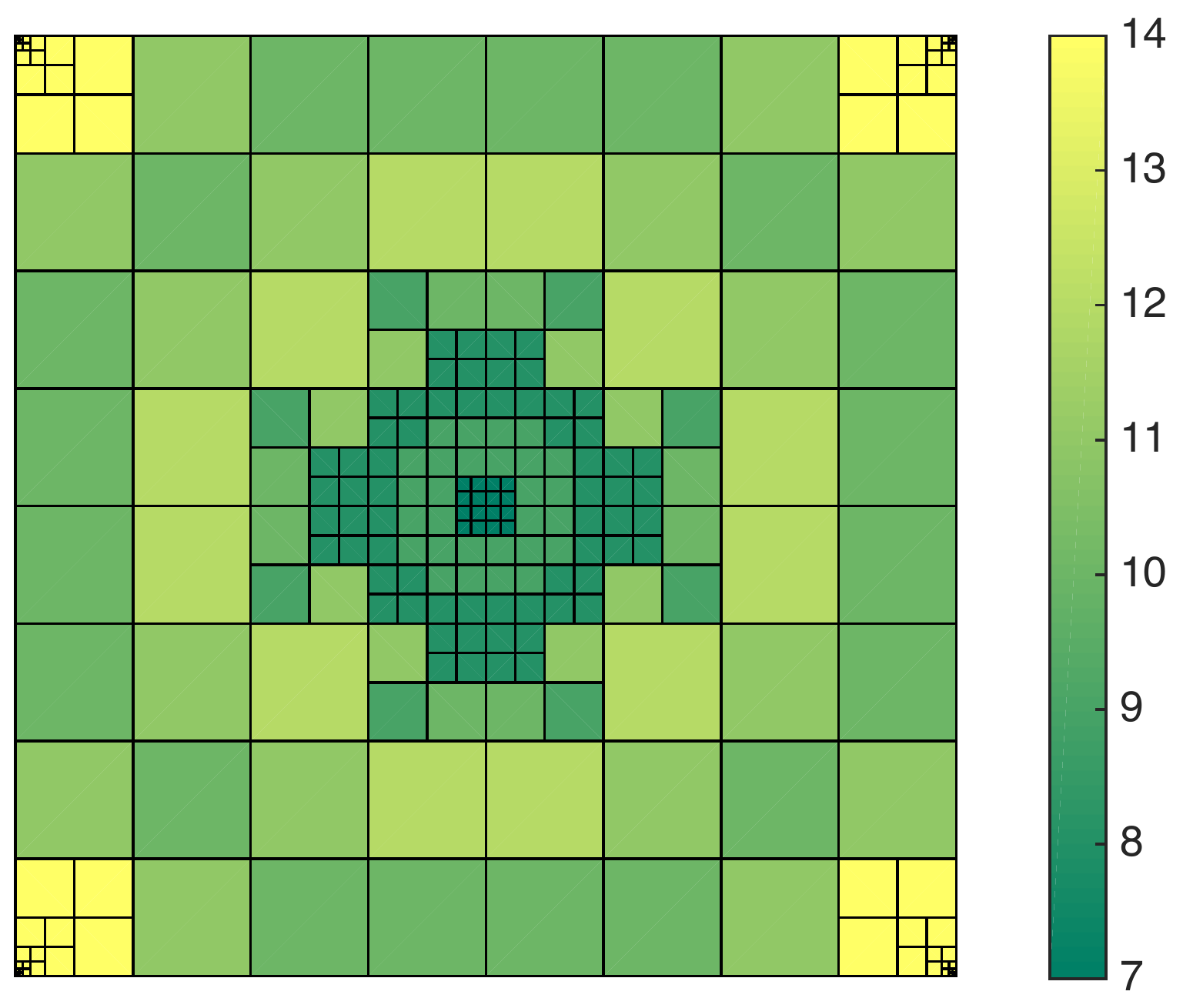} \\
(c) & (d) \\
\includegraphics[scale=0.4]{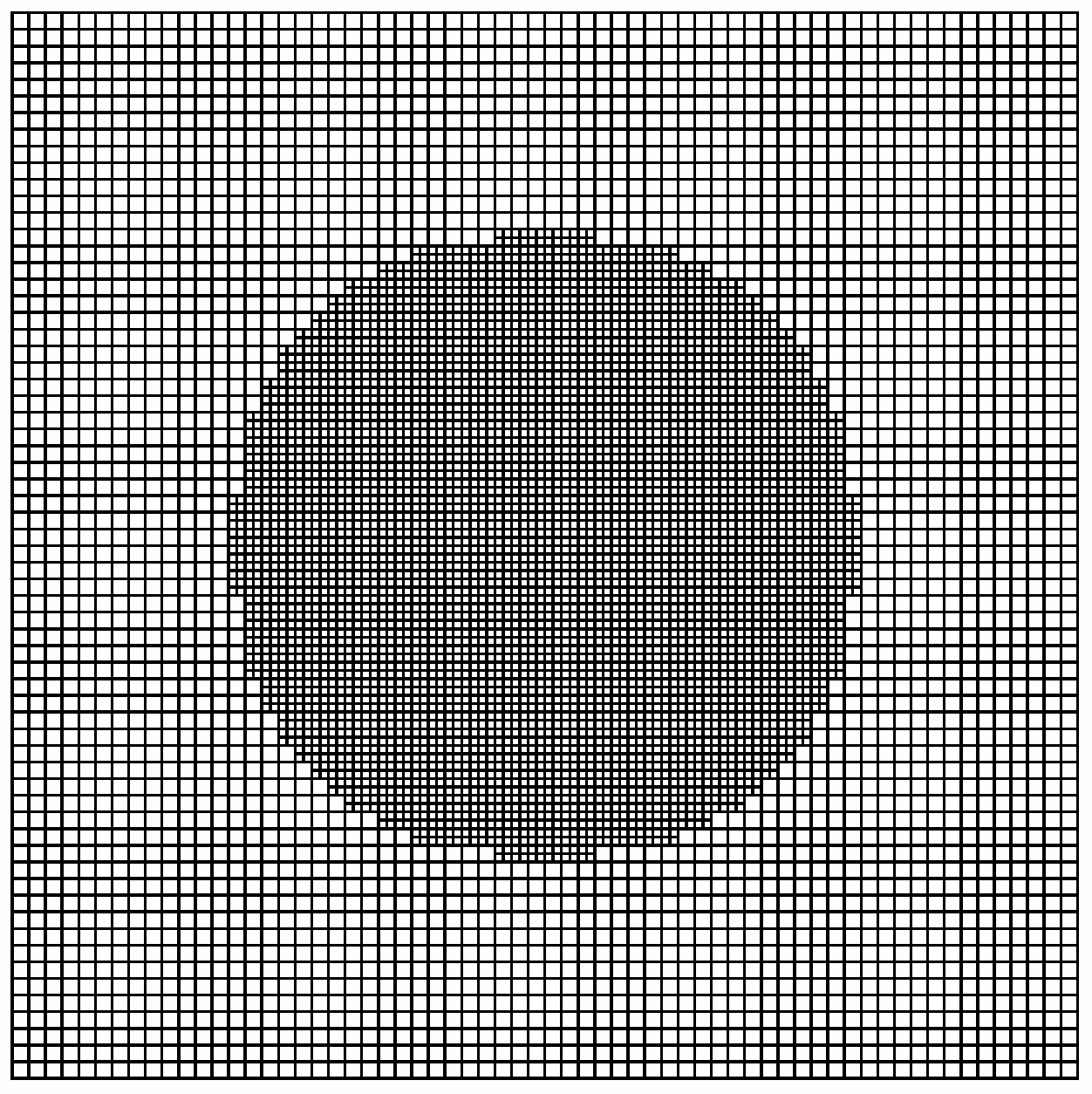} &
\includegraphics[scale=0.39]{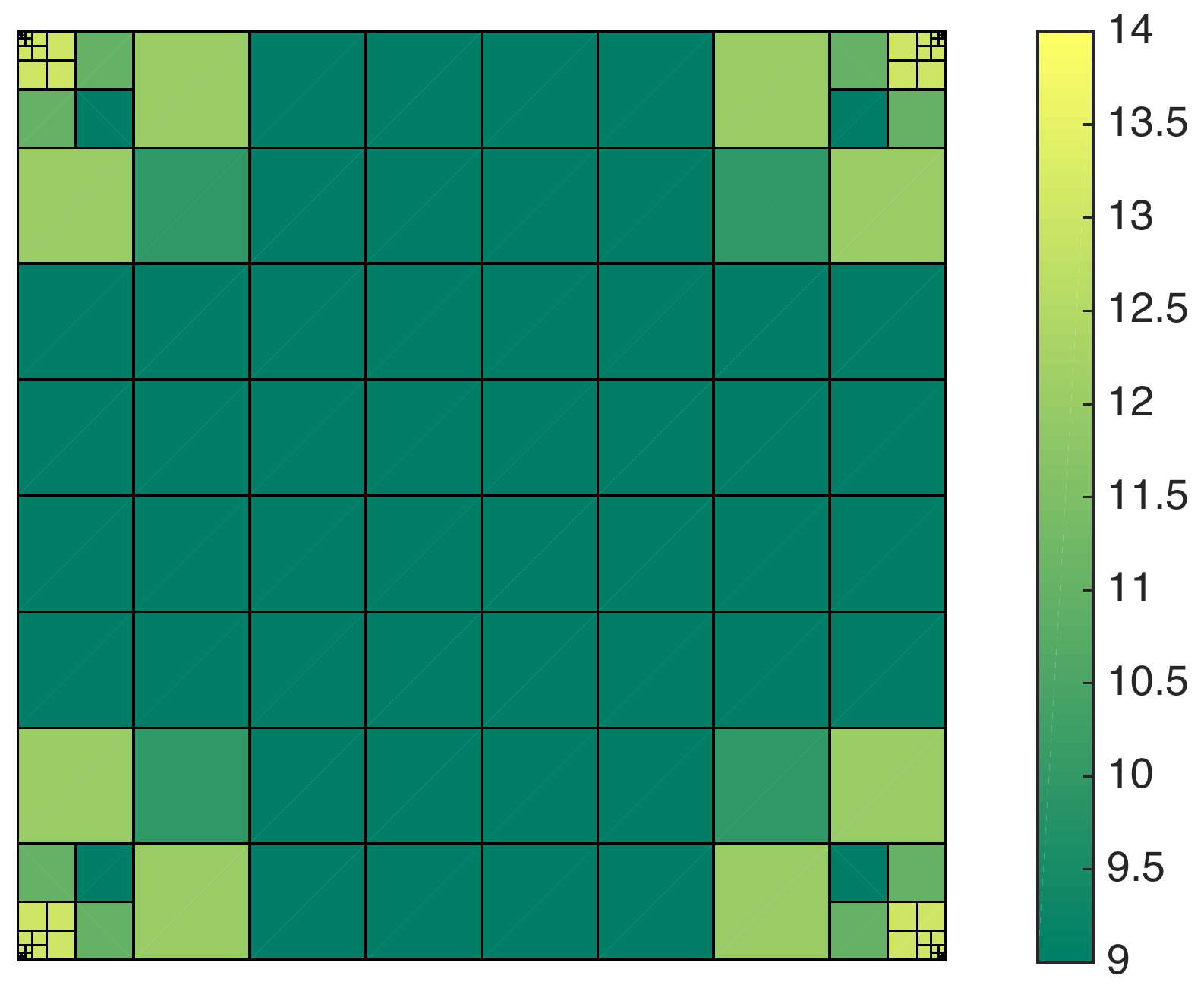} \\
(e) & (f)
\end{tabular}
\end{center}
\caption{Bratu Problem. Computational meshes.
Left: $h$--refinement; right: $hp$--refinement.
(a) \& (b) Upper solution computed with $\epsilon=1$; 
(c) \& (d) Upper solution computed with $\epsilon = 0.5$.
(e) \& (f) Critical solution.}
\label{fig:bratu_meshes}
\end{figure}

In Figure~\ref{fig:bratu_convergence} we demonstrate the performance of the
proposed $hp$--adaptive NDG algorithm, cf. Algorithm~\ref{al:full},
for the computation of the lower and upper solutions when
$\epsilon=1$ and $\epsilon=\nicefrac{1}{2}$, as well as for the numerical
approximation of the critical solution when $\epsilon = \epsilon_c$.
In each case we plot the residual estimator 
${\mathcal E}={\mathcal E}(\uDG_n,\uDG_{n+1},{\bm h},{\bm p})$ versus
the square root of the number of degrees of freedom in the finite element
space $\V$, based on employing both $h$-- and $hp$--refinement.
For each parameter value we observe that the $hp$--refinement algorithm
leads to an exponential decay of the residual estimator ${\mathcal E}$ as the
finite element space $\V$ is adaptively enriched: on a linear-log
plot, the convergence lines are roughly straight. Moreover, we observe the
superiority of $hp$--refinement in comparison with a standard 
$h$--refinement algorithm, in the sense that the former refinement strategy
leads to several orders of magnitude reduction in ${\mathcal E}$, 
for a given number of degrees of freedom, than the corresponding
quantity computed exploiting mesh subdivision only.

In Figure~\ref{fig:bratu_deltat} we plot the size of the Newton damping
$\Delta t_n$ versus the global iteration number. In many of the cases
considered here $\Delta t_n=1$ at all steps; for brevity, these results
have been omitted. For the cases presented in Figure~\ref{fig:bratu_deltat},
we observe that initially the damping parameter slowly increases when we are
far away from the solution; once the damping parameter is close to unity,
the condition 
$$
\delta_{n,\Omega}^2\le \Lambda \sum_{\kappa\in\T}{\eta_{\kappa,n}^2}
$$
in Algorithm~\ref{al:full} becomes fulfilled in which case the
finite element space $\V$ is adaptively enriched. In some cases,
particularly at the early stages of the algorithm, refinement
of $\V$ may then lead to a reduction in $\Delta t_n$, in which
case further Newton steps are required before the next refinement can be undertaken. As 
the iterates approach the solution more closely, the size of the damping parameter typically remains approximately~1.

Finally, in Figure~\ref{fig:bratu_meshes} we show the $h$-- and $hp$--refined
meshes generated for the numerical approximation of the upper solutions
when $\epsilon=1$ and $\epsilon=\nicefrac{1}{2}$, as well as for the
critical solution. Here we observe that when $h$--refinement is employed, 
the mesh is concentrated in the vicinity of the peak in the solution 
located at the centre of the computational domain, cf.
Figures~\ref{fig:bratu_slices} \&~\ref{fig:bratu_solutions}. 
In the $hp$--setting, we observe that while some mesh refinement
has been undertaken in the centre of the domain $\Omega$, the corners
of $\Omega$ have been significantly refined in order to resolve corner singularities typical for elliptic problems. Moreover, $p$--enrichement
has been employed both in these corner regions, as well as in the vicinity
of the peak in the computed solution. The corresponding meshes for the lower 
solutions are largely uniformly refined, due to the flat nature of the
solution; for brevity, these have been omitted.

\end{example}

\begin{figure}[t!]
	\begin{center}
	\begin{tabular}{cc}
\includegraphics[scale=0.12]{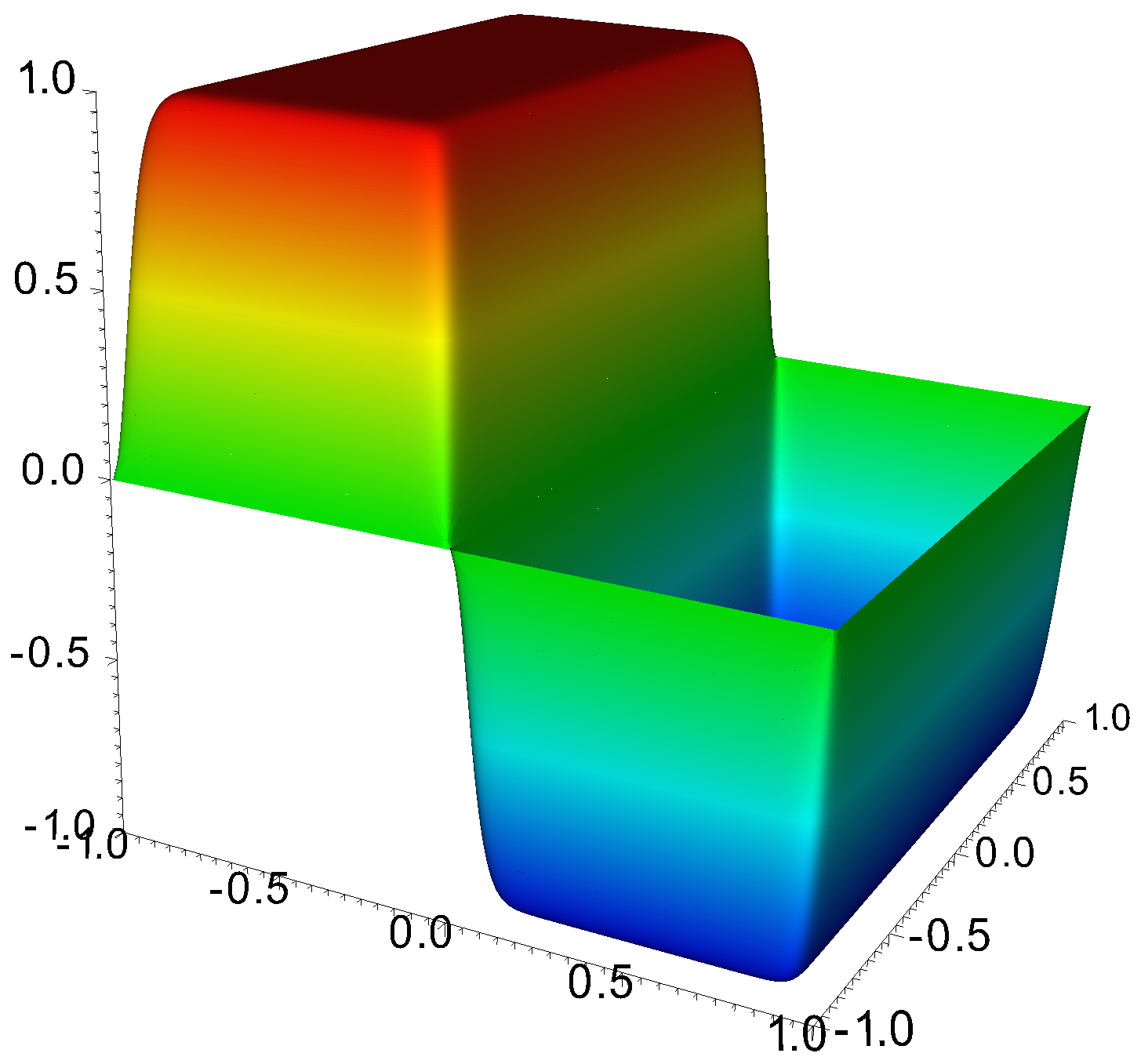} &
\includegraphics[scale=0.12]{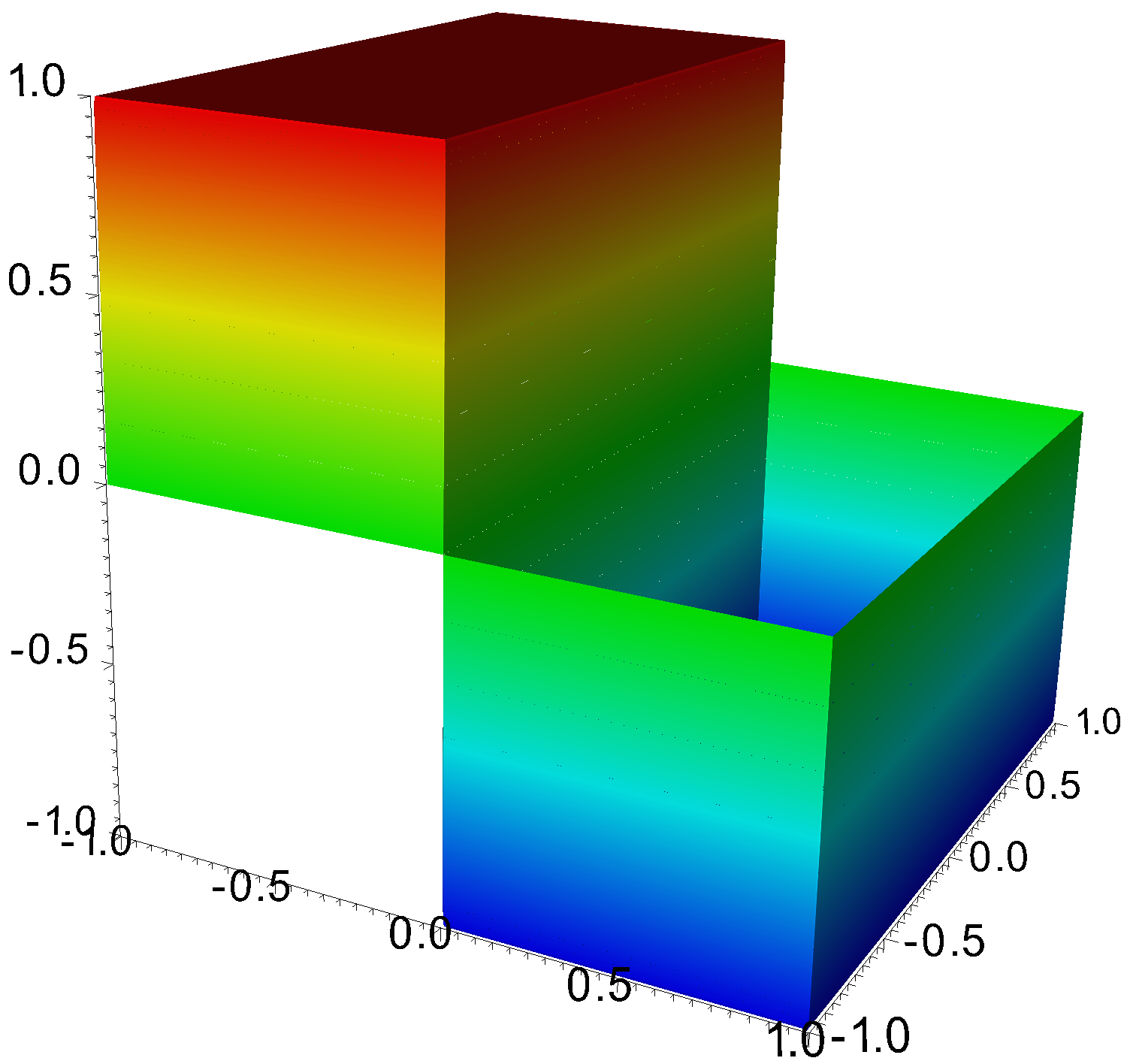} \\
(a) & (b)
\end{tabular}
\end{center}
\caption{Ginzburg-Landau equation. Solution computed with:
(a) $\epsilon=10^{-3}$; (b) $\epsilon = 10^{-6}$.}
\label{fig:ginzburg_solutions}
\end{figure}

\begin{example}
In this example, we consider the 
Ginzburg-Landau equation given by
\[
-\eps\Delta u+u=u(2-u^2) \quad \mbox{ in } (-1,1)^2,
\]
subject to homogeneous Dirichlet boundary conditions on $\partial\Omega$.
Following \cite{AmreinWihler:15}, we first note that $u\equiv 0$ is a 
solution; moreover, any solution $u$ appears in a pairwise fashion
as $-u$. In the absence of boundary conditions, it is clear that $u=\pm 1$
are solutions of the Ginzburg-Landau equation. Thereby, in the presence
of homogeneous Dirichlet boundary conditions, boundary layers will arise
in the vicinity of $\partial\Omega$, whose width will be governed by
the size of the diffusion coefficient $\epsilon$. Here,
we select the initial guess $ \uDG_{0} \in  \V $
to be the $L^2$--projection of the function $u_0(x,y) = -\mbox{sgn}(x)$ onto $\V$,
subject to the enforcement of the boundary conditions. In this case 
the solution to the Ginzburg-Landau equation will possess not only 
boundary layers, but also an internal layer along $x=0$;
in Figure~\ref{fig:ginzburg_solutions} we plot the solution computed
with both $\epsilon=10^{-3}$ and $\epsilon=10^{-6}$.

\begin{figure}[t!]
	\begin{center}
		\begin{tabular}{cc}
\includegraphics[scale=0.38]{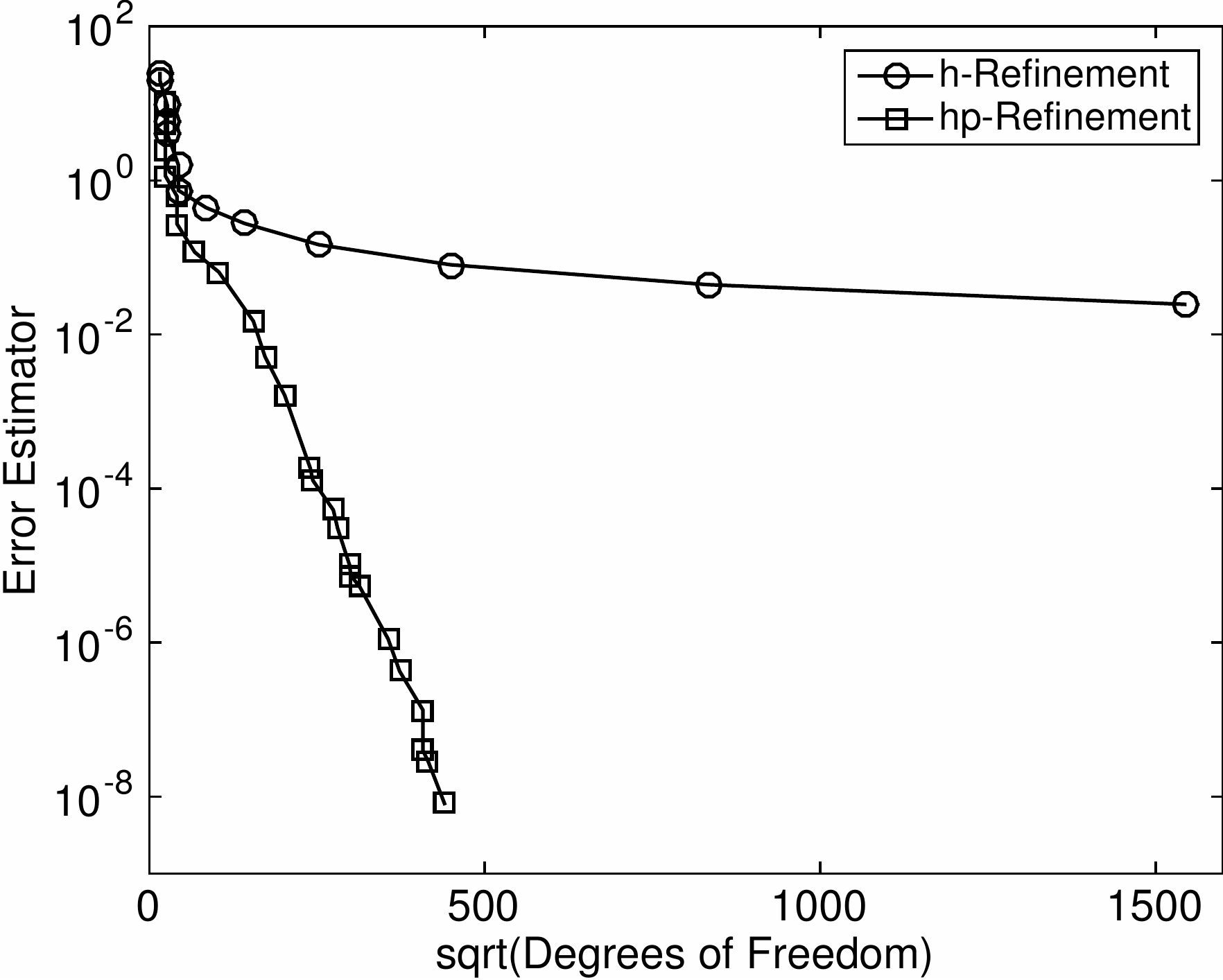} &
\includegraphics[scale=0.38]{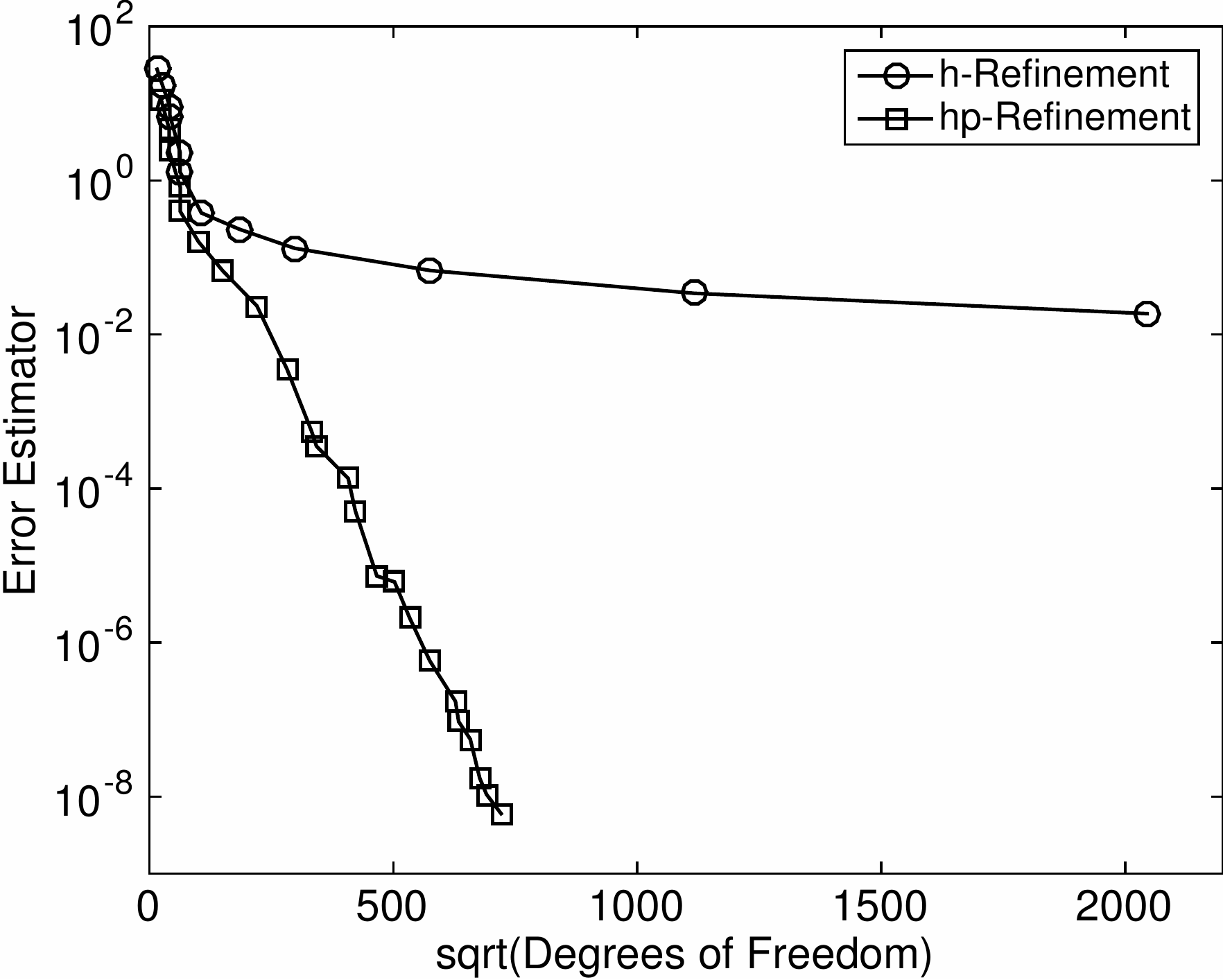} \\
(a) & (b) \\
~ & ~ \\
\includegraphics[scale=0.38]{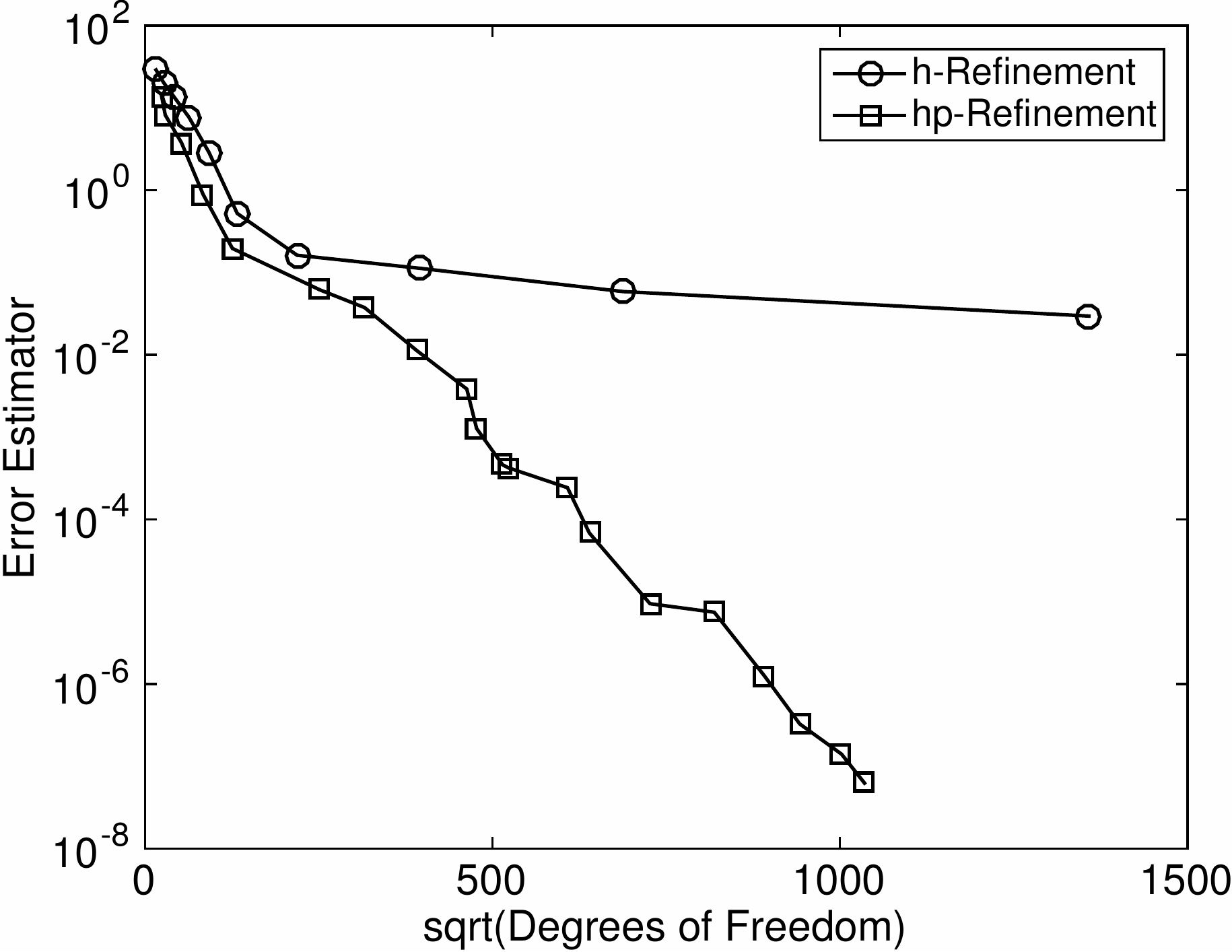} &
\includegraphics[scale=0.38]{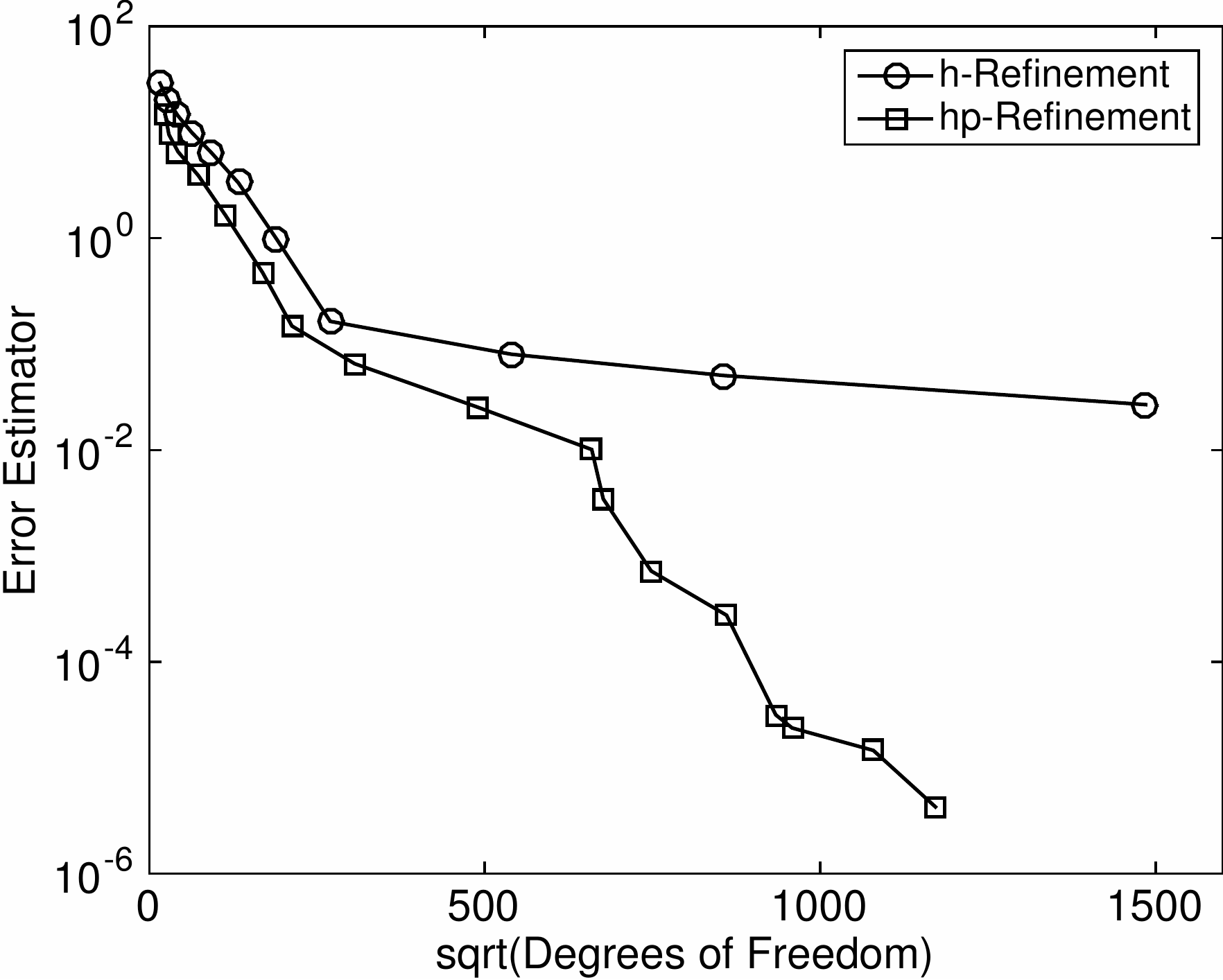} \\
(c) & (d) 
\end{tabular}
\end{center}
\caption{Ginzburg-Landau equation. Comparison between $h$-- and $hp$--refinement.
(a)~$\epsilon=10^{-3}$;
(b) $\epsilon=10^{-4}$;
(c) $\epsilon=10^{-5}$;
(d) $\epsilon=10^{-6}$.}
\label{fig:ginzburg_convergence}
\end{figure}

In Figure~\ref{fig:ginzburg_convergence} we demonstrate the performance of the
proposed $hp$--adaptive NDG algorithm, cf. Algorithm~\ref{al:full},
for the computation of the solution to the Ginzburg-Landau equation when
$\epsilon=10^{-3},~10^{-4},~10^{-5},~10^{-6}$.
In each case we plot the residual estimator 
${\mathcal E}$ versus
the square root of the number of degrees of freedom in the finite element
space $\V$, based on employing both $h$-- and $hp$--refinement.
For each value of $\epsilon$ we again observe that the $hp$--refinement algorithm
leads to an exponential decay of the residual estimator ${\mathcal E}$ as the
finite element space $\V$ is adaptively enriched. Moreover, we again
observe the superiority of exploiting $hp$--refinement in comparison with a standard 
$h$--refinement algorithm, in the sense that the former refinement strategy
leads to several orders of magnitude reduction in ${\mathcal E}$, 
for a given number of degrees of freedom, than the corresponding
quantity computed using $h$--refinement only. Furthermore, we note 
that as $\epsilon$ is reduced, additional $h$--enrichment of the computational
mesh is required before $p$--refinement is employed. Indeed, for $\epsilon=10^{-6}$
we observe that there is an initial transient, before the $hp$--version
convergence line becomes straight and exponential convergence is observed.

In Figure~\ref{fig:ginzburg_deltat_eps_m3} we plot
$\Delta t_n$ versus the global iteration number for $\epsilon=10^{-3}$;
for the other values of $\epsilon$ considered here, the damping parameter 
was close to one on all of the meshes considered. As in the previous
example, we again see an initial increase in $\Delta t_n$ as the
adaptive Newton algorithm proceeds, before the underlying mesh
is adaptively refined. Again, in the early stages of the algorithm,
enrichment of $\V$ may lead to some additional damping, before 
$\Delta t_n$ tends to one.

Finally, in Figure~\ref{fig:ginzburg_meshes} we plot the 
corresponding $h$-- and $hp$--meshes generated for 
$\epsilon=10^{-3}$ and $\epsilon=10^{-6}$. Here, we clearly observe
that the boundary and internal layers present in the analytical
solution are refined by our adaptive mesh adaptation strategy; in particular, we emphasise that the NDG iterates converge to a solution which features the same topology as the initial guess, and, hence, does not switch between various attractors (corresponding to different solutions; see, e.g, \cite{AmreinWihler:15}). In the $hp$--setting, we see that once the $h$--mesh has been sufficiently refined, then $p$--enrichment is employed.

\begin{figure}[t!]
	\begin{center}
		\begin{tabular}{cc}
\includegraphics[scale=0.38]{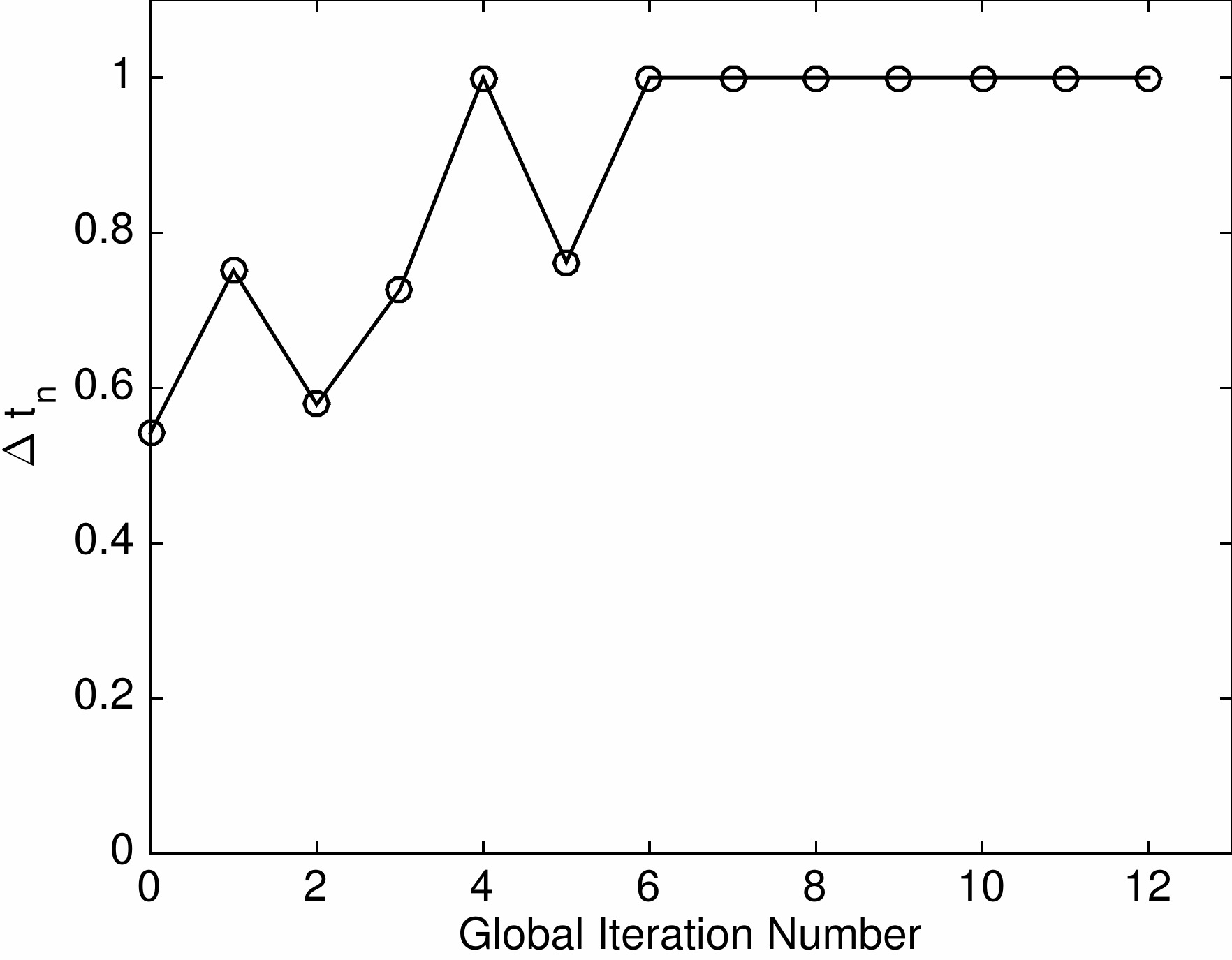} &
\includegraphics[scale=0.38]{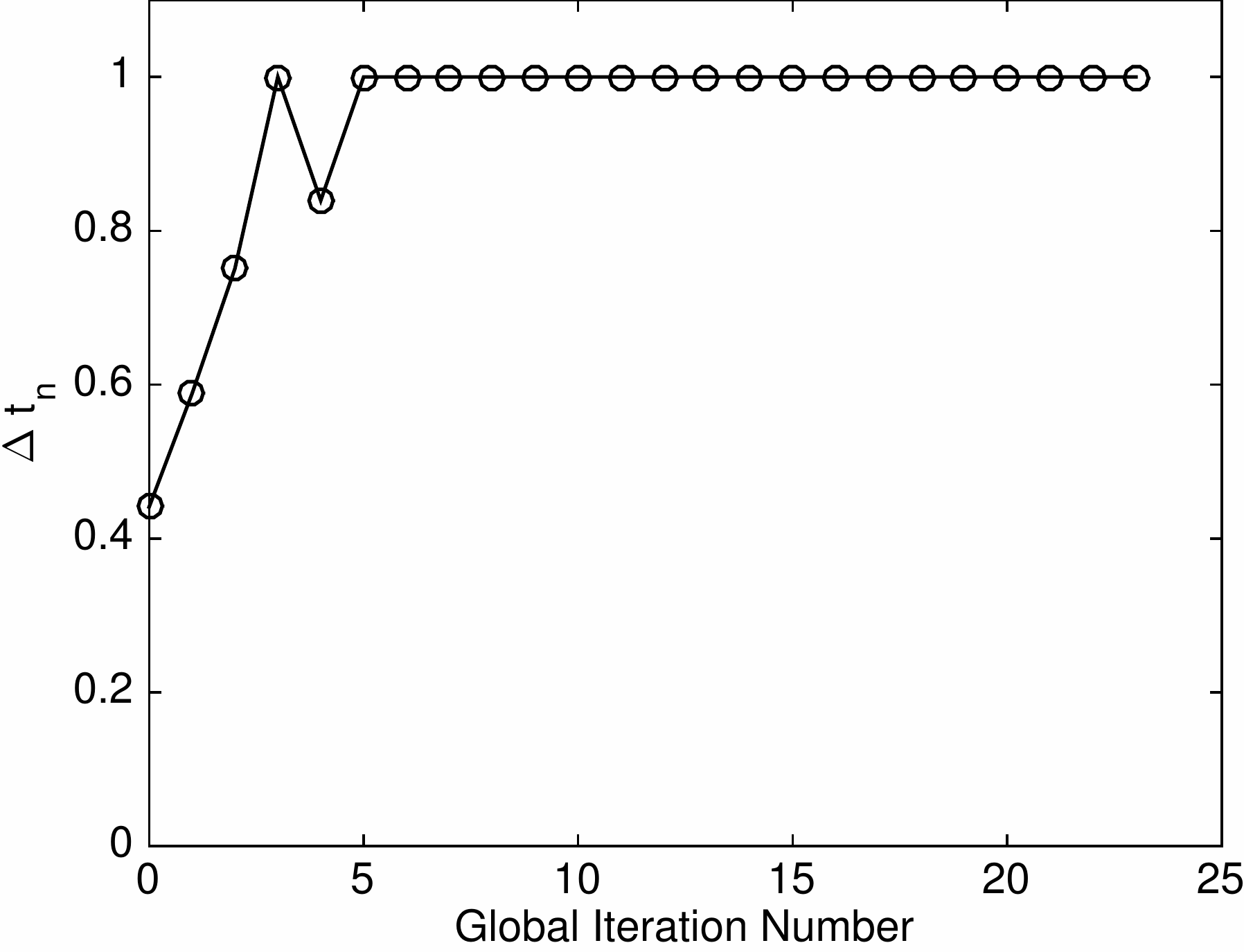} \\
(a) & (b)
\end{tabular}
\end{center}
\caption{Ginzburg-Landau equation. Damping parameter $\Delta t_n$ for $\epsilon=10^{-3}$.
(a) $h$--refinement;
(b) $hp$--refinement.}
\label{fig:ginzburg_deltat_eps_m3}
\end{figure}

\begin{figure}[t!]
	\begin{center}
	\begin{tabular}{cc}
\includegraphics[scale=0.4]{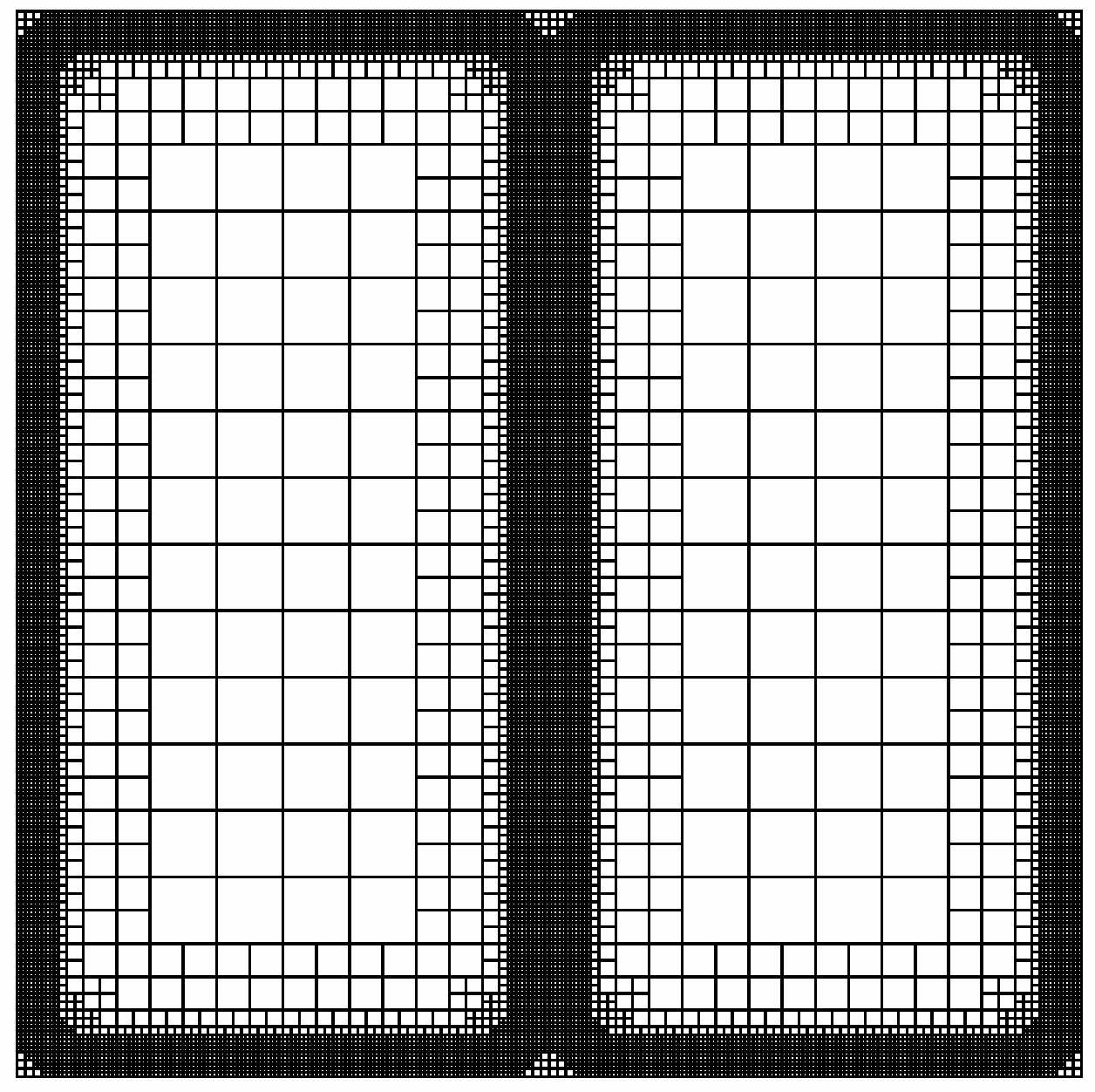} &
\includegraphics[scale=0.39]{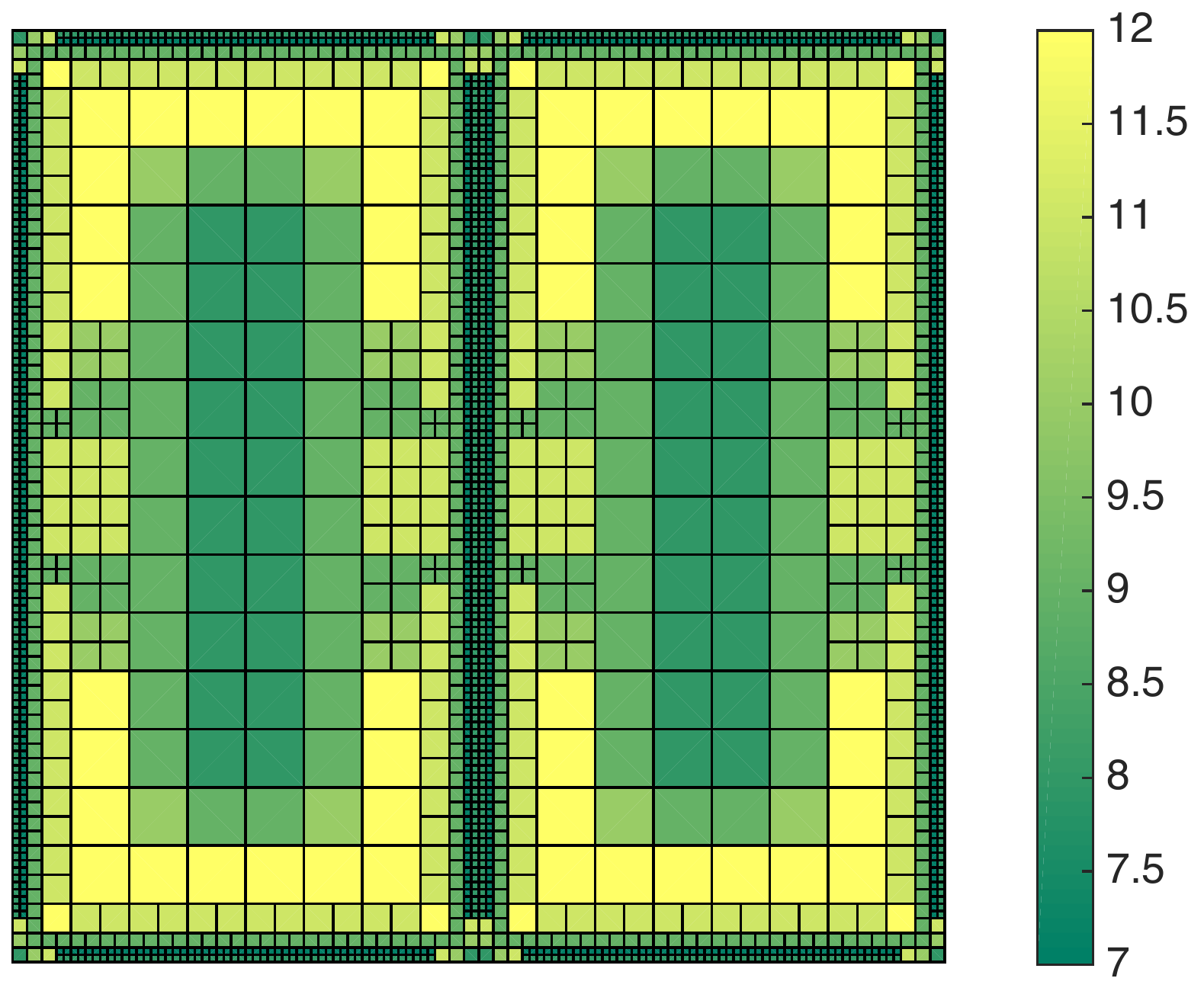} \\
(a) & (b) \\
\includegraphics[scale=0.4]{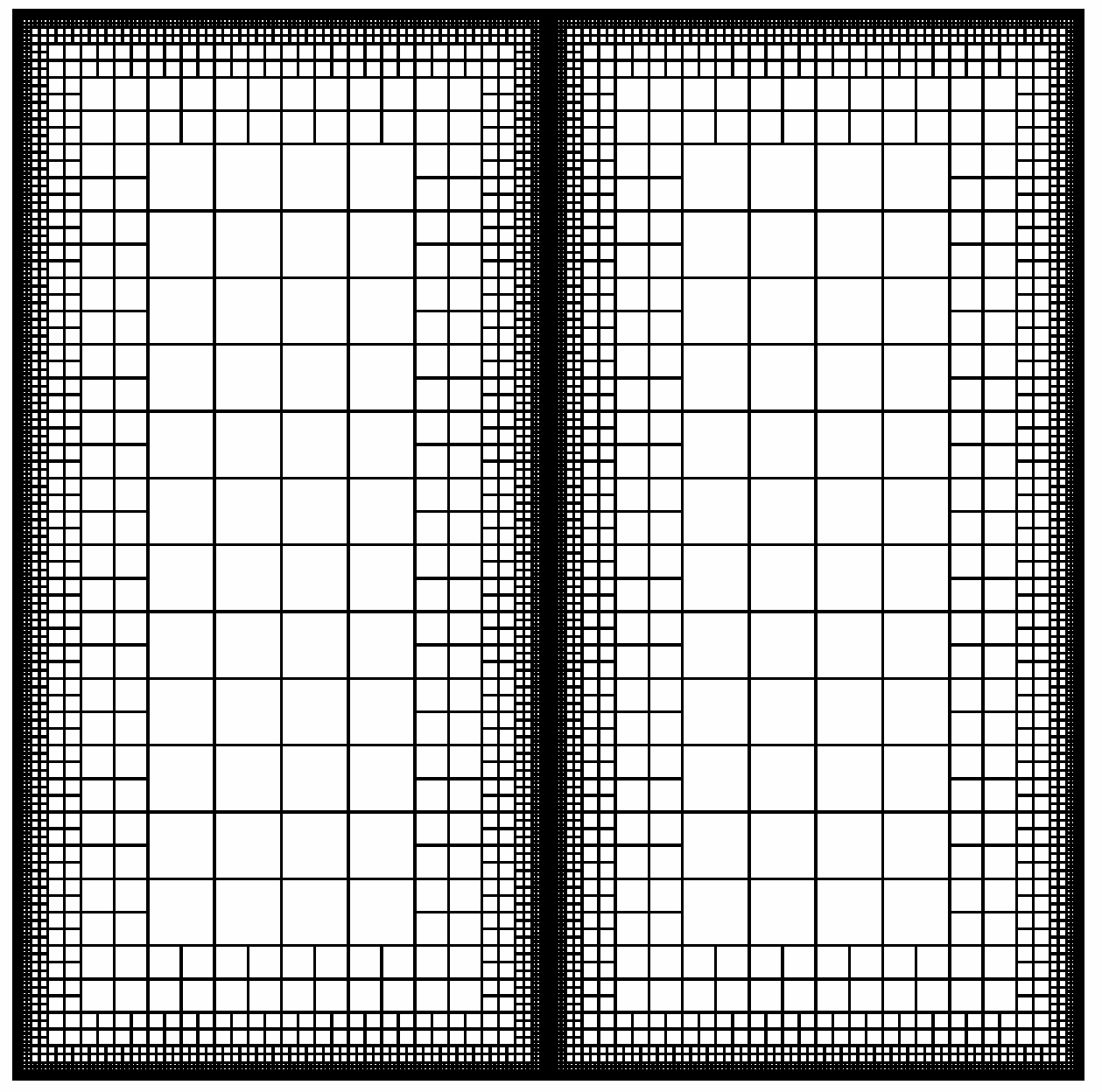} &
\includegraphics[scale=0.39]{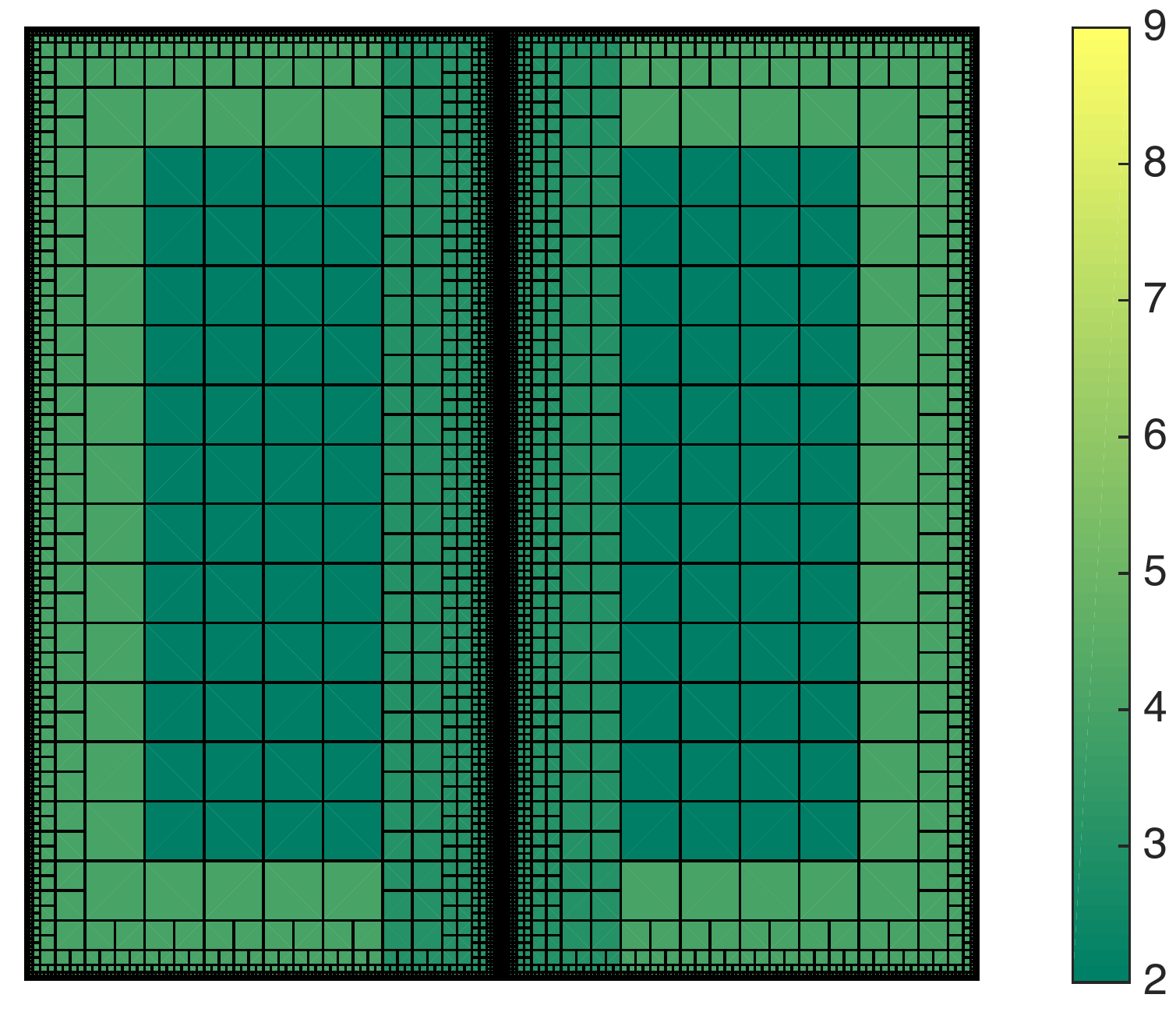} \\
(c) & (d)
\end{tabular}
\end{center}
\caption{Ginzburg-Landau equation. Computational meshes.
Left: $h$--refinement; right: $hp$--refinement.
(a) \& (b) $\epsilon=10^{-3}$; 
(c) \& (d) $\epsilon = 10^{-6}$.}
\label{fig:ginzburg_meshes}
\end{figure}

\end{example}

%
\section{Concluding remarks}
\label{sec:conlusions}

In this article we have introduced the $hp$--version of the NDG scheme for the
numerical approximation of second-order, singularly perturbed, semilinear 
elliptic boundary value problems. Here, the general approach is based on first
linearising the underlying PDE problem on a continuous level, followed by
subsequent discretisation of the resulting sequence of linear PDEs. For this latter
task, in the current article we have exploited the $hp$--version of the 
interior penalty DG method. Furthermore, we have derived an $\epsilon$-robust
{\em a posteriori} bound which takes into account both the linearisation
and discretisation errors. On the basis of this residual estimate, we have
designed and implemented an $hp$--adaptive refinement algorithm
which automatically controls both of these sources of error; the
practical performance of this strategy has been studied for a series
of numerical test problems. Future work will be devoted to the
extension of this technique to more general nonlinear PDE problems,
as well as to problems in three dimensions.


\bibliographystyle{amsplain}
\bibliography{myrefs}

\providecommand{\bysame}{\leavevmode\hbox to3em{\hrulefill}\thinspace}
\providecommand{\MR}{\relax\ifhmode\unskip\space\fi MR }
\providecommand{\MRhref}[2]{%
  \href{http://www.ams.org/mathscinet-getitem?mr=#1}{#2}
}
\providecommand{\href}[2]{#2}
\begin{thebibliography}{10}

\bibitem{MUMPS:2}
P.R. Amestoy, I.S. Duff, J.~Koster, and J.-Y. L'Excellent, \emph{A fully
  asynchronous multifrontal solver using distributed dynamic scheduling}, SIAM
  J. Mat. Anal. Appl. \textbf{23} (2001), no.~1, 15--41.

\bibitem{MUMPS:1}
P.R. Amestoy, I.S. Duff, and J.-Y. L'Excellent, \emph{Multifrontal parallel
  distributed symmetricand unsymmetric solvers}, Comput. Methods Appl. Mech.
  Eng. \textbf{184} (2000), 501--520.

\bibitem{MUMPS:3}
P.R. Amestoy, A.~Guermouche, J.-Y. L'Excellent, and S.~Pralet, \emph{Hybrid
  scheduling for the parallel solution of linear systems}, Parallel Computing
  \textbf{32} (2006), no.~2, 136--156.

\bibitem{AmreinWihler:14}
M.~Amrein and T.P. Wihler, \emph{{An adaptive {Newton}-method based on a
  dynamical systems approach}}, Commun. Nonlinear Sci. Numer. Simul.
  \textbf{19} (2014), no.~9, 2958--2973.

\bibitem{AmreinWihler:15}
\bysame, \emph{Fully adaptive {N}ewton-{G}alerkin methods for semilinear
  elliptic partial differential equations}, SIAM J. Sci. Comput. \textbf{37}
  (2015), no.~4, A1637--A1657.

\bibitem{ArnoldBrezziCockburnMarini:01}
D.N. Arnold, F.~Brezzi, B.~Cockburn, and L.D. Marini, \emph{Unified analysis of
  discontinuous {G}alerkin methods for elliptic problems}, SIAM J. Numer. Anal.
  \textbf{39} (2001), 1749--1779.

\bibitem{aschermattheijhrussell}
U.M. Ascher, M.M. Mattheij, and R.D. Russell, \emph{Numerical solution of
  boundary value problems for ordinary differential equation}, SIAM,
  Philadelphia, PA, 1995.

\bibitem{BarlesBurdeau:95}
G.~Barles and J.~Burdeau, \emph{{The {Dirichlet} problem for semilinear
  second-order degenerate elliptic equations and applications to stochastic
  exit time control problems}}, Comm. Partial Differential Equations
  \textbf{20} (1995), no.~1-2, 129--178.

\bibitem{BaroneEspositoMageeScott:71}
A.~Barone, F.~Esposito, C.J. Magee, and A.C. Scott, \emph{Theory and
  applications of the {S}ine-{G}ordon equation}, Riv. Nuovo Cim. \textbf{1}
  (1971), 227--267.

\bibitem{BerestyckiLions:83}
H.~Berestycki and P.-L. Lions, \emph{Nonlinear scalar field equations. {I}.
  {E}xistence of a ground state}, Arch. Rational Mech. Anal. \textbf{82}
  (1983), no.~4, 313--345.

\bibitem{BernardiDakroubMansourSayah:15}
C.~Bernardi, J.~Dakroub, G.~Mansour, and T.~Sayah, \emph{{A posteriori analysis
  of iterative algorithms for a nonlinear problem}}, J. Sci. Comput.
  \textbf{65} (2015), no.~2, 672--697.

\bibitem{BorisyukErmentroutFriedmanTerman:05}
A.~Borisyuk, B.~Ermentrout, A.~Friedman, and D.~Terman, \emph{{Tutorials in
  mathematical biosciences. {I}}}, Lecture {Notes} in {Mathematics}, vol. 1860,
  Springer-Verlag, Berlin, 2005, Mathematical neuroscience, Mathematical
  Biosciences Subseries.

\bibitem{calvetti_2000}
D.~Calvetti and L.~Reichel, \emph{Iterative methods for large continuation
  problems}, J. Comput. Appl. Math. \textbf{123} (2000), 217--240.

\bibitem{CantrellCosner:03}
R.S. Cantrell and C.~Cosner, \emph{{Spatial ecology via reaction-diffusion
  equations}}, Wiley {Series} in {Mathematical} and {Computational} {Biology},
  John Wiley \& Sons, Ltd., Chichester, 2003.

\bibitem{ChaillouSuri:06}
A.L. Chaillou and M.~Suri, \emph{{Computable error estimators for the
  approximation of nonlinear problems by linearized models}}, Comput. Methods
  Appl. Mech. Engrg. \textbf{196} (2006), no.~1-3, 210--224.

\bibitem{ChaillouSuri:07}
\bysame, \emph{{A posteriori estimation of the linearization error for strongly
  monotone nonlinear operators}}, J. Comput. Appl. Math. \textbf{205} (2007),
  no.~1, 72--87.

\bibitem{CHHPSbratu}
K.A. Cliffe, E.~Hall, P.~Houston, E.T. Phipps, and A.G. Salinger,
  \emph{Adaptivity and a posteriori error control for bifurcation problems {I}:
  {T}he {B}ratu problem}, Commun. Comput. Phys. \textbf{8} (2010), 845--865.

\bibitem{Congreve_Houston_2014}
S.~Congreve and P.~Houston, \emph{Two-grid $hp$-version discontinuous
  {G}alerkin finite element methods for quasi-{N}ewtonian fluid flows}, Int. J.
  Numer. Anal. Model. \textbf{11} (2014), no.~3, 496--524.

\bibitem{CongreveHoustonSuliWihler:13}
S.~Congreve, P.~Houston, E.~S{\"u}li, and T.P. Wihler, \emph{{Discontinuous
  {Galerkin} finite element approximation of quasilinear elliptic boundary
  value problems {II}: strongly monotone quasi-{Newtonian} flows}}, IMA J.
  Numer. Anal. \textbf{33} (2013), no.~4, 1386--1415.

\bibitem{CongreveHoustonWihler:13}
S.~Congreve, P.~Houston, and T.P. Wihler, \emph{{Two-grid $hp$-version
  discontinuous {Galerkin} finite element methods for second-order quasilinear
  elliptic {PDEs}}}, J. Sci. Comput. \textbf{55} (2013), no.~2, 471--497.

\bibitem{CongreveWihler:15}
S.~Congreve and T.P. Wihler, \emph{An iterative finite element method for
  strongly monotone quasi-linear diffusion-reaction problems}, in preparation,
  2015.

\bibitem{De07}
L.~Demkowicz, \emph{Computing with {$hp$}-adaptive finite elements. {V}ol. 1},
  Chapman \& Hall/CRC Applied Mathematics and Nonlinear Science Series, Chapman
  \& Hall/CRC, Boca Raton, FL, 2007, One and two dimensional elliptic and
  Maxwell problems.

\bibitem{Deuflhard:04}
P.~Deuflhard, \emph{{Newton methods for nonlinear problems}}, Springer {Series}
  in {Computational} {Mathematics}, vol.~35, Springer-Verlag, Berlin, 2004,
  Affine invariance and adaptive algorithms.

\bibitem{Edelstein-Keshet:05}
L.~Edelstein-Keshet, \emph{{Mathematical models in biology}}, Classics in
  {Applied} {Mathematics}, vol.~46, Society for Industrial and Applied
  Mathematics (SIAM), Philadelphia, PA, 2005, Reprint of the 1988 original.

\bibitem{EibnerMelenk:07}
T.~Eibner and J.~M. Melenk, \emph{{An adaptive strategy for $hp$-{FEM} based on
  testing for analyticity}}, Comput. Mech. \textbf{39} (2007), no.~5, 575--595.

\bibitem{El-AlaouiErnVohralik:11}
L.~El~Alaoui, A.~Ern, and M.~Vohral{\'\i}k, \emph{{Guaranteed and robust a
  posteriori error estimates and balancing discretization and linearization
  errors for monotone nonlinear problems}}, Comput. Methods Appl. Mech. Engrg.
  \textbf{200} (2011), no.~37-40, 2782--2795.

\bibitem{FankhauserWihlerWirz:14}
T.~Fankhauser, T.P. Wihler, and M.~Wirz, \emph{{The $hp$-adaptive {FEM} based
  on continuous {Sobolev} embeddings: isotropic refinements}}, Comp. Math.
  Appl. \textbf{67} (2014), no.~4, 854--868.

\bibitem{GarauMorinZuppa:11}
E.M. Garau, P.~Morin, and C.~Zuppa, \emph{{Convergence of an adaptive {Ka{\v
  c}anov} {FEM} for quasi-linear problems}}, Appl. Numer. Math. \textbf{61}
  (2011), no.~4, 512--529.

\bibitem{GibbonJamesMoroz:79}
J.D. Gibbon, I.N. James, and I.M. Moroz, \emph{The {S}ine-{G}ordon equation as
  a model for a rapidly rotating baroclinic fluid}, Phys. Script. \textbf{20}
  (1979), 402--408.

\bibitem{Han:94}
W.~Han, \emph{{A posteriori error analysis for linearization of nonlinear
  elliptic problems and their discretizations}}, Math. Method Appl. Sci.
  \textbf{17} (1994), no.~7, 487--508.

\bibitem{HoustonSchotzauWihler:06}
P.~Houston, D.~Sch{\"o}tzau, and T.P. Wihler, \emph{{An $hp$-adaptive mixed
  discontinuous {Galerkin} {FEM} for nearly incompressible linear elasticity}},
  Comput. Methods Appl. Mech. Engrg. \textbf{195} (2006), no.~25-28,
  3224--3246.

\bibitem{HoustonSchotzauWihler:07}
\bysame, \emph{{Energy norm a posteriori error estimation of $hp$-adaptive
  discontinuous {Galerkin} methods for elliptic problems}}, Math. Models
  Methods Appl. Sci. \textbf{17} (2007), no.~1, 33--62.

\bibitem{HoustonSuli02}
P.~Houston and E.~S{\"u}li, \emph{Adaptive finite element approximation of
  hyperbolic problems}, Error Estimation and Adaptive Discretization Methods in
  Computational Fluid Dynamics. Lect.~Notes Comput.~Sci.~Engrg. (T.~Barth and
  H.~Deconinck, eds.), vol.~25, Springer, 2002, pp.~269--344.

\bibitem{HoustonSuli:05}
\bysame, \emph{{A note on the design of $hp$-adaptive finite element methods
  for elliptic partial differential equations}}, Comput. Methods Appl. Mech.
  Engrg. \textbf{194} (2005), no.~2-5, 229--243.

\bibitem{HoustonSuliWihler:08}
P.~Houston, E.~S{\"u}li, and T.P. Wihler, \emph{A posteriori error analysis of
  {$hp$}-version discontinuous {G}alerkin finite-element methods for
  second-order quasi-linear elliptic {PDE}s}, IMA J. Numer. Anal. \textbf{28}
  (2008), no.~2, 245--273.

\bibitem{Houston2016977}
P.~Houston and T.P. Wihler, \emph{Adaptive energy minimisation for $hp$-finite
  element methods}, Comput. Math. Appl. \textbf{71} (2016), no.~4, 977 -- 990.

\bibitem{KarakashianPascal03}
O.A. Karakashian and F.~Pascal, \emph{A posteriori error estimation for a
  discontinuous {G}alerkin approximation of second order elliptic problems},
  SIAM J. Numer. Anal. \textbf{41} (2003), 2374--2399.

\bibitem{KarkulikMelenk:15}
M.~Karkulik and J.M. Melenk, \emph{{Local high-order regularization and
  applications to $hp$-methods}}, Comp. Math. Appl. \textbf{70} (2015), no.~7,
  1606--1639.

\bibitem{MitchellMcClain:14}
W.F. Mitchell and M.A. McClain, \emph{{A comparison of $hp$-adaptive strategies
  for elliptic partial differential equations}}, ACM. Transactions on
  Mathematical Software \textbf{41} (2014), no.~1, 2:1--39.

\bibitem{mohsen_2014}
A.~Mohsen, \emph{A simple solution of the {B}ratu problem}, Comput. Math. Appl.
  \textbf{67} (2014), 26--33.

\bibitem{mohsen08}
A.~Mohsen, L.F. Sedeek, and S.A. Mohamed, \emph{New smoother to enhance
  multigrid--based methods for {B}ratu problem}, Appl. Math. Comput.
  \textbf{204} (2008), 325--339.

\bibitem{Ni:11}
W.-M. Ni, \emph{{The mathematics of diffusion}}, {CBMS}-{NSF} {Regional}
  {Conference} {Series} in {Applied} {Mathematics}, vol.~82, Society for
  Industrial and Applied Mathematics (SIAM), Philadelphia, PA, 2011.

\bibitem{OkuboLevin:01}
A.~Okubo and S.A. Levin, \emph{{Diffusion and ecological problems: modern
  perspectives}}, second ed., Interdisciplinary {Applied} {Mathematics},
  vol.~14, Springer-Verlag, New York, 2001.

\bibitem{SchneebeliWihler:11}
H.R. Schneebeli and T.P. Wihler, \emph{{The {Newton}-{Raphson} method and
  adaptive {ODE} solvers}}, Fractals. Complex Geometry, Patterns, and Scaling
  in Nature and Society \textbf{19} (2011), no.~1, 87--99.

\bibitem{SchotzauSchwabToselli:02}
D.~Sch{\"o}tzau, C.~Schwab, and A.~Toselli, \emph{{Mixed $hp$-{DGFEM} for
  incompressible flows}}, SIAM J. Numer. Anal. \textbf{40} (2002), no.~6,
  2171--2194 (electronic) (2003).

\bibitem{opac-b1101124}
P.~Solin, K.~Segeth, and I.~Dolezel, \emph{Higher-order finite element
  methods}, Studies in advanced mathematics, Chapman \& Hall/CRC, Boca Raton,
  London, 2004.

\bibitem{StammWihler:10}
B.~Stamm and T.P. Wihler, \emph{{$hp$-{O}ptimal discontinuous {Galerkin}
  methods for linear elliptic problems}}, Math. Comp. \textbf{79} (2010),
  no.~272, 2117--2133.

\bibitem{Strauss:77}
W.A. Strauss, \emph{Existence of solitary waves in higher dimensions}, Comm.
  Math. Phys. \textbf{55} (1977), no.~2, 149--162.

\bibitem{Verfurth:98}
R.~Verf{\"u}rth, \emph{{Robust a posteriori error estimators for a singularly
  perturbed reaction-diffusion equation}}, Numer. Math. \textbf{78} (1998),
  no.~3, 479--493.

\bibitem{WihlerFrauenfelderSchwab:03}
T.P. Wihler, P.~Frauenfelder, and C.~Schwab, \emph{{Exponential convergence of
  the $hp$-{DGFEM} for diffusion problems}}, Comp. Math. Appl. \textbf{46}
  (2003), no.~1, 183--205.

\bibitem{Zhu_3D}
L.~Zhu, S.~Giani, P.~Houston, and D.~Sch\"{o}tzau, \emph{Energy norm a
  posteriori error estimation for $hp$-adaptive discontinuous {Galerkin}
  methods for elliptic problems in three dimensions}, Math. Models Methods
  Appl. Sci. \textbf{21} (2011), no.~2, 267--306.

\bibitem{SchotzauZhu:11}
L.~Zhu and D.~Sch\"otzau, \emph{A robust \emph{a posteriori} error estimate for
  $hp$-adaptive {DG} methods for convection-diffusion equations}, IMA J. Numer.
  Anal. \textbf{31} (2011), 971--1005.

\end{thebibliography}

\end{document}